\documentclass[11pt,a4paper]{article}
\usepackage[utf8]{inputenc}
\usepackage{amsmath}
\usepackage{amssymb}
\usepackage{amsthm}
\usepackage[english]{babel}
\usepackage{booktabs}
\usepackage{graphicx}
\usepackage[colorlinks,allcolors=blue]{hyperref}
\usepackage{pgfplots}
\pgfplotsset{compat=1.14}
\usepackage{xcolor}
\usepackage[hmargin={30mm,30mm},vmargin={30mm,35mm}]{geometry}
\usepackage{caption}
\usepackage{subcaption}
\usepackage{algorithm}
\usepackage{algpseudocode} % To write the pseudocode
\usepackage{nicefrac} % To write in a better way the exponential fractions
\usepackage[affil-it]{authblk}

\usepackage{newtxtext}
\usepackage{newtxmath}

\DeclareRobustCommand{\VEC}[1]{\boldsymbol{#1}}
\DeclareRobustCommand{\MAT}[1]{\boldsymbol{\mathbb{#1}}}
\numberwithin{equation}{section}    %The number of the equations change according to the section

\usepackage{xpatch} % Name of theorems are bold
\makeatletter
   \xpatchcmd{\@thm}{\fontseries\mddefault\upshape}{}{}{} % same font as thm-header
\makeatother

\newtheorem{theorem}{Theorem}

\newtheorem{lemma}[theorem]{Lemma}

\theoremstyle{remark}
\newtheorem{remark}[theorem]{Remark}
\theoremstyle{definition}
\theoremstyle{definition}
\newtheorem{definition}[theorem]{Definition}

\newtheorem*{Unilateral*}{Unilateral contact problem with no friction}

\theoremstyle{plain}

\newtheorem{construction}[theorem]{Construction}

\newtheorem{assumption}[theorem]{Assumption}

\DeclareMathOperator{\tr}{tr}

\newcommand{\vvvert}{\vert\kern-0.25ex\vert\kern-0.25ex\vert}
\newcommand{\norm}[1]{\vvvert #1 \vvvert}
\newcommand{\normHuno}[1]{\left\lVert #1\right\rVert_{1,\Omega}}
\newcommand{\normGamma}[1]{\left\lvert #1\right\rvert_{C,h}}

\newcommand{\normT}[1]{\left\lvert #1\right\rvert_{C,T}}

\newcommand{\energynorm}[1]{\left\lVert #1\right\rVert_{\rm en}}
\newcommand{\localdualnormresidual}[1]{{\color{black}\norm{\mathcal{R}_{\mathcal{T}_T}(#1)}_{*,\tilde{\omega}_T}}}
\newcommand{\dualnormresidual}[1]{{\color{black}\norm{\mathcal{R}(#1)}_*}}
\newcommand{\interpol}{\mathcal{I}_h}

\DeclareCaptionLabelFormat{andtable}{#1~#2  \&  \tablename~\thetable}

\tikzstyle{arrow} = [thick,->,>=stealth]

\definecolor{cadmiumgreen}{rgb}{0.0, 0.42, 0.24}

\definecolor{colorH1ad}{rgb}{0.803921568627451,0.36078431372549,0.36078431372549}
\definecolor{colorEnergyun}{rgb}{0.254901960784314,0.411764705882353,0.882352941176471}
\definecolor{colorEnergyad}{rgb}{1,0.647058823529412,0}
\definecolor{colorTEun}{rgb}{0.0980392156862745,0.0980392156862745,0.43921568627451}
\definecolor{colorUpad}{rgb}{0.647058823529412,0.164705882352941,0.164705882352941}
\definecolor{colorEffEnergy}{rgb}{0.729411764705882,0.333333333333333,0.827450980392157}
\definecolor{colorEffUp}{rgb}{0.580392156862745,0,0.827450980392157}

\definecolor{colorEstTot}{rgb}{0.12156862745098,0.466666666666667,0.705882352941177}
\definecolor{colorEstFlux}{rgb}{1,0.498039215686275,0.0549019607843137}
\definecolor{colorEstLin}{rgb}{0.172549019607843,0.627450980392157,0.172549019607843}
\definecolor{colorEstReg}{rgb}{0.83921568627451,0.152941176470588,0.156862745098039}
\definecolor{colorEstDisc}{rgb}{0.580392156862745,0.403921568627451,0.741176470588235}

\title{A posteriori error estimates via equilibrated stress reconstructions for contact problems approximated by Nitsche's method}

\author[1]{Daniele A. Di Pietro}
\author[,1,2]{Ilaria Fontana \thanks{Corresponding author}}
\author[2]{Kyrylo Kazymyrenko}

\affil[1]{IMAG, Univ Montpellier, CNRS, Montpellier, France}
\affil[2]{IMSIA, UMR EDF-CNRS-CEA-ENSTA 9219, Palaiseau, France}

\date{}

\begin{document}

\maketitle

\begin{abstract}
  We present an a posteriori error estimate based on equilibrated stress reconstructions for the finite element approximation of a unilateral contact problem with weak enforcement of the contact conditions.
  We start by proving a guaranteed upper bound for the dual norm of the residual.
  This norm is shown to control the natural energy norm up to a boundary term, which can be removed under a saturation assumption.
  The basic estimate is then refined to distinguish the different components of the error, and is used as a starting point to design an algorithm including adaptive stopping criteria for the nonlinear solver and automatic tuning of a regularization parameter.
  We then discuss an actual way of computing the stress reconstruction based on the Arnold--Falk--Winther finite elements.
  Finally, after briefly discussing the efficiency of our estimators, we showcase their performance on a panel of numerical tests.
\end{abstract}


\section{Introduction}

In various fields of solid mechanics and engineering it is essential to describe contact and friction between two bodies.
This is the case, e.g., when modelling foundations and joints in buildings or when considering impact problems. 
In this paper, we focus on a simplified unilateral contact problem without friction, for which contact is mathematically expressed by some inequalities and complementarity conditions.
These constraints translate non-penetration as well as the absence of cohesive forces and friction between the two bodies.
In order to deal with the above constraints from a numerical point of view, different strategies have been proposed in the literature, including penalized formulations, mixed formulations, and weak enforcement à la Nitsche.
We focus on the latter, which does not require the introduction of Lagrange multipliers and results in a coercive formulation that is easy to implement.
The literature on the numerical approximation of contact problems is vast, and a comprehensive state-of-the-art lies outside of the scope of the present papers.
We refer to the review articles \cite{Wohlmuth2011, Chouly-Mlika2017} and references therein for a broader discussion.

Nitsche's method was originally introduced in \cite{Nitsche1971} to weakly enforce Dirichlet boundary conditions. 
Its application to the unilateral contact problem considered in this paper was originally proposed in \cite{Chouly2013}, where the well-posedness and convergence of a conforming finite element scheme are studied.
Further extensions to problems involving friction or multiple bodies can be found in \cite{Chouly2015,Chouly2014,Renard2013,Gustafsson2020};
we also mention \cite{Chouly2020} concerning the adaptation of these ideas to Hybrid High-Order discretizations \cite{Di-Pietro.Ern.ea:14,Di-Pietro.Ern:15,DiPietro2019}.
A residual-based a posteriori error analysis can be found in \cite{Chouly-Fabre2017} based on a saturation assumption.

In this paper we follow a different path and carry out an a posteriori error analysis based on equilibrated tractions in the spirit of the Prager--Synge equality \cite{Prager.Synge:47} (see also \cite{Ern2015} and \cite[Chapter 7]{Vohralik2015}).
This approach, which does not require the saturation assumption when the dual norm of the residual is considered as an error measure, has also the advantage of avoiding unknown constants in the upper bound.
The corresponding error estimate can additionally be refined in order to distinguish the various components of the error (discretization, linearization, regularization).
This decomposition is leveraged here to design a fully adaptive resolution algorithm including an a posteriori-based stopping criterion for the nonlinear solver and the automatic tuning of the regularization parameter.

A crucial ingredient of our a posteriori analysis is a novel $\MAT{H}(\textbf{div})$-conforming stress reconstruction obtained from the numerical solution by solving small problems on patches around mesh vertices.
In the spirit of \cite{Riedlbeck-DiPietro2017, Botti2018}, we use the Arnold--Falk--Winther mixed finite element spaces with weakly enforced symmetry \cite{Arnold2007}; strong symmetry could be enforced using the Arnold--Winther finite element spaces \cite{Arnold2002} as in \cite{Riedlbeck-DiPietro2017}, but would come at a significantly higher computational cost.
The stress reconstruction is built so that its divergence and its normal component on the contact and Neumann portions of the domain boundary are locally in equilibrium with the volume and surface source terms, respectively (such equilibrium properties are not satisfied by stress fields resulting from $\VEC{H}^1$-conforming finite element approximations).
In order to distinguish the various error components, the stress reconstruction is additionally split so as to identify the contributions to the error resulting from linearization and regularization.

The main results of the paper are briefly summarized in what follows.
The basic error estimate of Theorem \ref{theorem - a posteriori error estimation for $u_h$} establishes a guaranteed upper bound for the dual norm of the residual.
Such norm is shown in Theorem \ref{theorem - upper bound of the energy norm} to control the energy norm of the error up to a boundary term on the contact region (this term can be eliminated when a saturation assumption similar to the one used in \cite{Chouly-Fabre2017} holds).
A refined error estimate distinguishing the error components is derived in Theorem \ref{theorem - a posteriori estimation u_h^k}.

The rest of the paper is organized as follows.
In Section \ref{sec:Setting} we describe the unilateral contact problem and its finite element approximation à la Nitsche.
In Section \ref{sec:A posteriori analysis} we derive a basic estimate for the dual norm of the residual and compare this norm with the energy norm.   
Section \ref{sec:identification.error.components} contains a refined version of the estimate distinguishing the error components which serves as a starting point for the development of a fully adaptive algorithm.
An equilibrated stress reconstruction based on the Arnold--Falk--Winther finite element is then proposed in Section \ref{sec - reconstruction}.
Section \ref{sec:efficiency} briefly addresses the efficiency of the error estimators.
Finally, some numerical results performed with the open source software FreeFem++ are presented in Section \ref{sec:Numerical results}.

\section{Setting}\label{sec:Setting}

In this section we discuss the contact problem and its finite element discretization with weakly enforced contact conditions.

\subsection{Unilateral contact problem}\label{sec:setting:pde}

\subsubsection{Strong formulation}

Let $d\in\{2,3\}$ and let $\Omega\subset\mathbb{R}^d$ be a connected open subset of $\mathbb{R}^d$ representing an elastic body. We suppose that  $\Omega$ is a polygon (if $d=2$) or a polyhedron (if $d=3$), and that its boundary $\partial\Omega$ is partitioned into three non-overlapping parts $\Gamma_D$, $\Gamma_N$, and $\Gamma_C$ such that $\lvert\Gamma_D\rvert>0$ and $\lvert\Gamma_C\rvert>0$ ($\lvert\,\cdot\,\rvert$ denotes here the Hausdorff measure).
In its reference configuration, the elastic body is in contact through $\Gamma_C$ with a rigid foundation, and we assume that the (unknown) contact zone in the deformed configuration will be included in $\Gamma_C$. The body is clamped at $\Gamma_D$ and it is subjected to a volume force $\VEC{f}\in \VEC{L}^2(\Omega)$ and to a surface load $\VEC{g}_N\in \VEC{L}^2(\Gamma_N)$ on $\Gamma_N$; see Figure \ref{fig:domain illustration} for an example.

Denoting by $\VEC{n}$ the unit normal vector on $\partial \Omega$ pointing out of $\Omega$, for any displacement field $ \VEC{v}\colon \Omega\to\mathbb{R}^d$ and for any density of surface forces $\VEC{\sigma}(\VEC{v})\VEC{n}$ defined on $\partial\Omega$, we have the following (unique) decomposition into normal and tangential components:
\begin{equation}\label{eq:decomposition.normal.tangential}
    \VEC{v}=v^n\VEC{n}+\VEC{v}^{\VEC{t}} \qquad\qquad \text{and} \qquad\qquad \VEC{\sigma}(\VEC{v})\VEC{n}=\sigma^n(\VEC{v}) \VEC{n} + \VEC{\sigma}^{\VEC{t}}(\VEC{v}).
\end{equation}

We consider the following problem:
Find the displacement field $\VEC{u}\colon\Omega\to \mathbb{R}^d$ such that
\begin{subequations}\label{StrongFormulation}
    \begin{alignat}{2}
      %\VEC{\nabla}\cdot 
      \VEC{\rm div}\, 
      \VEC{\sigma}(\VEC{u})+\VEC{f}&=\VEC{0}
      & \qquad & \text{in $\Omega$}, \label{eq:unilateral equilibrium}
      \\
      \VEC{\sigma}(\VEC{u}) &= \VEC{A} \VEC{\varepsilon} (\VEC{u})
      &\qquad & \text{in $\Omega$},
      \\
      \VEC{u} &= \VEC{0}
      & \qquad & \text{on $\Gamma_D$},
      \\
      \VEC{\sigma}(\VEC{u}) \VEC{n} &= \VEC{g}_N
      & \qquad & \text{on $\Gamma_N$}, \label{eq:unilateral Neumann}
      \\
      u^n\leq 0,\ \sigma^n(\VEC{u})\leq 0,\ \sigma^n(\VEC{u}) u^n &= 0
      & \qquad & \text{on $\Gamma_C$}, \label{eq:unilateral contact 1}
      \\
      \VEC{\sigma}^{\VEC{t}}(\VEC{u}) &= \VEC{0}
      & \qquad & \text{on $\Gamma_C$}, \label{eq:unilateral contact 2}
    \end{alignat}
\end{subequations}
where $\VEC{\varepsilon}(\VEC{v})\coloneqq\frac12(\VEC{\nabla v}+\VEC{\nabla v}^\top)$ is the strain tensor field, $\VEC{\sigma}(\VEC{v}) \in \mathbb{R}^{d\times d}_{\text{sym}}$ is the Cauchy stress tensor, $\VEC{\rm div}$ is the divergence operator acting row-wise on tensor valued functions, and $\VEC{A}$ is the fourth order symmetric elasticity tensor such that, for all second-order tensor $\VEC{\tau}$, $\VEC{A}\VEC{\tau} = \lambda \tr(\VEC{\tau}) \VEC{I}_d + 2\mu \VEC{\tau}$, with $\lambda$ and $\mu$ denoting the Lam\'e parameters.

\begin{remark}[Contact conditions]
  The first contact condition \eqref{eq:unilateral contact 1} is a complementarity condition:
  if, at a point $\VEC{x}\in\Gamma_C$, there is no contact in the deformed configuration (i.e., $u^n < 0$), then the normal stress vanishes at that point (i.e., $\sigma^n(\VEC{u})=0$);
  on the other hand, if at $\VEC{x}\in\Gamma_C$ the normal stress is nonzero (i.e., $\sigma^n(\VEC{u})<0$), then in the deformed configuration we still have contact in $\VEC{x}$ (i.e., $u^n=0$).
  These conditions %29-06%express the fact that there are no 
  also account for the absence of
  normal cohesive forces between the elastic body and the rigid foundation.
  The second contact condition \eqref{eq:unilateral contact 2} simply establishes the absence of friction on $\Gamma_C$.
\end{remark}

The incorporation of standard friction models to the following a posteriori theory seems possible, but lies outside of the scope of the present paper.
This subject will be addressed in future works.

\subsubsection{Weak formulation}

Let $D$ denote a measurable set of $\mathbb{R}^d$.
 In what follows, $D$ will be typically either equal to $\Omega$ or to the union of a finite subset of mesh elements.
We denote by $H^s(D)$, $s\in\mathbb{R}$, the usual Sobolev space of index $s$ on $D$, with the convention that $H^0(D)\coloneqq L^2(D)$, the space of square-integrable functions on $D$.
Its vector and tensor versions are denoted respectively by $\VEC{H}^s(D) \coloneqq [H^s(D)]^d$ and $\MAT{H}^s(D) \coloneqq [H^s(D)]^{d\times d}$.
We let, similarly, $\VEC{L}^2(D)\coloneq [L^2(D)]^d$ and $\MAT{L}^2(D)\coloneq [L^2(D)]^{d\times d}$.
Moreover, $\lVert\,\cdot\,\rVert_{s,D}$ denotes the norm of $H^s(D)$ or $\VEC{H}^s(D)$ according to its argument.
The first subscript is omitted when $s=0$, i.e., $\lVert\,\cdot\,\rVert_D$ is the standard norm of $L^2(D)$, $\VEC{L}^2(D)$, or $\MAT{L}^2(D)$ according to its argument.
The usual inner products of these spaces are denoted by $(\,\cdot\,,\,\cdot\,)_{D}$, with the convention that the subscript is omitted when $D=\Omega$.
In what follows, we will also need the space $\MAT{H}(\textbf{div},D)$ spanned by functions of $\MAT{L}^2(D)$ with weak (row-wise) divergence in $\VEC{L}^2(D)$.

Denote by $\VEC{H}_D^1(\Omega)$ the subspace of $\VEC{H}^1(\Omega)$ incorporating the Dirichlet boundary condition on $\Gamma_D$, and by $\VEC{K}$ its subset spanned by admissible displacements, i.e.,
\begin{equation*}
    \VEC{H}_D^1(\Omega) \coloneqq \left\{\VEC{v}\in \VEC{H}^1(\Omega)\ :\ \VEC{v}=\VEC{0} \ \text{on}\ \Gamma_D\right\},
    \qquad
    \VEC{K} \coloneqq \left\{\VEC{v}\in \VEC{H}^1_D(\Omega)\ :\ v^n\leq 0 \ \text{on}\ \Gamma_C\right\}.
\end{equation*}
The weak formulation of problem \eqref{StrongFormulation} corresponds to the following variational inequality (see, e.g., \cite{Haslinger1996}):
Find $\VEC{u}\in\VEC{K}$ such that
\begin{equation}\label{eq:weak}
    a(\VEC{u},\VEC{v}-\VEC{u})\geq L(\VEC{v}-\VEC{u}) \qquad \forall\VEC{v}\in\VEC{K}, 
\end{equation}
where the bilinear form $a\colon \VEC{H}^1(\Omega)\times\VEC{H}^1(\Omega) \to\mathbb{R}$ and the linear form $L\colon \VEC{H}^1(\Omega) \to\mathbb{R}$ are defined as follows:
For all $(\VEC{u},\VEC{v})\in\VEC{H}^1(\Omega)\times\VEC{H}^1(\Omega)$,
\begin{equation}\label{eq:definition a and L}
    a(\VEC{u},\VEC{v}) \coloneqq (\VEC{\sigma}(\VEC{u}),\VEC{\varepsilon}(\VEC{v})),
    \qquad
    L(\VEC{v})\coloneqq(\VEC{f},\VEC{v}) + (\VEC{g}_N,\VEC{v})_{\Gamma_N}.
\end{equation}
Problem \eqref{eq:weak} admits a unique solution by the Stampacchia theorem.

\subsection{Discretization}\label{sec:Nitsche-based method}

Let $\{\mathcal{T}_h\}_h$ be a family of conforming triangulations of $\Omega$, indexed by the mesh size $h\coloneqq\max_{T\in\mathcal{T}_h}h_T$, where $h_T$ is the diameter of the element $T$.
This family is assumed to be regular in the classical sense; see, e.g., \cite[Eq. (3.1.43)]{Ciarlet:02}.
Furthermore, each triangulation is conformal to the subdivision of the boundary into $\Gamma_D$, $\Gamma_N$, and $\Gamma_C$ in the sense that the interior of a boundary edge (if $d=2$) or face (if $d=3$) cannot have non-empty intersection with more than one part of the subdivision.
Mesh-related notations that will be used in the a posteriori error analysis are collected in Table \ref{tab - mesh notation}.
For the sake of simplicity, from this point on we adopt the three-dimensional terminology and speak of faces instead of edges also in dimension $d=2$.

\begin{table}[ht]
  \centering
  \begin{tabular}{p{0.25\textwidth}p{0.70\textwidth}}
    \toprule
    \multicolumn{1}{c}{\textbf{Notation}} & \multicolumn{1}{c}{\textbf{Definition}} \\
    \midrule
    $\mathcal{F}_h$ & Set of faces of $\mathcal{T}_h$ \\
    $\mathcal{F}_h^b$ & Set of boundary faces, i.e.,  $\{F\in\mathcal{F}_h\ :\ F\subset \partial\Omega\}$ \\
    $\mathcal{F}_h^{D}\cup\mathcal{F}_h^N\cup\mathcal{F}_h^C$ & Partition of $\mathcal{F}_h^b$ induced by the boundary and contact conditions \\
    $\mathcal{F}_h^i$ & Set of interior faces, i.e., $\mathcal{F}_h\setminus \mathcal{F}_h^b$ \\
    $\mathcal{F}_T$ & Set of faces of the element $T\in\mathcal{T}_h$, i.e., $\{F\in\mathcal{F}_h\ :\ F\subset\partial T\}$ \\
    $\mathcal{F}_T^\bullet$, $\bullet\in\{b, D, N, C\}$ & $\mathcal{F}_T\cap\mathcal{F}_h^\bullet$, $T\in\mathcal{T}_h$ \\
    $\mathcal{V}_h$ & Set of all the vertices of $\mathcal{T}_h$ \\
    $\mathcal{V}_h^b$ & Set of boundary vertices, i.e., $\{\VEC{a}\in\mathcal{V}_h\ :\ \VEC{a}\in\partial\Omega\}$ \\
    $\mathcal{V}_h^i$ & Set of interior vertices, i.e., $\mathcal{V}_h\setminus\mathcal{V}_h^b$ \\
    $\mathcal{V}_T$ & Set of vertices of the element $T\in\mathcal{T}_h$, i.e., $\{\VEC{a}\in\mathcal{V}_h\ :\ \VEC{a}\in\partial T\}$ \\
    $\mathcal{V}_F$ & Set of vertices of the mesh face $F\in\mathcal{F}_h$, i.e., $\{\VEC{a}\in\mathcal{V}_h\ :\ \VEC{a}\in\partial F\}$ \\
    $\omega_{\VEC{a}}$ & Union of the elements sharing the vertex $\VEC{a}\in\mathcal{V}_h$, i.e., $\displaystyle\bigcup_{T\in\mathcal{T}_h,\,\VEC{a}\in \partial T} T$ \\
    \bottomrule
  \end{tabular}
  \caption{Mesh-related notations.}
  \label{tab - mesh notation}
\end{table}

For any $X\in\mathcal{T}_h\cup\mathcal{F}_h$ mesh element or face, $\mathcal{P}^n(X)$ denotes the restriction to $X$ of $d$-variate polynomials of total degree $\le n$, and we set $\VEC{\mathcal{P}}^n(X)\coloneqq[\mathcal{P}^n(X)]^d$ and $\MAT{P}^n(X) \coloneqq [\mathcal{P}^n(X)]^{d\times d}$.
We seek the displacement in the standard Lagrange finite element space of degree $p\ge 1$ with strongly enforced boundary condition on $\Gamma_D$:
\[
\VEC{V}_h \coloneqq \left\{ \VEC{v}_h\in \VEC{H}_D^1(\Omega)\ :\ \VEC{v}_h|_T\in \VEC{\mathcal{P}}^p(T) \ \text{for any}\ T\in\mathcal{T}_h \right\}.
\]
Denote by $[\, \cdot\, ]_{\mathbb{R}^-}\colon\mathbb{R}\to\mathbb{R}^-$ the projection on the half-line of negative real numbers $\mathbb{R}^-$, i.e., $[x]_{\mathbb{R}^-} \coloneqq \frac12(x-\lvert x\rvert)$ for all $x\in\mathbb{R}$. For every real number $\theta$ and every positive bounded function $\gamma\colon \Gamma_C\to\mathbb{R}^+$, we define the following linear operator \cite{Chouly-Mlika2017}:
\begin{equation}\label{eq:definition P_theta,gamma}
    \begin{split}
      P_{\theta,\gamma}^n\colon \VEC{W} &\to L^2(\Gamma_C) \\
        \VEC{v} & \mapsto \theta\sigma^n(\VEC{v})-\gamma v^n,
    \end{split}
\end{equation}
where $\VEC{W}\coloneqq\left\{\VEC{v}\in\VEC{H}^1(\Omega)\ :\ \VEC{\sigma}(\VEC{v})\VEC{n}|_{\Gamma_C}\in\VEC{L}^2(\Gamma_C)\right\}$ (notice that $\VEC{V}_h \subset \VEC{W}$).
Assuming that $\VEC{u}\in\VEC{W}$, the first contact condition \eqref{eq:unilateral contact 1} can be written as (see \cite{Curnier1988, Chouly2013}):
\begin{equation}\label{eq:normal contact condition rewritten}
    \sigma^n(\VEC{u}) = \left[\sigma^n(\VEC{u})-\gamma u^n\right]_{\mathbb{R}^-} = \left[P_{1,\gamma}^n(\VEC{u})\right]_{\mathbb{R}^-}.
\end{equation}
\begin{remark}[Case $\theta = 0$]\label{remark Pthetagamma}
The linear operator $P^n_{\theta,\gamma}$ is well defined on $\VEC{V}_h$ since it is a subspace of the space of broken polynomials.
It can be easily extended to the space $\VEC{H}^1_D(\Omega)$ in the case $\theta=0$, for which $P^n_{0,\gamma}(\VEC{v})=-\gamma v^n$, as $\VEC{v}\in \VEC{H}^1(\Omega)$ guarantees $\VEC{v}|_{\Gamma_C} \in \VEC{L}^2(\Gamma_C)$ by the trace theorem.
\end{remark}
From now on, $\gamma_0 >0$ will denote a fixed constant called \emph{Nitsche parameter}, and we suppose that $\gamma$ is the positive piecewise constant function on $\Gamma_C$ which satisfies:
For all $T\in\mathcal{T}_h$ such that $\lvert \partial T\cap\Gamma_C\rvert >0$,
\[
  \gamma|_{\partial T\cap \Gamma_C} = \frac{\gamma_0}{h_T}.
\]
We consider the following method à la Nitsche to approximate problem \eqref{StrongFormulation}, originally introduced in \cite{Chouly2013}:
Find $\VEC{u}_h\in\VEC{V}_h$ such that
\begin{equation}\label{Nitsche-based_method theta=0}%!%\label{NitscheMethod}
  a(\VEC{u}_h,\VEC{v}_h)
  - \left(\left[P_{1,\gamma}^n(\VEC{u}_h)\right]_{\mathbb{R}^-}, v_h^n \right)_{\Gamma_C} 
  = L(\VEC{v}_h) \qquad \forall \VEC{v}_h\in\VEC{V}_h.
\end{equation}
For the a priori analysis of the method, we refer to \cite{Chouly2013}.

%------------------------------------------------------------------------------%

\section{Basic a posteriori error estimate}\label{sec:A posteriori analysis}

In this section we derive a basic a posteriori error estimate based on the notion of equilibrated stress reconstruction.

\subsection{Error measure}

In the framework of a posteriori error estimation, the dual norm of a residual functional can be used as a measure of the error between the exact solution $\VEC{u}$ of the problem and the solution $\VEC{u}_h$ obtained with the finite element method.
Denoting by $(\VEC{H}^1_D(\Omega))^*$ the dual space of $\VEC{H}^1_D(\Omega)$, for any $\VEC{w}_h\in\VEC{V}_h$ the \emph{residual} $\mathcal{R}(\VEC{w}_h)\in (\VEC{H}^1_D(\Omega))^*$ is defined by
\begin{equation}\label{residual definition2}
    \begin{split}
        \left\langle \mathcal{R}(\VEC{w}_h),\VEC{v}\right\rangle &\coloneqq 
        L(\VEC{v})-a(\VEC{w}_h,\VEC{v})+\left(\left[P_{1,\gamma}^n(\VEC{w}_h)\right]_{\mathbb{R}^-},v^n\right)_{\Gamma_C} \qquad \forall \VEC{v}\in \VEC{H}_D^1(\Omega),
    \end{split}
\end{equation}
where $\langle\,\cdot\,,\,\cdot\,\rangle$ denotes the duality pairing between $\VEC{H}^1_D(\Omega)$ and $(\VEC{H}^1_D(\Omega))^*$.
We equip $\VEC{H}_D^1(\Omega)$ with the following mesh-dependent norm:
\begin{equation}\label{eq:triple norm for dual norm}
    \norm{\VEC{v}}^2 \coloneqq \lVert \VEC{\nabla}\VEC{v}\rVert^2 + \normGamma{\VEC{v}}^2 \qquad \forall \VEC{v}\in \VEC{H}^1_D(\Omega),
\end{equation}
where
\begin{equation}\label{normGamma}
    \normGamma{\VEC{v}}^2 \coloneqq \sum_{F\in \mathcal{F}_h^C} \frac{1}{h_F} \lVert\VEC{v}\rVert_F^2  \qquad \forall \VEC{v}\in \VEC{H}^1_D(\Omega).
\end{equation}
It is easy to show that $\normGamma{\,\cdot\,}$ is subadditive and absolutely homogeneous, i.e., it is a seminorm. As a consequence, also $\norm{\,\cdot\,}$ is subadditive and absolutely homogeneous. Moreover, if we suppose $\norm{\VEC{v}} = 0$, then both $\lVert\VEC{\nabla}\VEC{v}\rVert$ and $\normGamma{\VEC{v}}$ are zero, and this implies $\VEC{v} = \VEC{0}$ by the Friedrichs inequality in $\VEC{H}_D^1(\Omega)$, showing that $\norm{\,\cdot\,}$ is indeed a norm on $\VEC{H}^1_D(\Omega)$.

The dual norm of the residual of a function $\VEC{w}_h\in\VEC{V}_h$ on the normed space $(\VEC{H}_D^1(\Omega), \norm{\,\cdot\,})$ is given by
\begin{equation}\label{eq - definition dual norm}
  \dualnormresidual{\VEC{w}_h}
  \coloneqq \sup_{\substack{\VEC{v}\in \VEC{H}_D^1(\Omega),\, \norm{\VEC{v}} = 1}}
  \left\langle \mathcal{R}(\VEC{w}_h),\VEC{v}\right\rangle.
\end{equation}
In what follows, the quantity $\dualnormresidual{\VEC{u}_h}$ will be used as a measure of the error committed approximating the exact solution $\VEC{u}$ with $\VEC{u}_h$. 

\subsection{A posteriori error estimate}\label{subsection_a posteriori error estimation for $u_h$}

We start this section by introducing the concept of equilibrated stress reconstruction and the definition of five error estimators.

\begin{definition}[Equilibrated stress reconstruction]\label{definition - equilibrated stress reconstruction}
  We will call \emph{equilibrated stress reconstruction} any second-order tensor $\VEC{\sigma}_h$ such that:
  \begin{enumerate}
  \item $\VEC{\sigma}_h\in \mathbb{H}(\textbf{div},\Omega)$,
  \item $\left(\VEC{\rm div}\, %\VEC{\nabla}\cdot
  \VEC{\sigma}_h+\VEC{f},\VEC{v}\right)_T=0$ for every $\VEC{v}\in \VEC{\mathcal{P}}^0(T)$ and every $T\in\mathcal{T}_h$,
  \item $(\VEC{\sigma}_h\VEC{n})|_F\in \VEC{L}^2(F)$ for every $F\in \mathcal{F}_h^N \cup \mathcal{F}_h^C$, and $\left(\VEC{\sigma}_h\VEC{n},\VEC{v}\right)_F = \left(\VEC{g}_N,\VEC{v}\right)_F$ for every $\VEC{v}\in \VEC{\mathcal{P}}^0(F)$ and  every $F\in \mathcal{F}_h^N$,
  \item $\VEC{\sigma}_h^{\VEC{t}}=\VEC{0}$ on $\Gamma_C$.
  \end{enumerate}
\end{definition}

Given an equilibrated stress reconstruction $\VEC{\sigma}_h$, for every element $T\in \mathcal{T}_h$, we define the following local error estimators:
\[
\begin{aligned}
  \eta_{\text{osc},T} &\coloneqq \frac{h_T}{\pi} \lVert \VEC{f}+\VEC{\rm div}\, %\VEC{\nabla}\cdot
  \VEC{\sigma}_h \rVert_T,
  &\qquad&\text{(oscillation)}
  \\
  \eta_{\text{str},T} &\coloneqq \lVert \VEC{\sigma}_h-\VEC{\sigma}(\VEC{u}_h) \rVert_T,
  &\qquad&\text{(stress)}
  \\
  \eta_{\text{Neu},T} &\coloneqq \sum_{F\in \mathcal{F}_T^N} C_{t,T,F}  h_F^{\nicefrac{1}{2}} \lVert \VEC{g}_N-\VEC{\sigma}_h \VEC{n} \rVert_F,
  &\qquad&\text{(Neumann)}
  \\
  \eta_{\text{cnt},T} &\coloneqq \sum_{F\in \mathcal{F}_T^C} h_F^{\nicefrac{1}{2}} \left\lVert \left[P_{1,\gamma}^n(\VEC{u}_h)\right]_{\mathbb{R}^-}-\sigma^n_h \right\rVert_F.
  &\qquad&\text{(contact)}
  %\\
  %\eta_{\text{fric},T} &\coloneqq \sum_{F\in \mathcal{F}_T^C} h_F^{\nicefrac{1}{2}} \lVert \VEC{\sigma}^{\VEC{t}}_h\rVert_F.
  %&\qquad&\text{(no friction)}
\end{aligned}
\]
Here, $C_{t,T,F}$ is the constant of the trace inequality $\lVert\VEC{v}-\overline{\VEC{v}}_F\rVert_F \leq C_{t,T,F} h_F^{\nicefrac{1}{2}} \lVert\VEC{\nabla v}\rVert_T$ with $\overline{\VEC{v}}_F\coloneq\frac{1}{|F|}\int_F\VEC{v}$ valid for every $\VEC{v}\in H^1(T)$ and $F\in\mathcal{F}_T$; see   \cite[Theorem 4.6.3]{Vohralik2015} or \cite[Section 1.4]{Di-Pietro.Ern:12}.

The estimator $\eta_{\text{osc},T}$ represents the residual of the force balance equation \eqref{eq:unilateral equilibrium} inside the element $T$, $\eta_{\text{str},T}$ the difference between the Cauchy stress tensor computed from the approximate solution and the equilibrated stress reconstruction, $\eta_{\text{Neu},T}$ the residual of the Neumann boundary condition \eqref{eq:unilateral Neumann}, and $\eta_{\text{cnt},T}$ %and $\eta_{\text{fric},T}$ the residual of the normal and tangential conditions \eqref{eq:unilateral contact 1} and \eqref{eq:unilateral contact 2} on the contact boundary, respectively.
the residual of the normal condition \eqref{eq:unilateral contact 1} on the contact boundary.

\begin{theorem}[A posteriori error estimate for the dual norm of the residual]\label{theorem - a posteriori error estimation for $u_h$}
  Let $\VEC{u}_h$ be the solution of \eqref{Nitsche-based_method theta=0}, $\mathcal{R}(\VEC{u}_h)$ the residual defined by \eqref{residual definition2}, and $\VEC{\sigma}_h$ an equilibrated stress reconstruction in the sense of Definition \ref{definition - equilibrated stress reconstruction}. Then,
  \begin{equation*}
    \dualnormresidual{\VEC{u}_h}
    \leq \Biggl(\sum_{T\in\mathcal{T}_h} \Bigl((\eta_{\emph{osc},T}+\eta_{\emph{str},T}+\eta_{\emph{Neu},T})^2 + (\eta_{\emph{cnt},T}%+\eta_{\emph{fric},T}
    )^2 \Bigr) \Biggr)^{\nicefrac{1}{2}}.
  \end{equation*}
\end{theorem}

\begin{proof}
  Thanks to the regularity of $\VEC{\sigma}_h$ and of its normal trace (see Properties 1. and 3. in Definition \ref{definition - equilibrated stress reconstruction}), the following Green formula holds:
  \begin{equation}\label{eq:Green's formula}
    (\VEC{\sigma}_h, \VEC{\nabla v})
    =
    -\left(\VEC{\rm div}\, %\VEC{\nabla}\cdot
    \VEC{\sigma}_h, \VEC{v}\right) +
    (\VEC{\sigma}_h\VEC{n}, \VEC{v})_{\Gamma_N} + (\sigma_h^n, v^n)_{\Gamma_C} %+ (\VEC{\sigma}_h^{\VEC{t}}, \VEC{v^t})_{\Gamma_C}
    \qquad \forall \VEC{v}\in \VEC{H}_D^1(\Omega),
  \end{equation}
  where we have also used the decomposition \eqref{eq:decomposition.normal.tangential} of the normal stress reconstruction $\VEC{\sigma}_h\VEC{n}$ into normal and tangential components on the contact boundary $\Gamma_C$, and the fact that $\VEC{\sigma}_h^{\VEC{t}}|_{\Gamma_C} = \VEC{0}$ thanks to Property 4. in Definition \ref{definition - equilibrated stress reconstruction}.
  Now, fix $\VEC{v}\in \VEC{H}^1_D(\Omega)$ such that $\norm{\VEC{v}}^2 = \lVert\nabla\VEC{v}\rVert^2 + \normGamma{\VEC{v}}^2 = 1$ and consider the argument of the supremum in the definition \eqref{eq - definition dual norm} of the dual norm of the residual.
  Expanding $L(\cdot)$ and $a(\cdot,\cdot)$ according to their definition \eqref{eq:definition a and L}, adding
   and subtracting the term $(\VEC{\sigma}_h, \VEC{\nabla}\VEC{v})$, using the symmetry of $\VEC{\sigma}(\VEC{u}_h)$, and applying Green's formula \eqref{eq:Green's formula}, we obtain
  \begin{equation*}
    \begin{split}
      \langle \mathcal{R}(\VEC{u}_h), \VEC{v}\rangle
      &= (\VEC{f},\VEC{v}) + (\VEC{g}_N,\VEC{v})_{\Gamma_N} - \left(\VEC{\sigma}(\VEC{u}_h),\VEC{\varepsilon} (\VEC{v})\right) + \left(\left[P_{1,\gamma}^n(\VEC{u}_h)\right]_{\mathbb{R}^-}, v^n\right)_{\Gamma_C}
      \\
      &\quad
      + (\VEC{\sigma}_h, \VEC{\nabla}\VEC{v}) - (\VEC{\sigma}_h, \VEC{\nabla}\VEC{v})
      \\
      &= (\VEC{f} + \VEC{\rm div}\, %\VEC{\nabla}\cdot
      \VEC{\sigma}_h,\VEC{v})
      + (\VEC{\sigma}_h-\VEC{\sigma}(\VEC{u}_h),\VEC{\nabla}\VEC{v})
      + (\VEC{g}_N-\VEC{\sigma}_h\VEC{n},\VEC{v})_{\Gamma_N} \\
      &\quad
      + \left( \left[P_{1,\gamma}^n(\VEC{u}_h) \right]_{\mathbb{R}^-} - \sigma^n_h,v^n \right)_{\Gamma_C}
      %- (\VEC{\sigma}^{\VEC{t}}_h,\VEC{v^t})_{\Gamma_C}
      \\
      &\eqcolon\mathfrak{T}_1+\cdots+\mathfrak{T}_4.
    \end{split}
  \end{equation*}
  We estimate each term separately.
  Denoting by $\VEC{\Pi}^0_T$ the $L^2$-orthogonal projection onto $\VEC{\mathcal{P}}^0(T)$, and
  using Property 2. of Definition \ref{definition - equilibrated stress reconstruction} with test function $\VEC{\Pi}^0_T \VEC{v} \in \VEC{\mathcal{P}}^0(T)$, the Cauchy-Schwarz inequality, and the Poincaré inequality $\lVert\VEC{v} - \VEC{\Pi}_T^0\VEC{v}\rVert_T \leq h_T \pi^{-1} \lVert\VEC{\nabla v}\rVert_T$, the first term becomes
  \begin{equation*}
    \begin{aligned}
      \mathfrak{T}_1 &= \sum_{T\in \mathcal{T}_h} (\VEC{f}+\VEC{\rm div}\, %\VEC{\nabla}\cdot
      \VEC{\sigma}_h,\VEC{v}-\VEC{\Pi}_T^0 \VEC{v})_T \leq \sum_{T\in \mathcal{T}_h} \lVert \VEC{f}+\VEC{\rm div}\, %\VEC{\nabla}\cdot
      \VEC{\sigma}_h\rVert_T \lVert \VEC{v}-\VEC{\Pi}^0_T \VEC{v} \rVert_T  \\
      &\leq \sum_{T\in\mathcal{T}_h} \frac{h_T}{\pi} \lVert \VEC{f}+\VEC{\rm div}\, %\VEC{\nabla}\cdot
      \VEC{\sigma}_h\rVert_T \lVert \VEC{\nabla v}\rVert_T = \sum_{T\in \mathcal{T}_h} \eta_{\text{osc},T} \lVert \VEC{\nabla v}\rVert_T.
    \end{aligned}
  \end{equation*}
  For the second term, we simply use the Cauchy-Schwarz inequality:
  \begin{equation*}
    \begin{split}
      \mathfrak{T}_2 \leq \sum_{T\in\mathcal{T}_h} \lVert \VEC{\sigma}_h-\VEC{\sigma}(\VEC{u}_h)\rVert_T \lVert \VEC{\nabla v}\rVert_T = \sum_{T\in\mathcal{T}_h} \eta_{\text{str},T} \lVert \VEC{\nabla v}\rVert_T.
    \end{split}
  \end{equation*}
  Denoting by $\VEC{\Pi}^0_F$ the $L^2$-orthogonal projection onto $\VEC{\mathcal{P}}^0(F)$, and
  using Property 3. of Definition \ref{definition - equilibrated stress reconstruction} with  $\VEC{\Pi}^0_F \VEC{v} \in \VEC{\mathcal{P}}^0(F)$ as a test function, the Cauchy-Schwarz inequality, and the trace inequality $\lVert\VEC{v} - \VEC{\Pi}_F^0\VEC{v}\rVert_F \leq C_{t,T,F} h_F^{\nicefrac{1}{2}} \lVert \VEC{\nabla v}\rVert_T$, $F\subset \partial T$, we have
  \begin{equation*}
    \begin{aligned}
      \mathfrak{T}_3 & = \sum_{T\in\mathcal{T}_h} \sum_{F\in \mathcal{F}_T^N} (\VEC{g}_N-\VEC{\sigma}_h\VEC{n},\VEC{v}-\VEC{\Pi}_F^0 \VEC{v})_F \leq \sum_{T\in\mathcal{T}_h} \sum_{F\in \mathcal{F}_T^N} \lVert \VEC{g}_N-\VEC{\sigma}_h\VEC{n}\rVert_F \lVert\VEC{v}-\VEC{\Pi}^0_F \VEC{v}\rVert_F  \\
      &\leq \sum_{T\in\mathcal{T}_h} \sum_{F\in \mathcal{F}_T^N} C_{t,T,F} h_F^{\nicefrac{1}{2}} \lVert \VEC{g}_N-\VEC{\sigma}_h\VEC{n}\rVert_F \lVert\VEC{\nabla v}\rVert_T = \sum_{T\in\mathcal{T}_h} \eta_{\text{Neu},T} \lVert\VEC{\nabla v}\rVert_T,
    \end{aligned}
  \end{equation*}
  where we recall that, for any $T\in\mathcal{T}_h$, $\mathcal{F}_T^N$ collects the Neumann faces of $T$ contained in $\mathcal{F}_h^N$.
  Finally, we consider the term %terms 
  on $\Gamma_C$.
  We define, for all $T\in\mathcal{T}_h$, the local counterpart of the seminorm \eqref{normGamma}
  \begin{equation*}
    \normT{\VEC{v}}^2\coloneqq\sum_{F\in \mathcal{F}_T^C} \frac{1}{h_F} \lVert\VEC{v}\rVert_F^2,
  \end{equation*}
  where $\mathcal{F}_T^C$ is the (possibly empty) set collecting the contact faces of $T$ contained in $\partial T\cap\Gamma_C$.
  Then, using the Cauchy-Schwarz inequality, we obtain
  \begin{equation*}
    \begin{aligned}
      \mathfrak{T}_4 &\leq \sum_{T\in \mathcal{T}_h} \sum_{F\in \mathcal{F}_T^C} \left\lVert \left[P_{1,\gamma}^n(\VEC{u}_h) \right]_{\mathbb{R}^-} - \sigma^n_h\right\rVert_F \lVert v^n\rVert_F  \leq \sum_{T\in\mathcal{T}_h} \sum_{F\in \mathcal{F}_T^C} h_F^{\nicefrac{1}{2}} \left\lVert \left[P_{1,\gamma}^n(\VEC{u}_h) \right]_{\mathbb{R}^-} - \sigma^n_h\right\rVert_F \normT{\VEC{v}}  \\
      &= \sum_{T\in\mathcal{T}_h} \eta_{\text{cnt},T} \normT{\VEC{v}}.
    \end{aligned}
  \end{equation*}
%  and
%  \begin{equation*}
%    \begin{aligned}
%      \mathfrak{T}_5 &\leq \sum_{T\in\mathcal{T}_h} \sum_{F\in \mathcal{F}_T^C} h_F^{\nicefrac{1}{2}} \lVert \VEC{\sigma}^{\VEC{t}}_h\rVert_F \normT{\VEC{v}} = \sum_{T\in\mathcal{T}_h} \eta_{\text{fric},T} \normT{\VEC{v}}.
%    \end{aligned}
%  \end{equation*}
  Let, for the sake of brevity, $\eta_{a,T} \coloneqq \eta_{\text{osc},T}+\eta_{\text{str},T}+\eta_{\text{Neu},T}$ %and $\eta_{b,T} \coloneqq \eta_{\text{cnt},T}+\eta_{\text{fric},T}$ 
  for any element $T\in\mathcal{T}_h$.
  Combining the above results and applying the Cauchy-Schwarz inequality, we obtain
  \begin{equation*}
    \begin{aligned}
      \dualnormresidual{\VEC{u}_h}
      &\leq \sup_{\substack{\VEC{v}\in \VEC{H}^1_D(\Omega),\, \norm{\VEC{v}} = 1}} \Biggl\{\sum_{T\in\mathcal{T}_h} \Bigl( \eta_{a,T} \lVert \VEC{\nabla}\VEC{v}\rVert_T + %\eta_{b,T}
      \eta_{\rm cnt, T}
      \normT{\VEC{v}}\Bigr)\Biggr\}  \\
      &\leq \sup_{\substack{\VEC{v}\in \VEC{H}^1_D(\Omega),\, \norm{\VEC{v}} = 1}} \Biggl\{\Biggl(\sum_{T\in\mathcal{T}_h}\left((\eta_{a,T})^2 + (%\eta_{b,T})^2
      \eta_{\rm cnt,T})^2
      \right)\Biggr)^{\nicefrac{1}{2}} \Biggl(\sum_{T\in\mathcal{T}_h} \left(\lVert \VEC{\nabla}\VEC{v}\rVert_T^2 + \normT{\VEC{v}}^2 \right) \Biggr)^{\nicefrac{1}{2}} \Biggr\}  \\
      &= \Biggl(\sum_{T\in\mathcal{T}_h}\Bigl((\eta_{\text{osc},T}+\eta_{\text{str},T}+\eta_{\text{Neu},T})^2 + (\eta_{\text{cnt},T}%+\eta_{\text{fric},T})^2
      )^2
      \Bigr) \Biggr)^{\nicefrac{1}{2}}.\qedhere
    \end{aligned}
  \end{equation*}
\end{proof}

\begin{remark}[A posteriori error estimate for stress reconstructions with contact friction]
	Even without Property 4. in Definition \ref{definition - equilibrated stress reconstruction} one can easily obtain an a posteriori error estimate similarly to Theorem \ref{theorem - a posteriori error estimation for $u_h$}:
	introducing a fifth local estimator
	\begin{equation*}
		\eta_{\text{fric},T} \coloneqq \sum_{F\in \mathcal{F}_T^C} h_F^{\nicefrac{1}{2}} \lVert \VEC{\sigma}^{\VEC{t}}_h\rVert_F,
		\qquad\text{(friction)}
	\end{equation*}
	which represents the residual of the tangential condition \eqref{eq:unilateral contact 2} on the contact boundary, one gets
	\begin{equation*}
	\dualnormresidual{\VEC{u}_h}
	\leq \Biggl(\sum_{T\in\mathcal{T}_h} \Bigl((\eta_{\emph{osc},T}+\eta_{\emph{str},T}+\eta_{\emph{Neu},T})^2 + (\eta_{\emph{cnt},T}+\eta_{\emph{fric},T}
	)^2 \Bigr) \Biggr)^{\nicefrac{1}{2}}.
	\end{equation*}
\end{remark}

\subsection{Comparison between the residual dual norm and the energy norm}\label{sec:comparison norms}

The goal of this section is to compare the dual norm $\dualnormresidual{\VEC{u}_h}$ with the energy norm $\energynorm{\VEC{u}-\VEC{u}_h}$ of the error, where
\begin{equation}\label{eq - energy norm}
\energynorm{\VEC{v}}^2 \coloneqq a(\VEC{v},\VEC{v}) = \left(\VEC{\sigma}(\VEC{v}), \VEC{\varepsilon}(\VEC{v})\right) \qquad \forall \VEC{v}\in \VEC{H}^1_D(\Omega).
\end{equation}

\begin{remark}[Coercivity of the bilinear form $a$]
  The bilinear form $a(\,\cdot\,,\,\cdot\,)$ on the space $(\VEC{H}^1_D(\Omega),\normHuno{\,\cdot\,})$ is elliptic with a constant $\alpha$ which depends on the Lamé parameter $\mu$ and on the Korn constant $C_K$:
  \begin{equation}\label{eq:ellipticity a}
    \alpha\normHuno{\VEC{v}}^2 \leq a(\VEC{v},\VEC{v}) = \energynorm{\VEC{v}}^2 \qquad \forall\VEC{v}\in \VEC{H}^1_D(\Omega).
  \end{equation}
\end{remark}

Throughout the rest of this section, we adopt the following shorthand notation:
  For every $a,b\in\mathbb{R}$, we write $a\lesssim b$ for $a \leq C b$ with $C>0$ independent of the mesh size $h$ and of the Nitsche parameter $\gamma_0$.

\begin{theorem}[Control of the energy norm]\label{theorem - upper bound of the energy norm}
  Assume that the solution $\VEC{u}$ of the continuous problem \eqref{StrongFormulation} belongs to $\VEC{H}^{\frac{3}{2}+\nu}(\Omega)$ for some $\nu >0$, and let $\VEC{u}_h\in\VEC{V}_h$ be the solution of the discrete problem \eqref{Nitsche-based_method theta=0}.  %!%\eqref{NitscheMethod} 
  %!%with $\theta=0$. 
  Then,
  \begin{equation}\label{thesis upper bound without saturation assumption}
    \alpha^{\nicefrac{1}{2}} \energynorm{\VEC{u}-\VEC{u}_h}
    \lesssim
    \dualnormresidual{\VEC{u}_h}
    + \sum_{F\in \mathcal{F}_h^C} \frac{1}{h_F^{\nicefrac{1}{2}}} \left\lVert \sigma^n(\VEC{u})-\left[P_{1,\gamma}^n(\VEC{u}_h)\right]_{\mathbb{R}^-} \right\rVert_{F}.
  \end{equation}
  Furthermore, if the saturation assumption (see \cite{Chouly-Mlika2017, Chouly-Fabre2017})
  \begin{equation}\label{saturation inequality}
    \left\lVert{\left(\frac{\gamma_0}{\gamma}\right)^{\nicefrac{1}{2}} 		\sigma^n(\VEC{u}-\VEC{u}_h)} \right\rVert_{\Gamma_C} \lesssim 					\normHuno{\VEC{u}-\VEC{u}_h}
  \end{equation}
  holds and $\gamma_0$ is sufficiently large, then
  \begin{equation}\label{thesis upper bound}
    \energynorm{\VEC{u}-\VEC{u}_h} \lesssim
    \dualnormresidual{\VEC{u}_h}.
  \end{equation}
\end{theorem}

\begin{remark}[Role of the regularity assumption]
  In Theorem \ref{theorem - upper bound of the energy norm}, the solution of the contact problem \eqref{StrongFormulation} $\VEC{u}$ is supposed to be sufficiently regular in order to ensure that the normal component of the Cauchy stress tensor is square-integrable on the contact boundary $\Gamma_C$. As a matter of fact, this regularity assumption implies $\VEC{\sigma}(\VEC{u})\VEC{n} \in \VEC{H}^\nu(\Gamma_C) \subset \VEC{L}^2(\Gamma_C)$.
\end{remark}

\begin{proof}[Proof of Theorem \ref{theorem - upper bound of the energy norm}]
  The proof adapts the ideas of \cite[Theorem 3.5]{Chouly-Fabre2017}.
  \medskip \\
  \noindent\emph{1) Proof of \eqref{thesis upper bound without saturation assumption}.}
  Let $\VEC{v}_h\in\VEC{V}_h$.
  Using the definition \eqref{eq - energy norm} of the energy norm and the bilinearity of $a(\,\cdot\,,\,\cdot\,)$, we can write
  \begin{equation}\label{eq:division energy norm}
    \begin{split}
      \energynorm{\VEC{u}-\VEC{u}_h}^2 &= a(\VEC{u}-\VEC{u}_h,\VEC{u}-\VEC{u}_h) = a(\VEC{u},\VEC{u}-\VEC{u}_h) - a(\VEC{u}_h,\VEC{u}-\VEC{v}_h)- a(\VEC{u}_h, \VEC{v}_h-\VEC{u}_h).
    \end{split}
  \end{equation}

  For the term $a(\VEC{u},\VEC{u}-\VEC{u}_h)$, we first use the definition \eqref{eq:definition a and L} of the bilinear form $a(\,\cdot\,,\,\cdot\,)$ followed by the symmetry of the Cauchy stress tensor $\VEC{\sigma}(\VEC{u})$ to replace $\VEC{\varepsilon}(\VEC{u}-\VEC{u}_h)$ with $\VEC{\nabla}(\VEC{u}-\VEC{u}_h)$, then an integration by parts, and, finally, the fact that $\VEC{u}$ satisfies \eqref{StrongFormulation} a.e. to infer:
  \begin{equation}\label{eq:upper first term a}
    \begin{split}
      a(\VEC{u},\VEC{u}-\VEC{u}_h)
      &= \left(\VEC{\sigma}(\VEC{u}), \VEC{\varepsilon}(\VEC{u}-\VEC{u}_h) \right)
      =\left(\VEC{\sigma}(\VEC{u}), \VEC{\nabla}(\VEC{u}-\VEC{u}_h) \right)
      \\
      &=
      -\left(\VEC{\rm div}\, %\VEC{\nabla}\cdot
      \VEC{\sigma}(\VEC{u}), \VEC{u}-\VEC{u}_h \right)
      + \left(\VEC{\sigma}(\VEC{u})\VEC{n}, \VEC{u}-\VEC{u}_h\right)_{\partial\Omega}  \\
      &= \left(\VEC{f}, \VEC{u}-\VEC{u}_h \right) + \left(\VEC{g}_N, \VEC{u}-\VEC{u}_h \right)_{\Gamma_N} + \left(\sigma^n(\VEC{u}), u^n-u^n_h \right)_{\Gamma_C}.
    \end{split}
  \end{equation}
  Notice that, in the last term, only the normal component of the traction appears as $\VEC{\sigma}^{\VEC{t}}(\VEC{u})=\VEC{0}$ on $\Gamma_C$ by \eqref{eq:unilateral contact 2}.
  
  Concerning the term $a(\VEC{u}_h,\VEC{v}_h-\VEC{u}_h)$ in \eqref{eq:division energy norm}, since $\VEC{u}_h$ solves \eqref{Nitsche-based_method theta=0} and $\VEC{v}_h-\VEC{u}_h\in\VEC{V}_h$, we have 
  \begin{equation}\label{eq:upper second term a}
    a(\VEC{u}_h,\VEC{v}_h-\VEC{u}_h)
    = \left( \VEC{f}, \VEC{v}_h-\VEC{u}_h \right) + \left( \VEC{g}_N, \VEC{v}_h-\VEC{u}_h \right)_{\Gamma_N} + \left( \left[P_{1,\gamma}^n(\VEC{u}_h)\right]_{\mathbb{R}^-}, v_h^n-u_h^n \right)_{\Gamma_C},
  \end{equation}
  where we have additionally expanded the linear form $L(\,\cdot\,)$ according to its definition \eqref{eq:definition a and L}.
  
  Plugging \eqref{eq:upper first term a} and \eqref{eq:upper second term a} into \eqref{eq:division energy norm}, we then obtain
  \begin{equation}\label{eq - initial result}
    \begin{split}
      \energynorm{\VEC{u}-\VEC{u}_h}^2
      &= 
      \left( \VEC{f}, \VEC{u}-\VEC{v}_h \right) + \left( \VEC{g}_N, \VEC{u}-\VEC{v}_h \right)_{\Gamma_N} + \left( \sigma^n(\VEC{u}), u^n-u^n_h \right)_{\Gamma_C}
      \\
      &\quad
      - \left( \VEC{\sigma}(\VEC{u}_h), \VEC{\varepsilon}(\VEC{u}-\VEC{v}_h)\right) - \left( \left[P_{1,\gamma}^n(\VEC{u}_h)\right]_{\mathbb{R}^-}, v^n_h-u^n_h \right)_{\Gamma_C}
      = \mathfrak{T}_1+\mathfrak{T}_2
    \end{split}
  \end{equation}
  where, recalling the definition \eqref{residual definition2} of the residual,
  \begin{equation*}
    \mathfrak{T}_1
    %% &\coloneqq \left( \VEC{f}, \VEC{u}-\VEC{v}_h \right) + \left( \VEC{g}_N, \VEC{u}-\VEC{v}_h \right)_{\Gamma_N} - \left( \VEC{\sigma}(\VEC{u}_h), \VEC{\varepsilon}(\VEC{u}-\VEC{v}_h) \right) + \left( \left[P_{1,\gamma}^n(\VEC{u}_h)\right]_{\mathbb{R}^-}, u^n-v^n_h \right)_{\Gamma_C} \\
    \coloneq \left\langle \mathcal{R}(\VEC{u}_h), \VEC{u}-\VEC{v}_h\right\rangle,\qquad`
    \mathfrak{T}_2 \coloneqq \left( \sigma^n(\VEC{u})-\left[P_{1,\gamma}^n(\VEC{u}_h)\right]_{\mathbb{R}^-}, u^n-u^n_h \right)_{\Gamma_C} .
  \end{equation*}
  Notice that the reformulation of $\mathfrak{T}_1$ in terms of the residual $\mathcal{R}(\VEC{u}_h)$ is a consequence of \eqref{residual definition2} and of the definition \eqref{eq:definition a and L} of the linear form $L(\,\cdot\,)$ and of the bilinear form $a(\,\cdot\,, \,\cdot\,)$.

  For the first term, we can write, by definition \eqref{eq - definition dual norm} of the dual norm,
  \begin{equation}\label{eq:first estimation P1}
    \begin{split}
      \mathfrak{T}_1 %!%&= \norm{\VEC{u}-\VEC{v}_h} \Biggl[ \left( \VEC{f}, \frac{\VEC{u}-\VEC{v}_h}{\norm{\VEC{u}-\VEC{v}_h}} \right) + \left( \VEC{g}_N, \frac{\VEC{u}-\VEC{v}_h}{\norm{\VEC{u}-\VEC{v}_h}} \right) -\\
      %!%& \hspace{2.5cm} - \left( \VEC{\sigma}(\VEC{u}_h), \VEC{\varepsilon}\left(\frac{\VEC{u}-\VEC{v}_h}{\norm{\VEC{u}-\VEC{v}_h}}\right) \right) + 
      %!%\left( \left[P_{1,\gamma}^n(\VEC{u}_h)\right]_{\mathbb{R}^-}, \frac{u^n-v^n_h}{\norm{\VEC{u}-\VEC{v}_h}} \right)_{\Gamma_C} \Biggr]\\
      &\leq \norm{\VEC{u}-\VEC{v}_h}
      \dualnormresidual{\VEC{u}_h}.
    \end{split}
  \end{equation}   
  We now want to show that $\norm{\VEC{u}-\VEC{v}_h}\lesssim \energynorm{\VEC{u}-\VEC{u}_h}$ for a properly selected function $\VEC{v}_h$. From now on, we fix $\VEC{v}_h=\VEC{u}_h+\interpol(\VEC{u}-\VEC{u}_h)$, where $\interpol\colon \VEC{H}^1_D(\Omega)\to \VEC{V}_h$ is the quasi-interpolation operator defined in \cite[Eq. (4.11)]{Bernardi1998}, whose main properties are summarized in \cite[Lemma 2.1]{Chouly-Fabre2017}. We analyze separately the two parts composing the norm $\norm{\,\cdot\,}$ (see \eqref{eq:triple norm for dual norm}). 
  For the $\VEC{H}^1$-seminorm, we use, in this order, the triangle inequality, the choice of $\VEC{v}_h$, the boundedness in the $\VEC{H}^1$-norm of the operator $\interpol$ (i.e., $\normHuno{\interpol\VEC{v}}\lesssim \normHuno{\VEC{v}}$ for every $\VEC{v}\in \VEC{H}^1_D(\Omega)$), and the ellipticity \eqref{eq:ellipticity a} of the bilinear form $a(\,\cdot\,,\,\cdot\,)$ to write:
  \begin{equation}\label{eq:norm of gradient estimate}
    \begin{aligned}
    \lVert \VEC{\nabla}(\VEC{u}-\VEC{v}_h)\rVert
    &\leq \normHuno{\VEC{u}-\VEC{v}_h}
    \leq \normHuno{\VEC{u}-\VEC{u}_h} + \normHuno{\VEC{u}_h-\VEC{v}_h}  \\
    &= \normHuno{\VEC{u}-\VEC{u}_h} + \normHuno{\interpol(\VEC{u}-\VEC{u_h})}
    \lesssim \normHuno{\VEC{u}-\VEC{u}_h} \leq \alpha^{-\nicefrac{1}{2}} \energynorm{\VEC{u}-\VEC{u}_h}.
    \end{aligned}
  \end{equation}
  Next, using %
  the definition \eqref{normGamma} of the seminorm $\normGamma{\,\cdot\,}$, the choice of $\VEC{v}_h$, and the ellipticity \eqref{eq:ellipticity a} of $a(\,\cdot\,,\,\cdot\,)$, we obtain:
  \begin{equation}\label{Gammanorm1}
    \begin{split}
      \normGamma{\VEC{u}-\VEC{v}_h}^2 &= \sum_{F\in \mathcal{F}_h^C} \frac{1}{h_F} \lVert \VEC{u}-\VEC{v}_h\rVert_F^2 = \sum_{F\in \mathcal{F}_h^C} \frac{1}{h_F} \left\lVert\VEC{u}-\VEC{u}_h-\interpol(\VEC{u}-\VEC{u}_h)\right\rVert_F^2  \\
      &\lesssim \sum_{F\in \mathcal{F}_h^C}  \left\lVert\VEC{u}-\VEC{u}_h\right\rVert_{1,\tilde{\omega}_F}^2 \lesssim \normHuno{\VEC{u}-\VEC{u}_h}^2 \leq \alpha^{-1} \energynorm{\VEC{u}-\VEC{u}_h}^2,
    \end{split}
  \end{equation}
  where, to pass to the second line, we have used the following trace approximation property of $\interpol$ (see \cite[Lemma 2.1]{Chouly-Fabre2017}):
  For all $F\in \mathcal{F}_h$,
  \begin{equation*}
    \left\lVert \VEC{v}-\interpol\VEC{v}\right\rVert_F \lesssim h_F^{\nicefrac{1}{2}} \left\lVert\VEC{v}\right\rVert_{1,\tilde{\omega}_F} \qquad  \forall \VEC{v}\in \VEC{H}^1_D(\Omega),
  \end{equation*}
  with $\tilde{\omega}_F$ standing for the union of the mesh elements sharing at least one vertex with $F$, see Figure \ref{fig:illustration tilde omega F}.
  Recalling the definition \eqref{eq:triple norm for dual norm} of the triple norm, squaring \eqref{eq:norm of gradient estimate} and summing it to \eqref{Gammanorm1}, and taking the square root of the resulting inequality, we conclude that
  \begin{equation*}
    \norm{\VEC{u}-\VEC{v}_h} = \left(\lVert\VEC{\nabla}(\VEC{u}-\VEC{v}_h)\rVert^2 + \normGamma{\VEC{u}-\VEC{v}_h}^2 \right)^{\nicefrac{1}{2}} \lesssim \alpha^{-\nicefrac{1}{2}} \energynorm{\VEC{u}-\VEC{u}_h}.
  \end{equation*}
  Combining this bound with \eqref{eq:first estimation P1}, we obtain
  \begin{equation}\label{estimation P1}
    \mathfrak{T}_1 \lesssim \alpha^{-\nicefrac{1}{2}} \energynorm{\VEC{u}-\VEC{u}_h} \dualnormresidual{\VEC{u}_h}.
  \end{equation}

  \begin{figure}[tb]
    \centering
    \begin{subfigure}{0.45\textwidth}
      \centering
      \begin{tikzpicture}[scale=0.75]
        \fill[fill=gray!20] (0.5,2) -- (3.5,1.5) -- (3.5,-0.5) -- (1,-1.5) -- (-1.5,-2) -- (-3.5,-1) -- (-2,1.5) -- (0.5,2);
        \draw[thick] (0.5,2) -- (3.5,1.5) -- (3.5,-0.5) -- (1,-1.5) -- (-1.5,-2) -- (-3.5,-1) -- (-2,1.5) -- (0.5,2);
        \draw[thick] (-2,1.5) -- (-1.5,0) -- (0.5,2) -- (1.5,0) -- (3.5,1.5);
        \draw[thick] (3.5,-0.5) -- (1.5,0) -- (-1.5,-2) -- (-1.5,0) -- (-3.5,-1);
        \draw[thick] (1.5,0) -- (1,-1.5);
        \draw[very thick, blue] (-1.5,0) -- (1.5, 0);
        \node at (-0.25,-0.4) {\textcolor{blue}{$F$}};
      \end{tikzpicture}
    \end{subfigure}
    \hfill
    \begin{subfigure}{0.45\textwidth}
      \centering
      \begin{tikzpicture}[scale=0.7]
        \fill[fill=gray!20] (-4,0) -- (5,0) -- (4,1.5) -- (1,3) -- (-3,2.5) -- (-4,0);
        \draw[thick] (-4,0) -- (5,0) -- (4,1.5) -- (1,3) -- (-3,2.5) -- (-4,0);
        \draw[thick] (1,3) -- (-2,0) -- (-3,2.5);
        \draw[thick] (1,3) -- (2,0) -- (4,1.5);
        \draw[very thick, blue] (-2,0) -- (2,0);
        \node at (0,-0.4) {\textcolor{blue}{$F$}};
      \end{tikzpicture}
    \end{subfigure}
    \caption{Illustration of $\tilde{\omega}_F$ for $F\in\mathcal{F}_h^i$ (\emph{left}) and for $F\in\mathcal{F}_h^b$ (\emph{right}).}
    \label{fig:illustration tilde omega F}
  \end{figure}
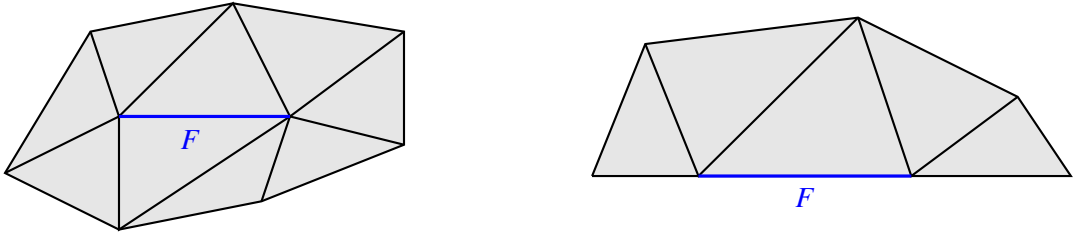

  We now consider the term $\mathfrak{T}_2$. Using the Cauchy-Schwarz and trace inequalities, we have
  \begin{equation}\label{estimation P2 without saturation assumption}
    \begin{split}
      \mathfrak{T}_2 &=
      \sum_{F\in\mathcal{F}_h^C}\left( \sigma^n(\VEC{u}) - \left[P_{1,\gamma}^n(\VEC{u}_h)\right]_{\mathbb{R}^-}, u^n-u^n_h \right)_F
       \\
       &\lesssim\bigg(
       \sum_{F\in \mathcal{F}_h^C} \frac{1}{h_F} \left\lVert \sigma^n(\VEC{u}) - \left[P_{1,\gamma}^n(\VEC{u}_h)\right]_{\mathbb{R}^-} \right\rVert_{F}^2
       \bigg)^{\nicefrac12}\normHuno{\VEC{u}-\VEC{u}_h}  \\
       & \lesssim \alpha^{-\nicefrac{1}{2}}
       \bigg(
       \sum_{F\in\mathcal{F}_h^C} \frac{1}{h_F} \left\lVert \sigma^n(\VEC{u}) - \left[P_{1,\gamma}^n(\VEC{u}_h)\right]_{\mathbb{R}^-} \right\rVert_{F}^2
       \bigg)^{\nicefrac{1}{2}} \energynorm{\VEC{u}-\VEC{u}_h}.
    \end{split}
  \end{equation}
  Inserting \eqref{estimation P1} and \eqref{estimation P2 without saturation assumption} into \eqref{eq - initial result}, we obtain \eqref{thesis upper bound without saturation assumption}.
  \medskip \\
  \noindent
  \emph{2) Proof of \eqref{thesis upper bound}.}
  For the second part of the theorem, we work under the saturation assumption \eqref{saturation inequality}.
  Using the contact condition $\sigma^n(\VEC{u})=\left[P_{1,\gamma}^n(\VEC{u})\right]_{\mathbb{R}^-}$ (see \eqref{eq:normal contact condition rewritten}), and the definition \eqref{eq:definition P_theta,gamma} of the operator $P_{1,\gamma}^n$, we have
  \begin{equation}\label{eq:initial estimate P2}
    \begin{split}
      \mathfrak{T}_2
      &= \left(\left[P_{1,\gamma}^n(\VEC{u})\right]_{\mathbb{R}^-} - \left[P_{1,\gamma}^n(\VEC{u}_h)\right]_{\mathbb{R}^-}, u^n-u^n_h \right)_{\Gamma_C}  \\
      &= \left(\left[P_{1,\gamma}^n(\VEC{u})\right]_{\mathbb{R}^-} - \left[P_{1,\gamma}^n(\VEC{u}_h)\right]_{\mathbb{R}^-}, \frac{1}{\gamma} \bigl[\gamma(u^n-u^n_h) - \sigma^n(\VEC{u}-\VEC{u}_h) + \sigma^n(\VEC{u}-\VEC{u}_h) \bigr] \right)_{\Gamma_C} \\
      &= -\left(\left[P_{1,\gamma}^n(\VEC{u})\right]_{\mathbb{R}^-} - \left[P_{1,\gamma}^n(\VEC{u}_h)\right]_{\mathbb{R}^-}, \frac{1}{\gamma} \left(P_{1,\gamma}^n(\VEC{u})-P_{1,\gamma}(\VEC{u}_h)\right) \right)_{\Gamma_C} \\
      &\quad + \left(\left[P_{1,\gamma}^n(\VEC{u})\right]_{\mathbb{R}^-} - \left[P_{1,\gamma}^n(\VEC{u}_h)\right]_{\mathbb{R}^-}, \frac{1}{\gamma} \sigma^n(\VEC{u}-\VEC{u}_h) \right)_{\Gamma_C}.
    \end{split}
  \end{equation}
  Due to the fact that $a[a]_{\mathbb{R}^-}=([a]_{\mathbb{R}^-})^2$ and $a[b]_{\mathbb{R}^-} \leq [a]_{\mathbb{R}^-} [b]_{\mathbb{R}^-}$, it follows that
  \begin{equation*}
    (a-b) ([a]_{\mathbb{R}^-}-[b]_{\mathbb{R}^-}) = a[a]_{\mathbb{R}^-} + b[b]_{\mathbb{R}^-} - a[b]_{\mathbb{R}^-} - b[a]_{\mathbb{R}^-} \geq ([a]_{\mathbb{R}^-}-[b]_{\mathbb{R}^-})^2
  \end{equation*}
  for every $a,b\in\mathbb{R}$.
  Using the latter inequality with $(a,b) = (P_{1,\gamma}^n(\VEC{u}), P_{1,\gamma}^n(\VEC{u}_h))$ for the first term in \eqref{eq:initial estimate P2} and the Cauchy-Schwarz inequality for the second one, we have
  \begin{equation*}
    \begin{split}
      \mathfrak{T}_2 &\leq - \left\lVert \gamma^{-\nicefrac{1}{2}} \left(\left[P_{1,\gamma}^n(\VEC{u})\right]_{\mathbb{R}^-} - \left[P_{1,\gamma}^n(\VEC{u}_h)\right]_{\mathbb{R}^-} \right) \right\rVert_{\Gamma_C}^2 \\
      &\quad
      + \left\lVert \gamma^{-\nicefrac{1}{2}} \left(\left[P_{1,\gamma}^n(\VEC{u})\right]_{\mathbb{R}^-}-\left[P_{1,\gamma}^n(\VEC{u}_h) \right]_{\mathbb{R}^-} \right) \right\rVert_{\Gamma_C} \left\lVert \gamma^{-\nicefrac{1}{2}} \sigma^n(\VEC{u}-\VEC{u}_h) \right\rVert_{\Gamma_C}.
    \end{split}
  \end{equation*}
  We continue using the generalized Young inequality $ab\leq a^2+b^2/4$ for the second term followed by the saturation assumption \eqref{saturation inequality} to write:
  %13-05%\begin{equation*}
  %13-05%\begin{split}
  %13-05%\mathcal{P}_2 &\leq %08-04%\left(1 -1 \right) \left\lVert \gamma^{-\nicefrac{1}{2}} \left(\left[P_{1,\gamma}^n(\VEC{u})\right]_{\mathbb{R}^-} - \left[P_{1,\gamma}^n(\VEC{u}_h)\right]_{\mathbb{R}^-} \right) \right\rVert_{\Gamma_C}^2 +
  %13-05%\frac{1}{4}\cdot \frac{1}{\gamma_0} \left\lVert{\left(\frac{\gamma_0}{\gamma}\right)^{\nicefrac{1}{2}} \sigma^n(\VEC{u}-\VEC{u}_h)}\right\rVert_{\Gamma_C}^2.
  %13-05%\end{split}
  %13-05%\end{equation*}
  %!%for any $\beta>0$. 
  %!%Choosing $\beta=\nicefrac{1}{2}$ and 
  %13-05%Using the saturation assumption \eqref{saturation inequality}, we obtain
  \begin{equation}\label{estimation P2}
    \mathfrak{T}_2 \leq \frac{1}{4\gamma_0} \left\lVert{\left(\frac{\gamma_0}{\gamma}\right)^{\nicefrac{1}{2}} \sigma^n(\VEC{u}-\VEC{u}_h)}\right\rVert_{\Gamma_C}^2 \lesssim \frac{1}{4 \gamma_0} \normHuno{\VEC{u}-\VEC{u}_h}^2 \leq \frac{1}{4 \gamma_0 \alpha} \energynorm{\VEC{u}-\VEC{u}_h}^2.
  \end{equation}
  Combining \eqref{eq - initial result}, \eqref{estimation P1}, and \eqref{estimation P2} we finally get, for a suitable real number $C>0$,
  \begin{equation*}
    \left(\alpha^{\nicefrac{1}{2}}- \frac{C}{4 \gamma_0 \alpha} \right) \energynorm{\VEC{u}-\VEC{u}_h}^2 \leq C \energynorm{\VEC{u}-\VEC{u}_h} \dualnormresidual{\VEC{u}_h}
  \end{equation*}
  and, taking $\gamma_0$ sufficiently large,
  \begin{equation*}
    \energynorm{\VEC{u}-\VEC{u}_h} \lesssim \dualnormresidual{\VEC{u}_h},
  \end{equation*}
  thus concluding the proof of \eqref{thesis upper bound}.
\end{proof}
\medskip

\begin{theorem}[Control of the dual norm of the residual]\label{theorem - lower bound of the energy norm}
Assume that the solution $\VEC{u}$ of the continuous problem \eqref{StrongFormulation} belongs to $\VEC{H}^{\frac{3}{2}+\nu}(\Omega)$ for some $\nu>0$, and let $\VEC{u}_h\in\VEC{V}_h$ be the solution of the discrete problem \eqref{Nitsche-based_method theta=0}. %!%\eqref{NitscheMethod} with $\theta=0$. 
Then,
\begin{equation}\label{lower bound energy norm 1}
  \dualnormresidual{\VEC{u}_h}
  \leq (d \lambda + 4\mu)^{\nicefrac{1}{2}} \energynorm{\VEC{u}-\VEC{u}_h} + \left(\sum_{F\in\mathcal{F}_h^C} h_F \left\lVert \sigma^n(\VEC{u})-\left[P_{1,\gamma}^n(\VEC{u}_h)\right]_{\mathbb{R}^-}\right\rVert_F^2\right)^{\nicefrac{1}{2}}.
\end{equation}
Moreover, if the saturation assumption \eqref{saturation inequality} holds, then
\begin{equation}\label{lower bound energy norm 2}
  \dualnormresidual{\VEC{u}_h}
  \lesssim \left[(d \lambda + 4\mu)^{\nicefrac{1}{2}}+\alpha^{-\nicefrac{1}{2}}\right]
    \energynorm{\VEC{u}-\VEC{u}_h} + \gamma_0 \left(\sum_{F\in\mathcal{F}_h^C} \frac{1}{h_F} \lVert \VEC{u}-\VEC{u}_h\rVert_F^2\right)^{\nicefrac{1}{2}}.
\end{equation}
\end{theorem}

\begin{proof}
  \emph{1) Proof of \eqref{lower bound energy norm 1}.}
  By definition \eqref{residual definition2} of the residual together with \eqref{StrongFormulation} (valid almost everywhere), and Green's formula, it holds:
  For any $\VEC{v}\in \VEC{H}^1_D(\Omega)$,
  \begin{equation*}
    \begin{split}
      \langle\mathcal{R}(\VEC{u}_h),\VEC{v}\rangle_{} &= (\VEC{f}, \VEC{v}) + (\VEC{g}_N, \VEC{v})_{\Gamma_N} - (\VEC{\sigma}(\VEC{u}_h), \VEC{\varepsilon}(\VEC{v})) + \left(\left[P_{1,\gamma}^n(\VEC{u}_h)\right]_{\mathbb{R}^-}, v^n \right)_{\Gamma_C} \\ 
      &= (\VEC{\sigma}(\VEC{u}-\VEC{u}_h),\VEC{\varepsilon}(\VEC{v}))-\left(\sigma^n(\VEC{u})-\left[P_{1,\gamma}^n(\VEC{u}_h)\right]_{\mathbb{R}^-},v^n\right)_{\Gamma_C}.
    \end{split}
  \end{equation*}
  Then, using the symmetry of the Cauchy stress tensor $\VEC{\sigma}(\VEC{u}-\VEC{u}_h)$, %!%applying 
  the Cauchy-Schwarz inequality, and the definition \eqref{eq:triple norm for dual norm} of the norm $\norm{\VEC{v}}$, and additionally observing that $$\lVert\VEC{\sigma}(\VEC{u}-\VEC{u}_h)\rVert \leq (d \lambda + 4\mu)^{\nicefrac{1}{2}} \energynorm{\VEC{u}-\VEC{u}_h},$$ we have
  \begin{align*}
    \langle \mathcal{R}(\VEC{u}_h),\VEC{v}\rangle %&= a(\VEC{u}-\VEC{w}^h,\VEC{v})-\left(\sigma^n(\VEC{u})-\left[P_{1,\gamma}^n(\VEC{u}_h)\right]_{\mathbb{R}^-},v^n\right)_{\Gamma_C} \leq \\
    &\leq \left\lVert\VEC{\sigma}(\VEC{u}-\VEC{u}_h)\right\rVert \left\lVert\VEC{\nabla}\VEC{v}\right\rVert + \sum_{F\in\mathcal{F}_h^C} \left\lVert \sigma^n(\VEC{u})-\left[P_{1,\gamma}^n(\VEC{u}_h)\right]_{\mathbb{R}^-}\right\rVert_F \left\lVert v^n\right\rVert_F  \\
    &\leq (d \lambda + 4\mu)^{\nicefrac{1}{2}} \energynorm{\VEC{u}-\VEC{u}_h} \left\lVert\VEC{\nabla}\VEC{v}\right\rVert + \sum_{F\in\mathcal{F}_h^C} h_F^{\nicefrac{1}{2}}\left\lVert \sigma^n(\VEC{u})-\left[P_{1,\gamma}^n(\VEC{u}_h)\right]_{\mathbb{R}^-}\right\rVert_F \frac{1}{h_F^{\nicefrac{1}{2}}}\left\lVert \VEC{v}\right\rVert_F  \\
    &\leq \left[(d \lambda + 4\mu)^{\nicefrac{1}{2}} \energynorm{\VEC{u}-\VEC{u}_h} + \left(\sum_{F\in\mathcal{F}_h^C} h_F\left\lVert \sigma^n(\VEC{u})-\left[P_{1,\gamma}^n(\VEC{u}_h)\right]_{\mathbb{R}^-}\right\rVert_F^2\right)^{\nicefrac{1}{2}} \right] \norm{\VEC{v}}.
  \end{align*}
  By definition \eqref{eq - definition dual norm} of the dual norm, this yields \eqref{lower bound energy norm 1}.
  \medskip\\
  \noindent
  \emph{2) Proof of \eqref{lower bound energy norm 2}.}
  Under the saturation assumption \eqref{saturation inequality}, starting from \eqref{lower bound energy norm 1} and using \eqref{eq:normal contact condition rewritten}, 
  we obtain, for all $F\in\mathcal{F}_h^C$,
  \begin{equation*}
    \begin{split}
      &h_F \left\lVert \sigma^n(\VEC{u})-\left[P_{1,\gamma}^n(\VEC{u}_h)\right]_{\mathbb{R}^-}\right\rVert_F^2 = h_F \left\lVert \left[P_{1,\gamma}^n(\VEC{u})\right]_{\mathbb{R}^-}-\left[P_{1,\gamma}^n(\VEC{u}_h)\right]_{\mathbb{R}^-}\right\rVert_F^2  \\
      &\hspace{1cm} \leq h_F \left\lVert P_{1,\gamma}^n(\VEC{u}) - P_{1,\gamma}^n(\VEC{u}_h)\right\rVert_F^2 \lesssim h_F \left\lVert \VEC{\sigma}(\VEC{u}-\VEC{u}_h) \right\rVert_F^2 + h_F \left\lVert \gamma(\VEC{u}-\VEC{u}_h)\right\rVert_F^2,
    \end{split}
  \end{equation*}
  where we have applied the property $([a]_{\mathbb{R}^-}-[b]_{\mathbb{R}^-})^2 \leq (a-b)^2$ valid for any $a,b\in\mathbb{R}$, with $(a, b) = (P_{1,\gamma}^n(\VEC{u}), P_{1,\gamma}^n(\VEC{u}_h))$ to pass to the second line and the triangle inequality to conclude.
  Then, using the saturation assumption \eqref{saturation inequality} together with the ellipticity property \eqref{eq:ellipticity a} and the choice of $\gamma$, we obtain:
  \begin{equation*}
    \begin{aligned}
      &\left(\sum_{F\in\mathcal{F}_h^C} h_F\left\lVert \sigma^n(\VEC{u})-\left[P_{1,\gamma}^n(\VEC{u}_h)\right]_{\mathbb{R}^-}\right\rVert_F^2\right)^{\nicefrac{1}{2}} \\
      &\qquad
      \lesssim\left(\sum_{F\in\mathcal{F}_h^C} h_F \left\lVert \VEC{\sigma}(\VEC{u}-\VEC{u}_h) \right\rVert_F^2 + \sum_{F\in\mathcal{F}_h^C} h_F \left\lVert \gamma(\VEC{u}-\VEC{u}_h)\right\rVert_F^2\right)^{\nicefrac{1}{2}}  \\
        &\qquad
        \leq \left( \sum_{T\in\mathcal{T}_h} \sum_{F\in\mathcal{F}_T^C} h_T \left\lVert \VEC{\sigma}(\VEC{u}-\VEC{u}_h) \right\rVert_F^2\right)^{\nicefrac{1}{2}} + \left(\sum_{T\in\mathcal{T}_h} \sum_{F\in\mathcal{F}_T^C} h_F \left(\frac{\gamma_0}{h_T}\right)^2 \left\lVert \VEC{u}-\VEC{u}_h\right\rVert_F^2 \right)^{\nicefrac{1}{2}}  \\
        &\qquad
        \leq \left\lVert\left(\frac{\gamma_0}{\gamma}\right)^{\nicefrac{1}{2}} \VEC{\sigma}(\VEC{u}-\VEC{u}_h) \right\rVert_{\Gamma_C} + \gamma_0 \left(\sum_{F\in\mathcal{F}_h^C} \frac{1}{h_F} \left\lVert \VEC{u}-\VEC{u}_h\right\rVert_F^2 \right)^{\nicefrac{1}{2}}  \\
        &\qquad
        \lesssim \alpha^{-\nicefrac{1}{2}} \energynorm{\VEC{u}-\VEC{u}_h} + \gamma_0 \left(\sum_{F\in\mathcal{F}_h^C} \frac{1}{h_F} \left\lVert \VEC{u}-\VEC{u}_h\right\rVert_F^2 \right)^{\nicefrac{1}{2}}.
    \end{aligned}
  \end{equation*}
  Combining this bound with \eqref{lower bound energy norm 1}, we obtain \eqref{lower bound energy norm 2}.
\end{proof}

\section{Identification of the error components}\label{sec:identification.error.components}

We consider the resolution of the (nonlinear) discrete problem \eqref{Nitsche-based_method theta=0} with an iterative method in which, at each iteration $k\ge1$, the nonlinear term $\left[P_{1,\gamma}^n(\,\cdot\,)\right]_{\mathbb{R}^-}$ is replaced by a linear approximation $P_{\text{lin}}^{k-1}(\,\cdot\,)$. A new approximation of the discrete solution is then obtained solving the following problem:
Find $\VEC{u}_h^k\in\VEC{V}_h$ such that
\begin{equation}\label{Linear approximated problem}
  a(\VEC{u}_h^k,\VEC{v}_h) - \left(P_{\text{lin}}^{k-1}(\VEC{u}_h^k),v_h^n\right)_{\Gamma_C} = L(\VEC{v}_h) \qquad \forall\VEC{v}_h\in\VEC{V}_h.
\end{equation}
The linearized operator $P_{\rm lin}^{k-1}(\,\cdot\,)$ is based on the following regularization of the projection $[\,\cdot\,]_{\mathbb{R}^-}$:
Given a real number $\delta>0$ (representing the amount of regularization),
\[
\left[x\right]_{\text{reg},\delta} \coloneqq \begin{cases}
  x & \qquad \text{if}\ x\leq -\delta\\
  -\displaystyle\frac{1}{4\delta}x^2 + \frac{1}{2}x-\frac{\delta}{4} & \qquad \text{if}\ \lvert x\rvert < \delta\\
  0 & \qquad \text{if}\ x\geq\delta.
\end{cases}
\]
Figure \ref{figure_projection and regularized operators} shows the graphs of the projection operator $\left[\cdot\right]_{\mathbb{R}^-}$ and of the regularized operator $\left[\cdot\right]_{\text{reg},\delta}$.
  Notice that they coincide for $\lvert x\rvert \geq \delta$ and $\left[\cdot\right]_{\text{reg},\delta}$ belongs to $C^1(\mathbb{R})$ (but not to $C^2(\mathbb{R})$).
  The linearized operator $P_{\text{lin}}^{k-1}(\,\cdot\,)$ is obtained setting,  for any $\VEC{w}_h\in\VEC{V}_h$,%
\begin{equation}\label{P_lin,delta definition}
    \begin{split}
        P_{\text{lin},\delta}^{k-1}(\VEC{w}_h) \coloneqq& \left[P_{1,\gamma}^n(\VEC{u}_h^{k-1}) \right]_{\text{reg},\delta} + \frac{\partial \left[P_{1,\gamma}^n (\VEC{v}) \right]_{\text{reg},\delta}}{\partial \VEC{v}} \Bigg|_{\VEC{v}=\VEC{u}_h^{k-1}} \cdot (\VEC{w}_h-\VEC{u}_h^{k-1})  \\
        =& \left[P_{1,\gamma}^n(\VEC{u}_h^{k-1}) \right]_{\text{reg},\delta} + \frac{\mathrm{d} \left[x\right]_{\text{reg},\delta}}{\mathrm{d}x} \Bigg|_{x=P_{1,\gamma}^n(\VEC{u}_h^{k-1})} \left(P_{1,\gamma}^n(\VEC{w}_h)-P_{1,\gamma}^n(\VEC{u}_h^{k-1})\right).
    \end{split}
\end{equation}
Here, we add the subscript $\delta$ to emphasize that the linear operator depends on the choice of this parameter.
The refined error estimate presented in the following section enables an automatic tuning of $\delta$.

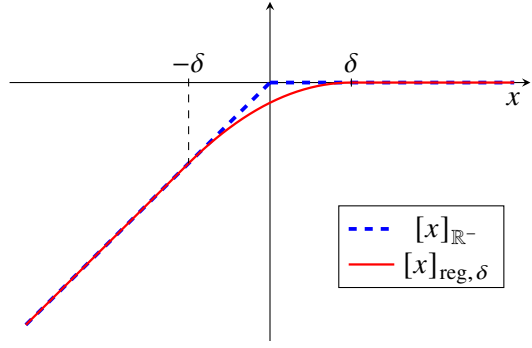
\begin{figure}[tbp]
    \centering
    %\resizebox{6.4cm}{4.2cm}{
    \begin{tikzpicture}
        \begin{axis}[xmin=-3.2,
                xmax=3.2,
                ymax=1,
                ymin=-3.2,
                %xticklabels={,,},      % To delete labels but not tick from the axis
                %yticklabels={,,},
                %xlabel=$x$,
                %enlargelimits = true,
                axis lines = middle,
                legend style={at={(0.95,0.4)}},
                axis equal image]
            \pgfplotsset{ticks=none}        % To delete ticks and labels from the axis
            \addplot[domain=-3:0,blue,line width=0.5mm,dashed] {x};
            \addplot[domain=-3:-1,red,line width=0.3mm] {x};
            \legend{$\left[x\right]_{\mathbb{R}^-}$,$\left[x\right]_{\text{reg},\delta}$};
            \addplot[domain=0:3,blue,line width=0.5mm,dashed] {0};
            \addplot[domain=1:3,red,line width=0.3mm] {0};
            \addplot[domain=-1:1,red,line width=0.3mm] {-0.25*x^2+0.5*x-0.25};
            \node at (3,-0.2) {$x$};
            \draw (1,-0.05) -- (1,0.05);
            \draw[dashed] (-1,0.05) -- (-1,-1);
            \node at (1,0.25) {$\delta$};
            \node at (-1,0.25) {$-\delta$};
       \end{axis}
    \end{tikzpicture}%}
   \caption{Comparison between the projection operator $\left[x\right]_{\mathbb{R}^-}$ (blue) and the regularized operator $\left[x\right]_{\text{reg},\delta}$ (red).}
    \label{figure_projection and regularized operators}
\end{figure}

\subsection{A posteriori error estimate distinguishing the error components}

We present in this section an error estimate for $\VEC{u}_h^k$ which enables one to identify and separate the different components of the error.
The estimate hinges on the following assumption:
\begin{assumption}[Decomposition of the stress reconstruction]\label{assumption on stress reconstruction}
  Let $\VEC{\sigma}_h^k$ be an equilibrated stress reconstruction in the sense of Definition \ref{definition - equilibrated stress reconstruction}.
  Then, $\VEC{\sigma}_h^k$ %
  can be decomposed into three parts 
  \begin{equation}\label{eq:sigmah.decomposition}
    \VEC{\sigma}_h^k = \VEC{\sigma}_{h,\rm dis}^k + \VEC{\sigma}_{h,\rm reg}^k + \VEC{\sigma}_{h,\rm lin}^k,
  \end{equation}
  where $\VEC{\sigma}_{h,\rm dis}^k$ represents \emph{discretization}, $\VEC{\sigma}_{h,\rm reg}^k$ represents \emph{regularization}, and $\VEC{\sigma}_{h,\rm lin}^k$ represents \emph{linearization}.
\end{assumption}

\noindent
In Section \ref{sec - reconstruction} we will show how to obtain an equilibrated stress reconstruction which satisfies this assumption.
Finally, we introduce 
the following local error estimators:
For every element $T\in\mathcal{T}_h$,
\begin{subequations}\label{local estimators for u_h^k}
  \begin{align}
    \eta_{\text{osc},T}^k &\coloneqq \frac{h_T}{\pi} \left\lVert\VEC{f} + \VEC{\rm div}\, %\VEC{\nabla}\cdot
    \VEC{\sigma}_h^k \right\rVert_T,
    &\qquad& \text{(oscillation)}
    \label{oscillation error estimator} \\
    \eta_{\text{str},T}^k &\coloneqq \lVert \VEC{\sigma}_{h,\rm dis}^k-\VEC{\sigma}(\VEC{u}_h^k)\rVert_T,
    &\qquad& \text{(stress)}
    \\
    \eta_{\text{reg1},T}^k &\coloneqq \lVert\VEC{\sigma}_{h,\rm reg}^k\rVert_T
    \qquad \text{and}\qquad
    \eta_{\text{reg2},T}^k \coloneqq \sum_{F\in\mathcal{F}_T^C} h_F^{\nicefrac{1}{2}} \bigl\lVert \sigma_{h,\rm reg}^{k,n} \bigr\rVert_{F},
    &\qquad& \text{(regularization)}
    \\
    \eta_{\text{lin1},T}^k &\coloneqq \lVert\VEC{\sigma}_{h,\rm lin}^k\rVert_T
    \qquad \text{and}\qquad
    \eta_{\text{lin2},T}^k \coloneqq \sum_{F\in\mathcal{F}_T^C} h_F^{\nicefrac{1}{2}} \bigl\lVert \sigma_{h,\rm lin}^{k,n} \bigr\rVert_F,
    &\qquad& \text{(linearization)}
    \\
    \eta_{\text{Neu},T}^k &\coloneqq \sum_{F\in\mathcal{F}_T^N} C_{t,T,F} h_F^{\nicefrac{1}{2}} \left\lVert \VEC{g}_N-\VEC{\sigma}_h^k \VEC{n} \right\rVert_F,
    &\qquad& \text{(Neumann)}
    \label{Neumann error estimator} \\
    \eta_{\rm cnt,T}^k &\coloneqq \sum_{F\in\mathcal{F}_T^C} h_F^{\nicefrac{1}{2}} \left\lVert \left[P_{1,\gamma}^n(\VEC{u}_h^k)\right]_{\mathbb{R}^-} - \sigma^{k,n}_{h,\rm dis}\right\rVert_F.
    &\qquad& \text{(contact)}
    \label{discretization error estimator}
    %% \\
    %% \eta_{\text{reg2},T}^k &\coloneqq \sum_{F\in\mathcal{F}_T^C} h_F^{\nicefrac{1}{2}} \bigl\lVert \sigma_{h,2}^{k,n} \bigr\rVert_{F}, \label{regularization error estimator 2} \\
    %% \eta_{\text{lin2},T}^k &\coloneqq \sum_{F\in\mathcal{F}_T^C} h_F^{\nicefrac{1}{2}} \bigl\lVert \sigma_{h,3}^{k,n} \bigr\rVert_F. \label{linearization error estimator 2}
  \end{align}
\end{subequations}
The corresponding global error estimators are defined setting
\begin{equation}\label{global estimators for u_h^k}
    \eta_{\bullet}^k \coloneqq \left(\sum_{T\in\mathcal{T}_h} \left(\eta_{\bullet,T}^k\right)^2\right)^{\nicefrac{1}{2}}.
\end{equation}

\begin{theorem}[A posteriori error estimate distinguishing the error components]\label{theorem - a posteriori estimation u_h^k}
  Let $\VEC{u}_h^k\in\VEC{V}_h$ be the solution of the linearized problem \eqref{Linear approximated problem} with $P_{\text{lin},\delta}(\,\cdot\,)$ defined by \eqref{P_lin,delta definition}, and let $\mathcal{R}(\VEC{u}_h^k)$ be the residual of $\VEC{u}_h^k$ defined by \eqref{residual definition2}.
  Then, under Assumption \ref{assumption on stress reconstruction}, it holds
  \begin{multline}\label{local upper bound with identification}
    \dualnormresidual{\VEC{u}_h^k}
    \\
    \leq \Biggl[
      \sum_{T\in\mathcal{T}_h} \Bigl((\eta_{\rm osc,T}^k + \eta_{\rm str,T}^k + \eta_{\rm reg1,T}^k + \eta_{\rm lin1,T}^k + \eta_{\rm Neu,T}^k)^2
      + (\eta_{\rm cnt,T}^k + \eta_{\rm reg2,T}^k + \eta_{\rm lin2,T}^k)^2 \Bigr)
      \Biggr]^{\nicefrac{1}{2}}
  \end{multline}
  and, as a result,
  \begin{equation}\label{global upper bound with identification}
    \dualnormresidual{\VEC{u}_h^k}
    \leq \Bigl[
      (\eta_{\rm osc}^k + \eta_{\rm str}^k + \eta_{\rm reg1}^k + \eta_{\rm lin1}^k + \eta_{\rm Neu}^k)^2 + (\eta_{\rm cnt}^k + \eta_{\rm reg2}^k + \eta_{\rm lin2}^k)^2
      \Bigr]^{\nicefrac{1}{2}}.
  \end{equation}
\end{theorem}

\begin{proof}
  Proceeding as in the proof of Theorem \ref{theorem - a posteriori error estimation for $u_h$}, we immediately get
  \begin{multline*}
    \dualnormresidual{\VEC{u}_h^k}
    \\
    \leq \Biggl[
    \sum_{T\in\mathcal{T}_h} \biggl(
      \bigl(\eta_{\text{osc},T}^k + \bigl\lVert\VEC{\sigma}_h^k - \VEC{\sigma}(\VEC{u}_h^k)\bigr\rVert_T + \eta_{\text{Neu},T}^k \bigr)^2
      + \biggr(
      \sum_{F\in\mathcal{F}_T^C} h_F^{\nicefrac{1}{2}} \left\lVert \left[P_{1,\gamma}^n(\VEC{u}_h^k)\right]_{\mathbb{R}^-} - \sigma_h^{k,n}\right\rVert_F
      \biggr)^2
      \biggr)
    \Biggr]^{\nicefrac{1}{2}}.
  \end{multline*}
  Decomposing $\VEC{\sigma}_h^k$ according to \eqref{eq:sigmah.decomposition} and using the triangle inequality, we arrive at \eqref{local upper bound with identification}.
  Finally,
  \eqref{global upper bound with identification} is obtained from \eqref{local upper bound with identification} applying twice the inequality $\sum_{T\in\mathcal{T}_h}\left(\sum_{i=1}^m a_{i,T}\right)^2\le\left(\sum_{i=1}^m a_i\right)^2$ valid for all families of nonnegative real numbers $(a_{i,T})_{1\le i\le m,\, T\in\mathcal{T}_h}$ with $a_i\coloneq\left(\sum_{T\in\mathcal{T}_h}a_{i,T}^2\right)^{\nicefrac12}$ for all $1\le i\le m$.
\end{proof}

\subsection{Fully adaptive algorithm}

We propose an adaptive algorithm based on the error estimators \eqref{local estimators for u_h^k} and \eqref{global estimators for u_h^k}, and on the result of Theorem \ref{theorem - a posteriori estimation u_h^k}.
Denote by $\gamma_{\text{reg}}, \gamma_{\text{lin}} \in (0,1)$ two user-dependent parameters that represent the relative magnitude of the regularization and linearization errors with respect to the total error. 
Moreover, we define the following local estimators:
\[
\begin{gathered}
  \eta_{\text{reg},T}^k \coloneqq \eta_{\text{reg1},T}^k + \eta_{\text{reg2},T}^k,
  \qquad
  \eta_{\text{lin},T}^k \coloneqq \eta_{\text{lin1},T}^k + \eta_{\text{lin2},T}^k.
\end{gathered}
\]
The corresponding global counterparts are given by \eqref{global estimators for u_h^k} with $\bullet\in\{\text{reg},\text{lin}\}$.
With these estimators and the parameters $\gamma_{\text{reg}}$, $\gamma_{\text{lin}}$, we define stopping criteria for the regularization and linearization loops, respectively, so that both the parameter $\delta$ and the number of Newton iterations on every mesh refinement iteration will be fixed automatically by the adaptive algorithm. 
For all $T\in\mathcal{T}_h$, the total error estimator is given by
\begin{gather}
	\eta_{\text{tot},T}^k \coloneqq\left[
    \bigl(\eta_{\text{osc},T}^k + \eta_{\text{str},T}^k + \eta_{\text{reg1},T}^k + \eta_{\text{lin1},T}^k + \eta_{\text{Neu},T}^k\bigr)^2 + \bigl(\eta_{\text{cnt},T}^k + \eta_{\text{reg2},T}^k + \eta_{\text{lin2},T}^k\bigr)^2.
    \right]^{\nicefrac12}\label{eq:local total estimator}
\end{gather}

\begin{algorithm}[H]
\caption{Adaptive algorithm}\label{algorithm}
\begin{algorithmic}[1]
\State \textbf{choose} an initial function $\VEC{u}_h^{0}\in \mathbf{V}_h$, $\delta > 0$, $\gamma_{\text{reg}}, \gamma_{\text{lin}} \in (0,1)$
%\State \textbf{compute} $\VEC{u}_h^{1,0,0}$, $\VEC{\sigma}_h^{1,0,0}$, and the local and global estimators
%\State \textbf{set} $i=1$, $j=0$, $k=0$
\Repeat \ \{mesh refinement loop\}
    \Repeat \ \{regularization loop\}
        \State \textbf{set} $k = 0$
        \Repeat \ \{Newton linearization loop\}
            \State \textbf{set} $k = k+1$
            \State \textbf{setup} the operator $P_{\text{lin},\delta}^{k-1}$ and the linear system
            \State \textbf{compute} $\VEC{u}_h^{k}$, $\VEC{\sigma}_h^{k}$, and the local and global estimators
        \Until {$\eta_{\rm lin}^{k} \leq \gamma_{\rm lin} (\eta_{\rm osc}^{k} + \eta_{\rm str}^{k} + \eta_{\rm Neu}^{k} + \eta_{\rm cnt}^{k})$} \label{alg:global stopping criterion lin} 
        \State \textbf{decrease} $\delta$ (e.g. $\delta = \delta/2$)
    \Until {$\eta_{\rm reg}^{k} \leq \gamma_{\rm reg} (\eta_{\rm osc}^{k} + \eta_{\rm str}^{k} + \eta_{\rm Neu}^{k} + \eta_{\rm cnt}^{k} + \eta_{\rm lin}^{k})$} \label{alg:global stopping criterion reg} 
    \State \textbf{set} $\delta$ at its previous value (e.g. $\delta = 2\delta$)
    \State \textbf{refine} the elements of the mesh where %$\eta_{\rm osc,T}^{k} + \eta_{\rm str,T}^{k} + \eta_{\rm Neu,T}^{k} + \eta_{\rm cnt,T}^{k} + \eta_{\rm lin,T}^{k} + \eta_{\rm reg,T}^{k}$ 
    $\eta_{\rm tot, T}^k$
    is higher
    \State \textbf{update} data
    \Until {$\eta_{\rm tot, T}^k$ is distributed evenly over the mesh}
\end{algorithmic}
\end{algorithm}

\begin{remark}[Local stopping criteria]
  The stopping critera in Lines \ref{alg:global stopping criterion lin} and \ref{alg:global stopping criterion reg} can alternatively be enforced locally inside each element:
  \begin{subequations}\label{eq:local stopping criteria}
    \begin{alignat}{4}
    	\eta_{\rm lin, T}^{k}
      &\leq \gamma_{\rm lin, T} (\eta_{\rm osc, T}^{k} + \eta_{\rm str, T}^{k} + \eta_{\rm Neu, T}^{k} + \eta_{\rm cnt, T}^{k})
      &\qquad& \forall T \in\mathcal{T}_h,
    	\label{eq:local stopping criterion 1} \\
    	\eta_{\rm reg, T}^{k}
      &\leq \gamma_{\rm reg, T} (\eta_{\rm osc, T}^{k} + \eta_{\rm str, T}^{k} + \eta_{\rm Neu, T}^{k} + \eta_{\rm cnt, T}^{k} + \eta_{\rm lin, T}^{k})
      &\qquad& \forall T \in\mathcal{T}_h,
    	\label{eq:local stopping criterion 2}
    \end{alignat}
  \end{subequations}
    where the parameters $\gamma_{\rm lin, T}, \gamma_{\rm reg, T} \in (0,1)$ can possibly vary element by element; see, e.g., \cite{Jiranek.Strakos.ea:10} and also the discussion in \cite[Section 4.1]{DiPietro2015}.
\end{remark}

\section{Equilibrated stress reconstructions}\label{sec - reconstruction}

We first show how to construct an equilibrated stress reconstruction $\VEC{\sigma}_h$ that satisfies the conditions of Definition \ref{definition - equilibrated stress reconstruction}, then modify the construction to match Assumption \ref{assumption on stress reconstruction}.
 
\subsection{Basic equilibrated stress reconstruction}\label{sec - reconstruction - basic}

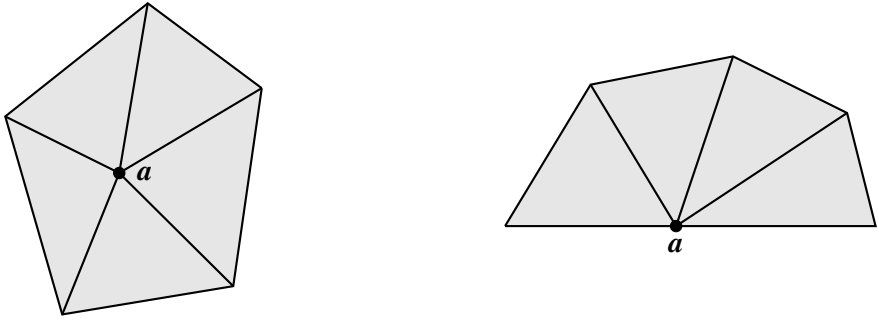
\begin{figure}[tb]
    \centering
    \begin{subfigure}{0.45\textwidth}
        \centering
        \begin{tikzpicture}[scale=0.75]
            \fill[fill=gray!20] (2.5,1.5) -- (2,-2) -- (-1,-2.5) -- (-2,1) -- (0.5,3) -- (2.5,1.5);
            \draw[thick] (0,0) -- (2.5,1.5);
            \draw[thick] (0,0) -- (2,-2);
            \draw[thick] (0,0) -- (-1,-2.5);
            \draw[thick] (0,0) -- (-2,1);
            \draw[thick] (0,0) -- (0.5,3);
            \draw[thick] (2.5,1.5) -- (2,-2) -- (-1,-2.5) -- (-2,1) -- (0.5,3) -- (2.5,1.5);
            \draw[fill=black] (0,0) circle (0.1);
            \node at (0.45,-0.01) {$\VEC{a}$};
        \end{tikzpicture}
    \end{subfigure}
    \hfill
    \begin{subfigure}{0.45\textwidth}
        \begin{tikzpicture}[scale=0.75]
            \fill[fill=gray!20] (-3,0) -- (-1.5,2.5) -- (1,3) -- (3,2) -- (3.5,0) -- (-3,0);
            \draw[thick] (0,0) -- (-1.5,2.5);
            \draw[thick] (0,0) -- (1,3);
            \draw[thick] (0,0) -- (3,2);
            \draw[thick] (-3,0) -- (-1.5,2.5) -- (1,3) -- (3,2) -- (3.5,0) -- (-3,0);
            \draw[fill=black] (0,0) circle (0.1);
            \node at (0,-0.35) {$\VEC{a}$};
        \end{tikzpicture}
    \end{subfigure}
    \caption{Illustration of a patch $\omega_{\VEC{a}}$ around an inner node $\VEC{a}\in\mathcal{V}_h^i$ (\emph{left}), and around a boundary node $\VEC{a}\in\mathcal{V}_h^b$ (\emph{right}).}
    \label{fig:illustration patch}
\end{figure}

Following the path of \cite{Botti2018}, we construct $\VEC{\sigma}_h$ patchwise around the mesh vertices using Arnold--Falk--Winther mixed finite element spaces \cite{Arnold2007}, which are based on a stress tensor constructed in the Brezzi--Douglas--Marini space (see, e.g., \cite[Chapter 3]{Brezzi1991} and \cite[Chapter 2]{Boffi2013}), along with Lagrange multipliers that enforce a weak symmetry constraint.
From this point on, $\mathcal{V}_h^D$ will denote the set collecting all mesh vertices which lie on some Dirichlet boundary face. Notice that $\mathcal{V}_h^D$ also contains the vertices lying at the intersection between $\Gamma_D$ and $\Gamma_{\bullet}$, $\bullet\in\{N, C\}$.

For any element $T\in\mathcal{T}_h$, and any integer $q\geq 1$, we set
\[
\VEC{\Sigma}_T \coloneqq \MAT{P}^q(T),\qquad
\VEC{U}_T \coloneqq \VEC{\mathcal{P}}^{q-1}(T),\qquad
\VEC{\Lambda}_T \coloneqq \left\{\VEC{\mu}\in\MAT{P}^{q-1}(T) \ :\ \VEC{\mu}=-\VEC{\mu}^T \right\}.
\]
At the global level, we define the following spaces
\[
\begin{aligned}
  \VEC{\Sigma}_h &\coloneqq \left\{\VEC{\tau}_h \in \MAT{H}(\textbf{div},\Omega) \ :\ \VEC{\tau}_h|_{T} \in \VEC{\Sigma}_T \ \text{for any}\ T\in\mathcal{T}_h\right\},\\
  \VEC{U}_h &\coloneqq \left\{\VEC{v}_h \in \VEC{L}^2(\Omega) \ :\ \VEC{v}_h|_T\in\VEC{U}_T \ \text{for any}\ T\in\mathcal{T}_h \right\},\\
  \VEC{\Lambda}_h &\coloneqq \left\{\VEC{\mu}_h \in \MAT{L}^2(\Omega) \ :\ \VEC{\mu}_h|_T \in \VEC{\Lambda}_T \ \text{for any}\ T\in\mathcal{T}_h \right\}.
\end{aligned}
\]
Notice that $\VEC{\Sigma}_h \subset \MAT{H}(\textbf{div},\Omega)$ implies that its elements have continuous normal components across interfaces \cite[Lemma 1.17]{DiPietro2019}.
Now, let $q=p$, let $\VEC{u}_h$ be the solution of \eqref{Nitsche-based_method theta=0}, and fix a mesh vertex $\VEC{a}$. We denote by $\omega_{\VEC{a}}$ the patch around the node $\VEC{a}$, see Figure \ref{fig:illustration patch}, by $\VEC{n}_{\omega_{\VEC{a}}}$ the normal unit outward vector on its boundary $\partial\omega_{\VEC{a}}$
and by $\psi_{\VEC{a}}$ the hat function associated with $\VEC{a}$.
On any patch $\omega_{\VEC{a}}$ we then define  the following spaces:
\begin{align}
  \VEC{\Sigma}_h^{\VEC{a}} &\coloneqq \begin{cases}
    \big\{
    \VEC{\tau}_h \in\VEC{\Sigma}_h(\omega_{\VEC{a}}) \ :\  
    \text{$\VEC{\tau}_h \VEC{n}_{\omega_{\VEC{a}}} = \VEC{0}$ on $\partial\omega_{\VEC{a}}\setminus\Gamma_D$}
    \big\}
    &\text{if $\VEC{a}\in\mathcal{V}_h^b$}
    \\
    \big\{
    \VEC{\tau}_h \in\VEC{\Sigma}_h(\omega_{\VEC{a}})\ :\
    \text{%
      $\VEC{\tau}_h \VEC{n}_{\omega_{\VEC{a}}} = \VEC{0}$ on $\partial\omega_{\VEC{a}}$
    }
   	\big\}
   	& \text{otherwise},
  \end{cases}\label{eq:Sigmaha} \\
  \VEC{\Sigma}_{h,N,C}^{\VEC{a}} &\coloneqq \begin{cases}
    \begin{aligned}
      &\big\{\VEC{\tau}_h\in\VEC{\Sigma}_h(\omega_{\VEC{a}}) \ :\ \VEC{\tau}_h \VEC{n}_{\omega_{\VEC{a}}} = \VEC{0} \ \text{on}\ \partial\omega_{\VEC{a}}\setminus\partial\Omega,\\
      &\quad \VEC{\tau}_h \VEC{n}_{\omega_{\VEC{a}}} = \Pi_{\VEC{\Sigma}_h \VEC{n}_{\omega_{\VEC{a}}}} \left(\psi_{\VEC{a}} \VEC{g}_N \right) \ \text{on}\ \partial\omega_{\VEC{a}}\cap\Gamma_N, \ \text{and}\\
      &\qquad \VEC{\tau}_h \VEC{n}_{\omega_{\VEC{a}}} = \Pi_{\VEC{\Sigma}_h \VEC{n}_{\omega_{\VEC{a}}}} \bigl(\psi_{\VEC{a}} \left[P_{1,\gamma}^n(\VEC{u}_h)\right]_{\mathbb{R}^-} \VEC{n}\bigr) \ \text{on}\ \partial\omega_{\VEC{a}}\cap\Gamma_C \big\},
    \end{aligned} &
    \text{if $\VEC{a}\in\mathcal{V}_h^b$}
    \\
    \VEC{\Sigma}_h^{\VEC{a}} & \text{otherwise}
  \end{cases} \label{eq:definition Sigma h,N,C}
  \\
  \VEC{U}_h^{\VEC{a}} &\coloneqq \begin{cases}
    \VEC{U}_h(\omega_{\VEC{a}}) &
    \text{if $\VEC{a} \in \mathcal{V}_h^D$}
    \\
    \left\{\VEC{v}_h \in\VEC{U}_h(\omega_{\VEC{a}}) \ :\ (\VEC{v}_h,\VEC{z})_{\omega_{\VEC{a}}} = 0 \ \text{for any}\ \VEC{z}\in\VEC{RM}^d\right\} &
    \text{otherwise},
 	\end{cases}\nonumber\\
  \VEC{\Lambda}_h^{\VEC{a}} &\coloneqq \VEC{\Lambda}_h(\omega_{\VEC{a}})\nonumber.
\end{align}
Above, $\VEC{\Sigma}_h(\omega_{\VEC{a}})$, $\VEC{U}_h(\omega_{\VEC{a}})$, and $\VEC{\Lambda}_h(\omega_{\VEC{a}})$ denote the restrictions of the spaces $\VEC{\Sigma}_h$, $\VEC{U}_h$ and $\VEC{\Lambda}_h$ to the subdomain $\omega_{\VEC{a}}$, respectively. Moreover, $\VEC{\Sigma}_h\VEC{n}_{\omega_{\VEC{a}}}$ is the space of normal traces on the patch boundary $\partial\omega_{\VEC{a}}$ of elements in $\VEC{\Sigma}_h(\omega_{\VEC{a}})$, i.e., it is the space of vector-valued broken polynomials of total degree $\le p$ on the set of boundary faces of the patch, while $\VEC{RM}^d$ is space of rigid-body motions, i.e., $\VEC{RM}^2 \coloneqq \left\{\VEC{b}+c(x_2,-x_1)^{\top} \ :\ \VEC{b}\in\mathbb{R}^2, c\in\mathbb{R}\right\}$ and $\VEC{RM}^3 \coloneqq \left\{\VEC{b}+\VEC{c}\times\VEC{x} \ :\ \VEC{b},\VEC{c}\in\mathbb{R}^3\right\}$.

\begin{remark}[Boundary condition for the reconstruction on internal vertices]
	In the definition \eqref{eq:Sigmaha} of $\VEC{\Sigma}_h^{\VEC{a}}$, we distinguish between boundary and internal vertices in order to ensure, in the case $\VEC{a}\in\mathcal{V}_h^i$, zero normal components on the whole boundary of the patch $\omega_{\VEC{a}}$ even if $\lvert \partial\omega_{\VEC{a}} \cap \Gamma_D\rvert > 0$.  
\end{remark}

\begin{construction}[Basic equilibrated stress reconstruction]\label{reconstruction1}
  Let, for any vertex $\VEC{a}\in\mathcal{V}_h$, $(\VEC{\sigma}_h^{\VEC{a}},\VEC{r}_h^{\VEC{a}},\VEC{\lambda}_h^{\VEC{a}}) \in \VEC{\Sigma}_{h,N,C}^{\VEC{a}} \times \VEC{U}_h^{\VEC{a}} \times \VEC{\Lambda}_h^{\VEC{a}}$ be the solution to the following problem:
\begin{subequations}\label{eq - reconstruction1}
    \begin{alignat}{2}
      (\VEC{\sigma}_h^{\VEC{a}},\VEC{\tau}_h)_{\omega_{\VEC{a}}} + (\VEC{r}_h^{\VEC{a}},\VEC{\rm div}\, %\VEC{\nabla}\cdot 
      \VEC{\tau}_h)_{\omega_{\VEC{a}}} + (\VEC{\lambda}_h^{\VEC{a}},\VEC{\tau}_h)_{\omega_{\VEC{a}}} &= (\psi_{\VEC{a}}\VEC{\sigma}(\VEC{u}_h),\VEC{\tau}_h)_{\omega_{\VEC{a}}}
      &\quad& \forall\VEC{\tau}_h\in\VEC{\Sigma}_h^{\VEC{a}},\\ 
      (\VEC{\rm div}\, %\VEC{\nabla}\cdot
      \VEC{\sigma}_h^{\VEC{a}},\VEC{v}_h)_{\omega_{\VEC{a}}} &= (-\psi_{\VEC{a}} \VEC{f} + \VEC{\sigma}(\VEC{u}_h)\VEC{\nabla}\psi_{\VEC{a}}, \VEC{v}_h)_{\omega_{\VEC{a}}}
      &\quad& \forall\VEC{v}_h\in\VEC{U}_h^{\VEC{a}}, \label{second_equation Reconstruction 1} \\ 
      (\VEC{\sigma}_h^{\VEC{a}},\VEC{\mu}_h)_{\omega_{\VEC{a}}} &= 0
      &\quad& \forall\VEC{\mu}_h \in \VEC{\Lambda}_h^{\VEC{a}}. \label{third_equation Reconstruction 1}
    \end{alignat}
\end{subequations}
Extending $\VEC{\sigma}_h^{\VEC{a}}$ by zero outside the patch $\omega_{\VEC{a}}$, we set $\VEC{\sigma}_h\coloneqq\sum_{\VEC{a}\in\mathcal{V}_h} \VEC{\sigma}_h^{\VEC{a}}$.
\end{construction}

By definition of the space $\VEC{\Sigma}_{h,N,C}^{\VEC{a}}$, a homogeneous Neumann %!%no-flux
boundary condition is enforced on the whole boundary of $\omega_{\VEC{a}}$ for interior vertices and on $\partial \omega_{\VEC{a}} \setminus \partial\Omega$ for boundary vertices. In particular, for boundary vertices in $\mathcal{V}_h^b\setminus\mathcal{V}_h^D$, a possibly non homogeneous Neumann boundary condition is enforced on the boundary faces of the patch.
Therefore, when $\VEC{a}\in\mathcal{V}_h^i$ or $\VEC{a}\in\mathcal{V}_h^b\setminus\mathcal{V}_h^D$, the right hand side of \eqref{second_equation Reconstruction 1} has to verify the following Neumann compatibility condition:
\begin{multline}\label{Neumann_compatibility_condition1}
  (-\psi_{\VEC{a}} \VEC{f} + \VEC{\sigma}(\VEC{u}_h)\VEC{\nabla}\psi_{\VEC{a}}, \VEC{z})_{\omega_{\VEC{a}}} \\
  = \left(\Pi_{\VEC{\Sigma}_h \VEC{n}_{\omega_{\VEC{a}}}} \left(\psi_{\VEC{a}} \VEC{g}_N \right),\VEC{z}\right)_{\partial\omega_{\VEC{a}}\cap \Gamma_N} + \Bigl(\Pi_{\VEC{\Sigma}_h \VEC{n}_{\omega_{\VEC{a}}}} \left(\psi_{\VEC{a}} \left[P_{1,\gamma}^n(\VEC{u}_h)\right]_{\mathbb{R}^-} \VEC{n}\Bigr), \VEC{z}\right)_{\partial\omega_{\VEC{a}}\cap \Gamma_C}
\end{multline}
for any $\VEC{z}\in\VEC{RM}^d$. Fixing a rigid-body motion $\VEC{z}$, it is possible to check that \eqref{Neumann_compatibility_condition1} holds by taking $\psi_{\VEC{a}} \VEC{z}$ as test function in \eqref{Nitsche-based_method theta=0}. 
The following Lemma lists the main properties of the tensor $\VEC{\sigma}_h$ resulting from Construction \ref{reconstruction1}. In particular, it shows that $\VEC{\sigma}_h$ satisfies all the conditions of Definition \ref{definition - equilibrated stress reconstruction}, i.e., it is an equilibrated stress reconstruction.

\begin{lemma}[Properties of $\VEC{\sigma}_h$]\label{lemma - properties of reconstruction1}
  Let $\VEC{\sigma}_h$ be defined by Construction \ref{reconstruction1}. Then, it holds
  \begin{enumerate}
  \item $\VEC{\sigma}_h\in\MAT{H}(\emph{\textbf{div}},\Omega)$;
  \item For every $T\in\mathcal{T}_h$ and every $\VEC{v}_T\in \VEC{\mathcal{P}}^{p-1}(T)$, $(\VEC{\rm div}\, %\VEC{\nabla}\cdot
  \VEC{\sigma}_h+ \VEC{f},\VEC{v}_T)_T=0$;
  \item For every $F\in\mathcal{F}_h^N$ and every $\VEC{v}_F\in\VEC{\mathcal{P}}^p(F)$, $(\VEC{\sigma}_h\VEC{n},\VEC{v}_F)_F=(\VEC{g}_N,\VEC{v}_F)_F$;
  \item For every $F\in\mathcal{F}_h^C$ and every $\VEC{v}_F\in\VEC{\mathcal{P}}^p(F)$,
  \[
  (\VEC{\sigma}_h\VEC{n},\VEC{v}_F)_F=\left(\left[P_{1,\gamma}^n(\VEC{u}_h)\right]_{\mathbb{R}^-} \VEC{n},\VEC{v}_F\right)_F = \left(\left[P_{1,\gamma}^n(\VEC{u}_h)\right]_{\mathbb{R}^-}, v_F^n\right)_F.
  \]
  \end{enumerate}
\end{lemma}

\begin{proof}
  1) By definition, $\VEC{\sigma}_h^{\VEC{a}}\in\MAT{H}(\textbf{div},\omega_{\VEC{a}})$ for any $\VEC{a}\in\mathcal{V}_h$. Due to the no-flux boundary condition on internal faces enforced in the local problem on $\omega_{\VEC{a}}$, the extension of $\VEC{\sigma}_h^{\VEC{a}}$ by zero outside the patch is in $\MAT{H}(\textbf{div},\Omega)$ and, as a consequence, $\VEC{\sigma}_h\in\MAT{H}(\textbf{div},\Omega)$.
  \medskip\\
  2) First, we check that, for any $\VEC{a}\in \mathcal{V}_h$, equation \eqref{second_equation Reconstruction 1} holds for every $\VEC{v}_h\in \VEC{U}_h(\omega_{\VEC{a}})$.
  If $\VEC{a}\in\mathcal{V}_h^D$, this is trivial since $\VEC{U}_h^{\VEC{a}} = \VEC{U}_h(\omega_{\VEC{a}})$.
  If, on the other hand, $\VEC{a}\in\mathcal{V}_h\setminus\mathcal{V}_h^D$, it is sufficient to use the fact that $\VEC{U}_h^{\VEC{a}}=(\VEC{RM}^d)^\bot$ (with orthogonal taken with respect to the $\VEC{L}^2(\omega_{\VEC{a}})$-product) along with the Green formula, the definition \eqref{eq:definition Sigma h,N,C} of $\VEC{\Sigma}_h^{\VEC{a}}$, the Neumann compatibility condition \eqref{Neumann_compatibility_condition1}, and \eqref{third_equation Reconstruction 1}.
  
  Now, fix $T\in\mathcal{T}_h$ and let $\VEC{v}_T \in \VEC{\mathcal{P}}^{p-1}(T)$.
  Extending $\VEC{v}_T$ by zero outside of $T$, we have $\VEC{v}_T\in\VEC{U}_h(\omega_{\VEC{a}})$ for all $\VEC{a}\in\mathcal{V}_T$. Indeed, by definition, $\VEC{U}_h(\omega_{\VEC{a}})$ is composed by piecewise polynomials of degree at most $p-1$ that can be chosen independently inside each element of the patch.
  Summing \eqref{second_equation Reconstruction 1} over $\VEC{a}\in \mathcal{V}_T$ we obtain:
  \begin{equation*}
    \begin{split}
      0 &=\sum_{\VEC{a}\in\mathcal{V}_T} \Bigl[\left(\VEC{\rm div}\, %\VEC{\nabla}\cdot
      \VEC{\sigma}_h^{\VEC{a}}, \VEC{v}_T \right)_{\omega_{\VEC{a}}} + \left(\psi_{\VEC{a}} \VEC{f}, \VEC{v}_T \right)_{\omega_{\VEC{a}}} - \left( \VEC{\sigma}(\VEC{u}_h) \VEC{\nabla} \psi_{\VEC{a}}, \VEC{v}_T\right)_{\omega_{\VEC{a}}}\Bigr] = \left( \VEC{\rm div}\, %\VEC{\nabla}\cdot
      \VEC{\sigma}_h + \VEC{f}, \VEC{v}_T\right)_{T}.
    \end{split}
  \end{equation*}
  Here, we have used the fact that $\VEC{\sigma}_h|_T = \sum_{\VEC{a}\in\mathcal{V}_T} \VEC{\sigma}_h^{\VEC{a}}|_T$ and $\sum_{\VEC{a}\in\mathcal{V}_T} \psi_{\VEC{a}} = 1$ over $T$ (so that, in particular, $\sum_{\VEC{a}\in\mathcal{V}_T}\VEC{\nabla}\psi_{\VEC{a}}\equiv 0$).
  \medskip \\
  3,4) We only detail the proof of 3) as that of 4) is similar.
  Let $F\in\mathcal{F}_h^N$ and let $\VEC{v}_F$ be a polynomial defined on $F$ from the discrete normal trace space $(\VEC{\Sigma}_h\VEC{n})|_{F}$, i.e., a polynomial of total degree at most $p$. Then, by the definition of $\VEC{\Sigma}_{h,N,C}^{\VEC{a}}$ \eqref{eq:definition Sigma h,N,C},
  \begin{equation*}
    \left(\VEC{\sigma}_h\VEC{n}, \VEC{v}_F\right)_F = \sum_{\VEC{a}\in\mathcal{V}_F} \left(\VEC{\sigma}_h^{\VEC{a}}\VEC{n}, \VEC{v}_F\right)_F = \sum_{\VEC{a}\in\mathcal{V}_F} \left(\psi_{\VEC{a}}\VEC{g}_N,\VEC{v}_F\right)_F = \left(\VEC{g}_N,\VEC{v}_F\right)_F.\qedhere
  \end{equation*}
\end{proof}

\subsection{Stress reconstruction distinguishing the error components}\label{sec - reconstruction - error decomposition}

In Construction \ref{reconstruction1} we used the solution $\VEC{u}_h$ of the nonlinear problem \eqref{Nitsche-based_method theta=0} to reconstruct an equilibrated stress $\VEC{\sigma}_h$.
However, as argued in Section \ref{sec:identification.error.components}, in practice we only dispose of an approximated solution obtained by means of a linearization method. 
Let $k\geq 1$ be an integer and let $\VEC{u}_h^k$ be the solution of the linearized problem \eqref{Linear approximated problem} with operator $P_{\text{lin},\delta}(\,\cdot\,)$ defined by \eqref{P_lin,delta definition}. 
Then, for any boundary vertex $\VEC{a}\in\mathcal{V}_h^b$, we set
\begin{align*}
    \begin{split}
       \VEC{\Sigma}_{h,N,C,\rm dis}^{\VEC{a},k} &\coloneqq \{\VEC{\tau}_h\in\VEC{\Sigma}_h(\omega_{\VEC{a}}) \ :\ \VEC{\tau}_h \VEC{n}_{\omega_{\VEC{a}}} = \VEC{0} \ \text{on}\ \partial\omega_{\VEC{a}}\setminus\partial\Omega,\\
        &\hspace{1.8cm} \VEC{\tau}_h \VEC{n}_{\omega_{\VEC{a}}} = \Pi_{\VEC{\Sigma}_h \VEC{n}_{\omega_{\VEC{a}}}} \left(\psi_{\VEC{a}} \VEC{g}_N \right) \ \text{on}\ \partial\omega_{\VEC{a}}\cap\Gamma_N \ \text{and}\\
        &\hspace{3.6cm} \VEC{\tau}_h \VEC{n}_{\omega_{\VEC{a}}} = \Pi_{\VEC{\Sigma}_h \VEC{n}_{\omega_{\VEC{a}}}} \left(\psi_{\VEC{a}} \left[P_{1,\gamma}^n(\VEC{u}_h^k)\right]_{\mathbb{R}^-} \VEC{n}\right) \ \text{on}\ \partial\omega_{\VEC{a}}\cap\Gamma_C \},
    \end{split}
    \\
    \begin{split}
       \VEC{\Sigma}_{h,N,C,\rm reg}^{\VEC{a},k} &\coloneqq \{\VEC{\tau}_h\in\VEC{\Sigma}_h(\omega_{\VEC{a}}) \ :\ \VEC{\tau}_h \VEC{n}_{\omega_{\VEC{a}}} = \VEC{0} \ \text{on}\ \partial\omega_{\VEC{a}}\setminus\partial\Omega\ \text{and on}\ \partial\omega_{\VEC{a}}\cap\Gamma_N, \ \text{and}\\
        &\hspace{0.8cm} \VEC{\tau}_h \VEC{n}_{\omega_{\VEC{a}}} = \Pi_{\VEC{\Sigma}_h \VEC{n}_{\omega_{\VEC{a}}}} \left(\psi_{\VEC{a}} \left(\left[P_{1,\gamma}^n(\VEC{u}_h^k)\right]_{\text{reg},\delta} - \left[P_{1,\gamma}^n(\VEC{u}_h^k)\right]_{\mathbb{R}^-}\right) \VEC{n}\right) \ \text{on}\ \partial\omega_{\VEC{a}}\cap\Gamma_C \},
    \end{split}
    \\
    \begin{split}
       \VEC{\Sigma}_{h,N,C,\rm lin}^{\VEC{a},k} &\coloneqq \{\VEC{\tau}_h\in\VEC{\Sigma}_h(\omega_{\VEC{a}}) \ :\ \VEC{\tau}_h \VEC{n}_{\omega_{\VEC{a}}} = \VEC{0} \ \text{on}\ \partial\omega_{\VEC{a}}\setminus\partial\Omega\ \text{and on}\ \partial\omega_{\VEC{a}}\cap\Gamma_N, \ \text{and}\\
        &\hspace{1.3cm} \VEC{\tau}_h \VEC{n}_{\omega_{\VEC{a}}} = \Pi_{\VEC{\Sigma}_h \VEC{n}_{\omega_{\VEC{a}}}} \left(\psi_{\VEC{a}} \left(P_{\text{lin}}^{k-1}(\VEC{u}_h^k) - \left[P_{1,\gamma}^n(\VEC{u}_h^k)\right]_{\text{reg},\delta}\right) \VEC{n}\right) \ \text{on}\ \partial\omega_{\VEC{a}}\cap\Gamma_C \},
    \end{split}
\end{align*}
and, for any internal vertex $\VEC{a}\in\mathcal{V}_h^i$, $\VEC{\Sigma}_{h,N,C,\bullet}^{\VEC{a},k} \coloneqq \VEC{\Sigma}_h^{\VEC{a}}$ (see \eqref{eq:Sigmaha}) for $\bullet\in\{\rm dis,\rm reg,\rm lin\}$.
Moreover, let $\VEC{y}^k, \tilde{\VEC{y}}^k \in \VEC{RM}^d$ be such that, for all $\VEC{z}\in\VEC{RM}^d$,
\[
\begin{aligned}
  (\VEC{y}^k, \VEC{z})_{\omega_{\VEC{a}}}
  &= (-\psi_{\VEC{a}} \VEC{f} + \VEC{\sigma}(\VEC{u}_h^k)\VEC{\nabla}\psi_{\VEC{a}}, \VEC{z})_{\omega_{\VEC{a}}} - \left(\Pi_{\VEC{\Sigma}_h \VEC{n}_{\omega_{\VEC{a}}}} \left(\psi_{\VEC{a}} \VEC{g}_N \right),\VEC{z}\right)_{\partial\omega_{\VEC{a}}\cap \Gamma_N} 
  \\
  &\quad - \Bigl(\Pi_{\VEC{\Sigma}_h \VEC{n}_{\omega_{\VEC{a}}}} \left(\psi_{\VEC{a}} \left[P_{1,\gamma}^n(\VEC{u}_h^k)\right]_{\mathbb{R}^-} \VEC{n}\Bigr), \VEC{z}\right)_{\partial\omega_{\VEC{a}}\cap \Gamma_C},
  \\
  (\tilde{\VEC{y}}^k, \VEC{z})_{\omega_{\VEC{a}}}
  &= \left(\Pi_{\VEC{\Sigma}_h \VEC{n}_{\omega_{\VEC{a}}}} \left(\psi_{\VEC{a}} \left(\left[P_{1,\gamma}^n(\VEC{u}_h^k)\right]_{\mathbb{R}^-} - \left[P_{1,\gamma}^n(\VEC{u}_h^k)\right]_{\text{reg},\delta}\right) \VEC{n}\right), \VEC{z}\right)_{\partial\omega_{\VEC{a}}\cap\Gamma_C}
\end{aligned}
\]
if $\VEC{a} \in\mathcal{V}_h^b$, and $\VEC{y}^{k} = \tilde{\VEC{y}}^{k} = \VEC{0}$ if $\VEC{a}\in\mathcal{V}_h^i$.

\begin{construction}[Equilibrated stress reconstruction distinguishing the error components]\label{reconstruction2}
  Let, for $\bullet \in \{\rm dis, reg, lin\}$ and any vertex $\VEC{a}\in\mathcal{V}_h$, $(\VEC{\sigma}_{h,\bullet}^{\VEC{a},k},\VEC{r}_{h,\bullet}^{\VEC{a},k},\VEC{\lambda}_{h,\bullet}^{\VEC{a},k}) \in \VEC{\Sigma}_{h,N,C,\bullet}^{\VEC{a},k} \times \VEC{U}_h^{\VEC{a}} \times \VEC{\Lambda}_h^{\VEC{a}}$ be the solution to the following problem:
  \[
  \begin{alignedat}{4}
    (\VEC{\sigma}_{h,\bullet}^{\VEC{a},k},\VEC{\tau}_h)_{\omega_{\VEC{a}}} + (\VEC{r}_{h,\bullet}^{\VEC{a},k},\VEC{\rm div}\, %\VEC{\nabla}\cdot
    \VEC{\tau}_h)_{\omega_{\VEC{a}}} + (\VEC{\lambda}_{h,\bullet}^{\VEC{a},k},\VEC{\tau}_h)_{\omega_{\VEC{a}}} &= (\VEC{\tau}_{h,\bullet}^{\VEC{a},k},\VEC{\tau}_h)_{\omega_{\VEC{a}}}
    & \qquad & \forall\VEC{\tau}_h\in\VEC{\Sigma}_h^{\VEC{a}},
    \\ 
    (\VEC{\rm div}\, %\VEC{\nabla}\cdot
    \VEC{\sigma}_{h,\bullet}^{\VEC{a},k},\VEC{v}_h)_{\omega_{\VEC{a}}} &= (\VEC{v}_{h,\bullet}^{\VEC{a},k}, \VEC{v}_h)_{\omega_{\VEC{a}}}
    & \qquad & \forall\VEC{v}_h\in\VEC{U}_h^{\VEC{a}},
    \\ 
    (\VEC{\sigma}_{h,\bullet}^{\VEC{a},k},\VEC{\mu}_h)_{\omega_{\VEC{a}}} &= 0
    & \qquad & \forall\VEC{\mu}_h \in \VEC{\Lambda}_h^{\VEC{a}},
  \end{alignedat}
  \]
  where
  \begin{equation*}
    \VEC{\tau}_{h,\bullet}^{\VEC{a},k} \coloneqq \begin{cases}
      \psi_{\VEC{a}} \VEC{\sigma}(\VEC{u}_h^k) & \text{if $\bullet = \rm dis$}, \\
      0 &  \text{if $\bullet \in \{\rm reg, lin\}$},
    \end{cases}
    \qquad
    \VEC{v}_{h,\bullet}^{\VEC{a},k} \coloneqq \begin{cases}
      -\psi_{\VEC{a}} \VEC{f} + \VEC{\sigma}(\VEC{u}_h^k)\VEC{\nabla}\psi_{\VEC{a}} - \VEC{y}^k &  \text{if $\bullet = \rm dis$},\\
      -\tilde{\VEC{y}}^k & \text{if $\bullet = \rm reg$}, \\
      \VEC{y}^k + \tilde{\VEC{y}}^k & \text{if $\bullet = \rm lin$}.
    \end{cases}
  \end{equation*}
  Extending $\VEC{\sigma}_{h,\bullet}^{\VEC{a},k}$ by zero outside the patch $\omega_{\VEC{a}}$, we set $\VEC{\sigma}_{h,\bullet}^k\coloneqq\sum_{\VEC{a}\in\mathcal{V}_h} \VEC{\sigma}_{h,\bullet}^{\VEC{a},k}$, and we define $\VEC{\sigma}_h^k \coloneqq \VEC{\sigma}_{h,\rm dis}^k + \VEC{\sigma}_{h,\rm reg}^k + \VEC{\sigma}_{h,\rm lin}^k$.
\end{construction}

By definition, $\VEC{y}^k$ and $\tilde{\VEC{y}}^k$ ensure that the forcing terms $\VEC{v}_{h,\bullet}^{\VEC{a},k}$ satisfy the following Neumann compatibility conditions for $\VEC{a}\in\mathcal{V}_h^b\setminus \mathcal{V}_h^D$:
\begin{gather*}
    (\VEC{v}_{h,\rm dis}^{\VEC{a},k}, \VEC{z})_{\omega_{\VEC{a}}} = \left(\Pi_{\VEC{\Sigma}_h \VEC{n}_{\omega_{\VEC{a}}}} \left(\psi_{\VEC{a}} \VEC{g}_N \right),\VEC{z}\right)_{\partial\omega_{\VEC{a}}\cap \Gamma_N} + \Bigl(\Pi_{\VEC{\Sigma}_h \VEC{n}_{\omega_{\VEC{a}}}} \left(\psi_{\VEC{a}} \left[P_{1,\gamma}^n(\VEC{u}_h^k)\right]_{\mathbb{R}^-} \VEC{n}\Bigr), \VEC{z}\right)_{\partial\omega_{\VEC{a}}\cap \Gamma_C},
    \\
    (\VEC{v}_{h,\rm reg}^{\VEC{a},k}, \VEC{z})_{\omega_{\VEC{a}}} =  \left(\Pi_{\VEC{\Sigma}_h \VEC{n}_{\omega_{\VEC{a}}}} \left(\psi_{\VEC{a}} \left(\left[P_{1,\gamma}^n(\VEC{u}_h^k)\right]_{\text{reg},\delta} - \left[P_{1,\gamma}^n(\VEC{u}_h^k)\right]_{\mathbb{R}^-}\right) \VEC{n}\right), \VEC{z}\right)_{\partial\omega_{\VEC{a}}\cap\Gamma_C},
    \\
    (\VEC{v}_{h,\rm lin}^{\VEC{a},k}, \VEC{z})_{\omega_{\VEC{a}}} = \left(\Pi_{\VEC{\Sigma}_h \VEC{n}_{\omega_{\VEC{a}}}} \left(\psi_{\VEC{a}} \left(P_{\text{lin}}^{k-1}(\VEC{u}_h^k) - \left[P_{1,\gamma}^n(\VEC{u}_h^k)\right]_{\text{reg},\delta}\right) \VEC{n}\right), \VEC{z}\right)_{\partial\omega_{\VEC{a}}\cap\Gamma_C}
\end{gather*}
for any $\VEC{z}\in\VEC{RM}^d$.
The obtained tensor $\VEC{\sigma}_h^k$ is an equilibrated stress reconstruction in the sense of Definition \ref{definition - equilibrated stress reconstruction}, and in particular it satisfies the properties stated by the following lemma whose proof is similar to that of Lemma \ref{lemma - properties of reconstruction1} and is therefore omitted for the sake of conciseness.

\begin{lemma}[Properties of $\VEC{\sigma}_h^k$]\label{lemma - properties of reconstruction2}
Let $\VEC{\sigma}_h^k$ be defined by Construction \ref{reconstruction2}. Then
\begin{enumerate}
	\item $\VEC{\sigma}_{h,\rm dis}^k, \VEC{\sigma}_{h,\rm reg}^k, \VEC{\sigma}_{h,\rm lin}^k, \VEC{\sigma}_h^k\in\MAT{H}(\emph{\textbf{div}},\Omega)$;
	%\item $\VEC{\nabla}\cdot\VEC{\sigma}_h^k |_{T} = - \VEC{\Pi}_{T}^{p-1} \VEC{f}$ for every $T\in \mathcal{T}_h$;
	\item For every $T\in\mathcal{T}_h$ and every $\VEC{v}_T\in \VEC{\mathcal{P}}^{p-1}(T)$, $(\VEC{\rm div}\, %\VEC{\nabla}\cdot
	\VEC{\sigma}_h^k+ \VEC{f},\VEC{v}_T)_T=0$;
    %\item $(\VEC{\sigma}_h^k\VEC{n})|_{F} = \VEC{\Pi}_{F}^p \VEC{g}_N$ for every $F\in\mathcal{F}_h^N$;
    \item For every $F\in\mathcal{F}_h^N$ and every $\VEC{v}_F\in\VEC{\mathcal{P}}^p(F)$, $(\VEC{\sigma}_h^k\VEC{n},\VEC{v}_F)_F=(\VEC{g}_N,\VEC{v}_F)_F$;
    %\item $(\VEC{\sigma}_{h,1}^k\VEC{n})|_{F} = \bigl(\Pi_{F}^p \left[P_{1,\gamma}^n(\VEC{u}_h^k)\right]_{\mathbb{R}^-} \bigr) \VEC{n}$, $(\VEC{\sigma}_{h,2}^k\VEC{n})|_{F} = \bigl(\Pi_{F}^p \bigl(\left[P_{1,\gamma}^n(\VEC{u}_h^k)\right]_{\emph{reg},\delta} - \left[P_{1,\gamma}^n(\VEC{u}_h^k)\right]_{\mathbb{R}^-}\bigr) \bigr) \VEC{n}$, and $(\VEC{\sigma}_{h,3}^k\VEC{n})|_{F} = \bigl(\Pi_{F}^p \bigl(P_{\emph{lin}}^{k-1}(\VEC{u}_h^k) - \left[P_{1,\gamma}^n(\VEC{u}_h^k)\right]_{\emph{reg},\delta}\bigr) \bigr) \VEC{n}$ for every $F\in\mathcal{F}_h^C$;  
    \item For every $F\in\mathcal{F}_h^C$ and every $\VEC{v}_F\in\VEC{\mathcal{P}}^p(F)$, 
    \[
    (\VEC{\sigma}_{h,\rm dis}^k \VEC{n},\VEC{v}_F)_F=\left(\left[P_{1,\gamma}^n(\VEC{u}_h^k)\right]_{\mathbb{R}^-} \VEC{n},\VEC{v}_F\right)_F, %= \left(\left[P_{1,\gamma}^n(\VEC{u}_h)\right]_{\mathbb{R}^-}, v_F^n\right)_F
    \] 
    \[
    (\VEC{\sigma}_{h,\rm reg}^k \VEC{n},\VEC{v}_F)_F=\left(\left(\left[P_{1,\gamma}^n(\VEC{u}_h^k)\right]_{\rm reg,\delta} - \left[P_{1,\gamma}^n(\VEC{u}_h^k)\right]_{\mathbb{R}^-} \right) \VEC{n},\VEC{v}_F\right)_F,
    \]
    and 
    \[
    (\VEC{\sigma}_{h,\rm lin}^k \VEC{n},\VEC{v}_F)_F = \break \left(\left(P_{\rm lin}^{k-1}(\VEC{u}_h^k) - \left[P_{1,\gamma}^n(\VEC{u}_h^k)\right]_{\rm reg,\delta} \right) \VEC{n},\VEC{v}_F\right)_F.
    \] 
    %\item [5)] $\VEC{\sigma}_{h,reg}^k \rightarrow 0$ as $\delta \rightarrow 0$, and $\VEC{\sigma}_{h,lin}^k \rightarrow 0$ as $k\rightarrow +\infty$.
\end{enumerate}
\end{lemma}

\begin{remark}[Validity of Property 4. in Definition \ref{definition - equilibrated stress reconstruction} and of Assumption \ref{assumption on stress reconstruction}]\label{remark_no tangential part for sigma_h}
  The fourth property of the previous lemma implies that $(\VEC{\sigma}_{h,\bullet}^k\VEC{n})|_F$ has the same direction as the normal vector $\VEC{n}$, and, as a consequence, $\VEC{\sigma}_{h,\bullet}^{k,\VEC{t}} = \VEC{0}$ on $F\in \mathcal{F}_h^C$, for $\bullet\in\{\rm dis, reg, lin\}$. 
  Moreover, by definition, $\VEC{\sigma}_h^k$ is the sum of three tensors representing discretization, regularization and linearization, respectively.
  Therefore, $\VEC{\sigma}_h^k$ is an equilibrated stress reconstruction in the sense of Definition \ref{definition - equilibrated stress reconstruction} that additionally satisfies Assumption \ref{assumption on stress reconstruction}.
\end{remark}

\begin{remark}[Alternative expressions of local estimators]
Thanks to Lemma \ref{lemma - properties of reconstruction2}, we can rewrite the oscillation \eqref{oscillation error estimator}, Neumann \eqref{Neumann error estimator}, and contact \eqref{discretization error estimator} estimators as follows:
\begin{gather*}
    \eta_{\emph{osc},T}^k = \frac{h_T}{\pi} \left\lVert \VEC{f}-\VEC{\Pi}_T^{p-1} \VEC{f} \right\rVert_T,\qquad
    \eta_{\emph{Neu},T}^k = \sum_{F\in \mathcal{F}_T^C} C_{t,T,F} h_F^{\nicefrac{1}{2}} \left\lVert \VEC{g}_N-\VEC{\Pi}_F^p \VEC{g}_N\right\rVert_F, \\
    \eta_{\emph{cnt},T}^k = \sum_{F\in\mathcal{F}_T^C} h_F^{\nicefrac{1}{2}} \left\lVert \left[P_{1,\gamma}^n(\VEC{u}_h^k)\right]_{\mathbb{R}^-}
    - \Pi_F^p \left[P_{1,\gamma}^n(\VEC{u}_h^k)\right]_{\mathbb{R}^-} \right\rVert_F, 
\end{gather*}
where $\VEC{\Pi}_T^{p-1}$, $\VEC{\Pi}_F^p$, and $\Pi_F^p$ denote the $L^2$-orthogonal projectors on the polynomial spaces $\VEC{\mathcal{P}}^{p-1}(T)$, $\VEC{\mathcal{P}}^p(F)$, and $\mathcal{P}^p(F)$, respectively.
\end{remark}

\section{Efficiency of local estimators}\label{sec:efficiency}

\begin{figure}[tb]
	\centering
	\begin{subfigure}{0.45\textwidth}
		\centering
		\begin{tikzpicture}[scale=0.5]
		\fill[fill=gray!20] (-3,-4) -- (-5,0) -- (-4,3) -- (-1,4) -- (3,3) -- (5,0) -- (1,-3) -- (-3,-4);
		\fill[fill=violet!20] (-2,-1) -- (-1,2) -- (2,-1);
		\draw[thick] (-3,-4) -- (-5,0) -- (-4,3) -- (-1,4) -- (3,3) -- (5,0) -- (1,-3) -- (-3,-4);
		\draw[thick] (-3,-4) -- (-2,-1) -- (-5,0) -- (-1,2) -- (-4,3);
		\draw[thick] (-1,4) -- (-1,2) -- (3,3) -- (2,-1) -- (5,0);
		\draw[thick] (1,-3) -- (2,-1) -- (-1,2) -- (-2,-1) -- (2,-1) -- (1,-3) -- (-2,-1);
		\node at (-0.5,0) {\textcolor{violet}{$T$}};
		\end{tikzpicture}
	\end{subfigure}
	\hfill
	\begin{subfigure}{0.45\textwidth}
		\centering
		\begin{tikzpicture}[scale=0.6]
		\fill[fill=gray!20] (5,0) -- (-5,0) -- (-4,2) -- (-2,4) -- (2,5) -- (4,2) -- (5,0);
		\fill[fill=violet!20] (-2,0) -- (-0.5,3) -- (2,0) -- (-2,0);
		\draw[thick] (5,0) -- (-5,0) -- (-4,2) -- (-2,4) -- (2,5) -- (4,2) -- (5,0);
		\draw[thick] (-2,0) -- (-4,2) -- (-0.5,3) -- (-2,0);
		\draw[thick] (-2,4) -- (-0.5,3) -- (2,0) -- (4,2) -- (-0.5,3) -- (2,5);
		\node at (-0.25,1) {\textcolor{violet}{$T$}};
		\end{tikzpicture}
	\end{subfigure}
	\caption{Illustration of $\tilde{\omega}_T$ for $T\in\mathcal{T}_h$ such that $\mathcal{F}_T^b = \emptyset$ (\emph{left}) and that $\mathcal{F}_T^b \neq \emptyset$ (\emph{right}).}
	\label{fig:illustration tilde omega T}
\end{figure}
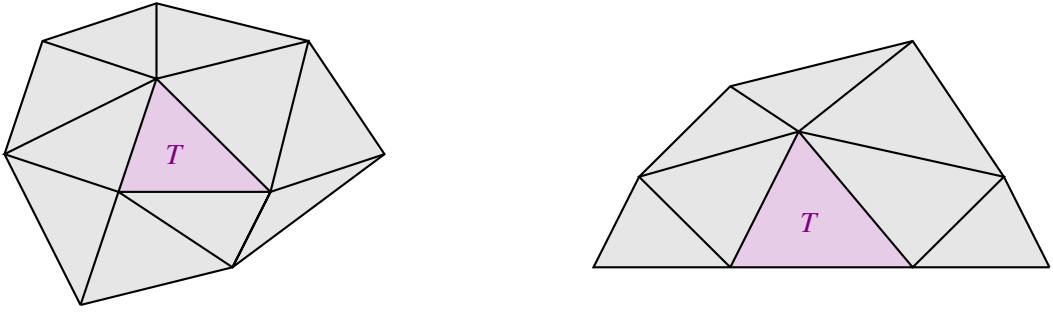

In this section we briefly discuss the local efficiency of the estimators defined by \eqref{local estimators for u_h^k} when using the stress reconstruction described in Section \ref{sec - reconstruction - error decomposition}.

Following \cite{Verfurth1999}, for any $T\in\mathcal{T}_h$ we denote by $\tilde{\omega}_T$ the union of all elements sharing at least one vertex with $T$ (see Figure \ref{fig:illustration tilde omega T}) and by $\mathcal{T}_T$ the corresponding set of elements.
Moreover, as in Subsection \ref{sec:comparison norms}, $a\lesssim b$ stands for $a\leq C b$, where $C>0$ is a constant which is independent of the mesh size $h$ and of the Nitsche parameter $\gamma_0$.
We introduce, for all $T\in\mathcal{T}_h$, the \emph{local residual} defined as follows:
For all $\VEC{w}_h\in\VEC{V}_h$ and all $\VEC{v}\in\VEC{H}_D^1(\tilde{\omega}_T)$
\begin{equation*}
\left\langle\mathcal{R}_{\mathcal{T}_T}(\VEC{w}_h), \VEC{v}\right\rangle_{\tilde{\omega}_T} \coloneqq (\VEC{f},\VEC{v})_{\tilde{\omega}_T} + (\VEC{g}_N,\VEC{v})_{\partial\tilde{\omega}_T\cap\Gamma_N} - \bigl(\VEC{\sigma}(\VEC{w}_h),\VEC{\varepsilon}(\VEC{v})\bigr)_{\tilde{\omega}_T} + \left(\left[P_{1,\gamma}^n(\VEC{w}_h)\right]_{\mathbb{R}^-},v^n\right)_{\partial\tilde{\omega}_T\cap \Gamma_C},
\end{equation*}
%08-04%for any $\VEC{w}_h\in\VEC{V}_h$ and for any $\VEC{v}\in \VEC{H}^1_D(\tilde{\omega}_T)$, 
where
\begin{equation*}
\VEC{H}^1_D(\tilde{\omega}_T) \coloneqq \left\{\VEC{v}\in \VEC{H}^1(\tilde{\omega}_T)\ :\ \VEC{v}=\VEC{0}\ \text{on}\  \partial\tilde{\omega}_T\cap\Gamma_D \ \text{and on}\ \partial\tilde{\omega}_T\cap\Omega\right\}.
\end{equation*}
%08-04%Then, the local dual norm of a function $\VEC{w}_h\in\VEC{V}_h$ with respect to the norm
Letting
\begin{equation*}
\norm{\VEC{v}}_{\tilde{\omega}_T}\coloneqq \left(\left\lVert\VEC{\nabla}\VEC{v}\right\rVert_{\tilde{\omega}_T}^2 + \left\lvert\VEC{v}\right\rvert^2_{C,\tilde{\omega}_T} \right)^{\nicefrac{1}{2}} = \biggl(\left\lVert\VEC{\nabla}\VEC{v}\right\rVert_{\tilde{\omega}_T}^2 + \sum_{F\in\mathcal{F}_{\mathcal{T}_T}^C} \frac{1}{h_F} \lVert\VEC{v}\rVert_F^2 \biggr)^{\nicefrac{1}{2}},
\end{equation*}
with $\mathcal{F}_{\mathcal{T}_T}^C$ denoting the (possibly empty) set of faces of $\mathcal{T}_T$ that lie on $\Gamma_C$,
the corresponding dual norm of the local residual for a function $\VEC{w}_h\in\VEC{V}_h$ is
\begin{equation}\label{eq:dual norm of the local residual}
	\begin{split}
		\localdualnormresidual{\VEC{w}_h} = \sup_{\substack{\VEC{v}\in \VEC{H}^1_D(\tilde{\omega}_T),\ \norm{\VEC{v}}_{\tilde{\omega}_T} = 1}} \left\langle \mathcal{R}_{\mathcal{T}_T}(\VEC{w}_h), \VEC{v}\right\rangle_{\tilde{\omega}_T} .
	\end{split}
\end{equation}

\begin{theorem}[Local efficiency]\label{theo:local efficiency}
	Assume $d=2$.
	Let $\VEC{u}_h^k\in\VEC{V}_h$ and let $\VEC{\sigma}_h^{k}$ be the stress reconstruction of Construction \ref{reconstruction2}, and assume that the local stopping criteria \eqref{eq:local stopping criteria} are used in Lines \ref{alg:global stopping criterion lin} and \ref{alg:global stopping criterion reg} of Algorithm \ref{algorithm}, respectively.
	Then, for every element $T\in\mathcal{T}_h$, it holds
	\begin{multline}\label{thesis_lemma local efficiency}
	\eta_{\emph{osc},T}^k + \eta_{\emph{str},T}^k + \eta_{\emph{Neu},T}^k + \eta_{\emph{cnt},T}^k + \eta_{\emph{lin},T}^k + \eta_{\emph{reg},T}^k 
	\\ 
	\lesssim
	\localdualnormresidual{\VEC{u}_h^k}
	+ \eta_{\emph{osc},\mathcal{T}_T}^k + \eta_{\emph{Neu},\mathcal{T}_T}^k %13-05%+ \eta_{\emph{cnt},\mathcal{T}_T}^k
	+ \eta_{\emph{cnt},\mathcal{T}_T}^k,
	\end{multline}
	where
	\[
	\eta_{\bullet,\mathcal{T}_T}^k \coloneqq \left(\sum_{T'\in\mathcal{T}_T} \left(\eta_{\bullet,T'}^k\right)^2\right)^{\nicefrac{1}{2}} \qquad \text{with}\ \bullet\in\{ \emph{osc}, \emph{Neu}, \emph{cnt}\}.
	\]
\end{theorem}

\begin{remark}[Restriction on the space dimension]
  The proof of Lemma \ref{lemma-local stress estimator} below requires the introduction of a space with suitable properties that are known only for in dimension $d=2$ (see the definition of the space $\VEC{M}_h^{a}$ in \cite[Section 4.4]{Botti2018}).
  This assumption reverberates in Theorem \ref{theo:local efficiency}, whose proof uses Lemma \ref{lemma-local stress estimator}.
\end{remark}

\begin{proof}[Proof of Theorem \ref{theo:local efficiency}]
  The proof hinges on classical arguments, so we only briefly outline the main ideas and refer to \cite[Section 5]{Fontana2022(these)} for further details.
  
  We introduce, for any element $T\in\mathcal{T}_h$, a local residual-based estimator $\eta_{\sharp,T}^k$ on the patch $\tilde{\omega}_T$, following the path of \cite{Verfurth1999}, and then compare it with the local estimators \eqref{local estimators for u_h^k} and the dual norm of the local residual operator \eqref{eq:dual norm of the local residual}.
  In particular, it is possible to prove the following two lemmas by adapting the approach of \cite[Appendix A]{ElAlaoui2011} and \cite[Subsection 4.4]{Botti2018}, respectively:

  \begin{lemma}[Control of the residual-based estimator $\eta_{\sharp,T}$]\label{lemma-local sharp estimator}
	  For any element $T\in\mathcal{T}_h$,
	  \begin{equation}\label{thesis_lemma local sharp_estimator}
	    \eta_{\sharp,T}^k \lesssim
	    \localdualnormresidual{\VEC{u}_h^k}
	    + \eta_{\emph{osc},\mathcal{T}_T}^k + \eta_{\emph{Neu},\mathcal{T}_T}^k + \eta_{\emph{cnt},\mathcal{T}_T}^k,
	  \end{equation}
  \end{lemma}

  \begin{lemma}[Control of the local stress estimator]\label{lemma-local stress estimator}
	  Assume $d=2$.
	  Then, for every element $T\in\mathcal{T}_h$,
	  \[
	  \eta_{\emph{str},T}^k \lesssim \eta_{\sharp,T}^k.
	  \]
  \end{lemma}

	\noindent The estimate \eqref{thesis_lemma local efficiency} follows using the local stopping criteria along with Lemmas \ref{lemma-local sharp estimator} and \ref{lemma-local stress estimator}.
\end{proof}

\begin{remark}[Global efficiency]\label{rem:global efficiency}
	%Let $k\geq 1$.
	The results of this subsection can be easily extended in order to prove the global efficiency. 
	%Indeed, adapting the proof of Lemma \ref{lemma-local sharp estimator} as shown in \cite[Section 5]{Fontana2022(these)}, it is possible to show that
	%\[
	%\eta_{\sharp}^k \lesssim \dualnormresidual{\VEC{u}_h^k} + \eta_{\emph{osc}}^k + \eta_{\emph{Neu}}^k + \eta_{\emph{cnt}}^k.
	%\]
	%Moreover, Lemma \ref{lemma-local stress estimator} implies
	%\[
	%\eta_{\emph{str}}^k \lesssim \eta_{\sharp}^k.
	%\]
	%Finally, assuming that the global criteria shown in Line \ref{alg:global stopping criterion lin} and \ref{alg:global stopping criterion reg} of Algorithm \ref{algorithm} % \eqref{global stopping criterion1} and \eqref{global stopping criterion2} 
	%hold, and combining these results,
	%we achieve
	Indeed, combining the global criteria shown in Line \ref{alg:global stopping criterion lin} and \ref{alg:global stopping criterion reg} of Algorithm \ref{algorithm} and the global counterpart of Lemma \ref{lemma-local sharp estimator} and \ref{lemma-local stress estimator}, we achieve
	\[
	\eta_{\emph{osc}}^k + \eta_{\emph{str}}^k + \eta_{\emph{Neu}}^k + \eta_{\emph{cnt}}^k + \eta_{\emph{lin}}^k + \eta_{\emph{reg}}^k \lesssim  \dualnormresidual{\VEC{u}_h^k} + \eta_{\emph{osc}}^k + \eta_{\emph{Neu}}^k + \eta_{\emph{cnt}}^k.   
	\]
\end{remark}

\section{Numerical results}\label{sec:Numerical results}

We present numerical cases to validate the a posteriori error estimate of Theorem \ref{theorem - a posteriori estimation u_h^k} and show its use in the framework of an adaptive algorithm.
The simulations are performed with the open source finite element library FreeFem++ (see \cite{FreeFEM} and also \url{https://freefem.org/}). 
We will use the notion of local and global total estimator: $\eta_{\rm tot,T}$ defined by \eqref{eq:local total estimator} and 
\[
\eta_{\text{tot}} \coloneqq \left(\sum_{T\in\mathcal{T}_h} \left(\eta_{\text{tot},T}\right)^2\right)^{\nicefrac{1}{2}}.
\]
For the sake of brevity, above and throughout this section we omit the superscript $k$ which identifies the step of the Newton method.

\begin{figure}[tb]
    \centering
    \begin{tikzpicture}[scale=0.9]
    \fill[gray!20] (-4,0) -- (4,0) -- (4,4) -- (-4,4) -- (-4,0);
        \draw[thick] (-4,0) -- (4,0) -- (4,4) -- (-4,4) -- (-4,0);
        \draw[very thick, cadmiumgreen] (-4,0) -- (0,0);
        \foreach \i in {0, ..., 20}
            {\draw[thick, cadmiumgreen] (-0.2*\i,0) -- (-0.2*\i-0.3,-0.3);}
        \fill[gray] (0,0) rectangle (4.05, -0.2);
        \draw[very thick, blue] (0,0) -- (4,0);
        \draw[very thick, red] (4,0) -- (4,4);
        \draw[very thick, orange] (4,4) -- (-4,4) -- (-4,0);
        \node at (-2,0.4) {\textcolor{cadmiumgreen}{$\Gamma_D$}};
        \node at (2,0.4) {\textcolor{blue}{$\Gamma_C$}};
        \node at (3.4,2) {\textcolor{red}{$\Gamma_{N,1}$}};
        \node at (-3.2,3.5) {\textcolor{orange}{$\Gamma_{N,2}$}};
        \draw[arrow, thick, red] (5,3.5) -- (4,3.5);
        \draw[arrow, thick, red] (5,2.5) -- (4,2.5);
        \draw[arrow, thick, red] (5,1.5) -- (4,1.5);
        \draw[arrow, thick, red] (5,0.5) -- (4,0.5);
        \node at (4.5,2) {\textcolor{red}{$\VEC{g}_N$}};
        \draw[arrow, thick] (0,2.5) -- (0,1.5);
        \draw[arrow, thick] (1.7,2.5) -- (1.7,1.5);
        \draw[arrow, thick] (-1.7,2.5) -- (-1.7,1.5);
        \node at (-0.85,2) {$\VEC{f}$};
    \end{tikzpicture}
    \caption{Illustration of the rectangular domain with the subdivision of the boundary as $\partial\Omega = \Gamma_D \dot{\cup} (\Gamma_{N,1} \dot{\cup} \Gamma_{N,2}) \dot{\cup} \Gamma_C$.}
    \label{fig:domain illustration}
\end{figure}
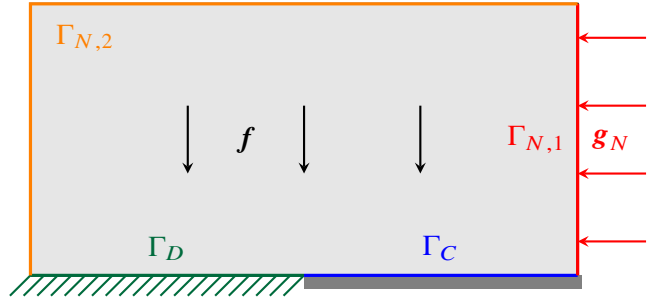

%\subsection{First example: rectangular domain}

%\begin{figure}[tb]
%	\centering
%	\begin{subfigure}{0.49\textwidth}
%		\centering
%		\includegraphics[width=0.95\textwidth]{Figures/5P-0-MeshCUT.eps}
%	\end{subfigure}
%	\hfill
%	\begin{subfigure}{0.49\textwidth}
%		\centering
%		\includegraphics[width=0.9\textwidth]{Figures/5P-3-MeshCUT.eps}
%	\end{subfigure}
%	\begin{subfigure}{0.49\textwidth}
%	    \vspace{0.6cm}
%		\centering
%		\includegraphics[width=0.9\textwidth]{Figures/5P-6-MeshCUT.eps}
%	\end{subfigure}
%	\hfill
%	\begin{subfigure}{0.49\textwidth}
%	    \vspace{0.6cm}
%		\centering
%		\includegraphics[width=0.9\textwidth]{Figures/5P-9-MeshCUT.eps}
%	\end{subfigure}
%	\caption{Rectangular domain: initial mesh and adaptively refined mesh after 3, 6 and 9 steps, respectively}
%	\label{fig:Rrefinement mesh}
%\end{figure}

\begin{figure}[tb]
	\centering
	\begin{subfigure}{0.55\textwidth}
		\centering
		\includegraphics[width=\textwidth]{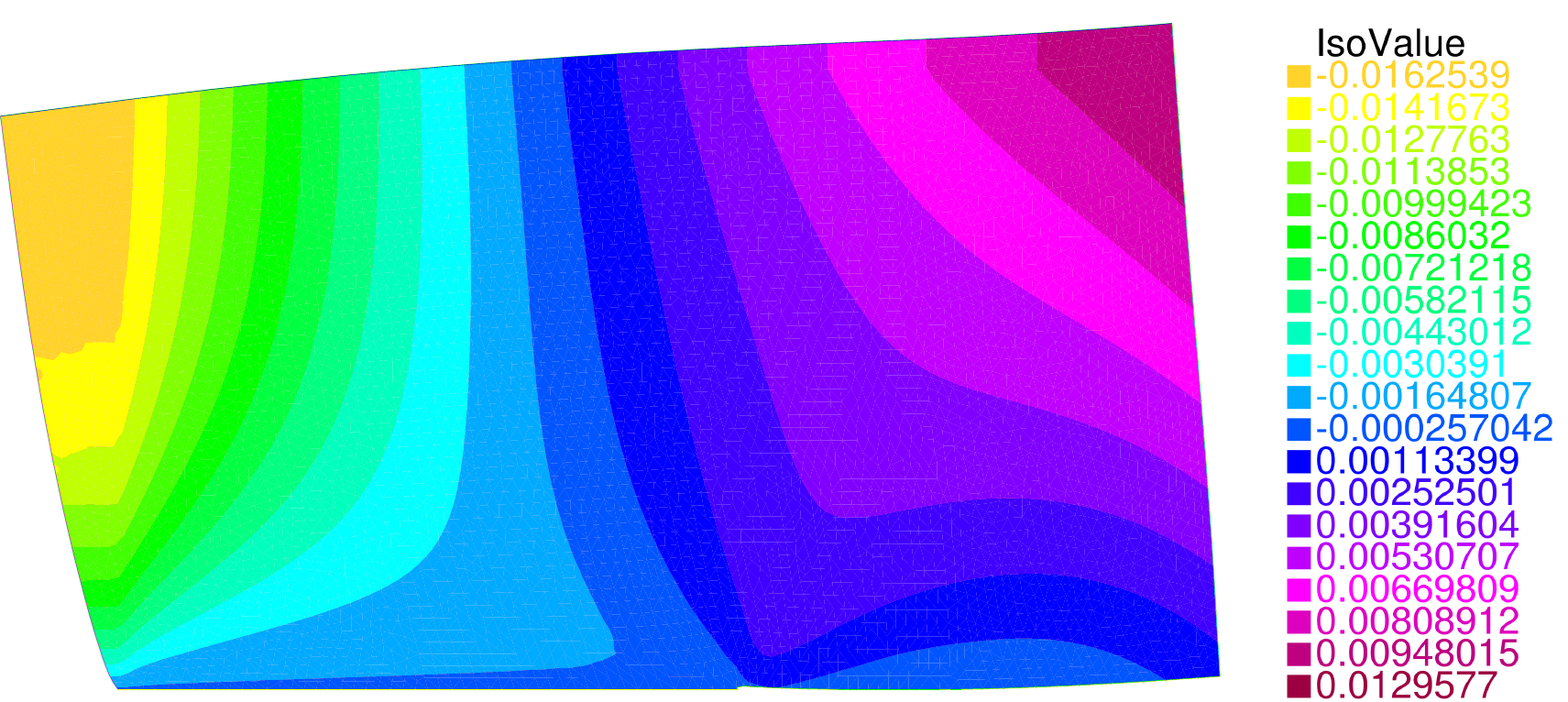}
		\subcaption{Vertical displacement in the deformed domain (amplification factor = 5).}
		\label{fig:Rvertical displacement}
	\end{subfigure}
	\hfill
	\begin{subfigure}{0.43\textwidth}
		\centering
		\resizebox{!}{0.7\textwidth}{%
			% This file was created by tikzplotlib v0.9.8.
\begin{tikzpicture}

\definecolor{color0}{rgb}{0.12156862745098,0.466666666666667,0.705882352941177}

\begin{axis}[
	legend cell align={left},
	legend style={
		at={(0.02,0.98)},
		anchor=north west, 
		fill opacity=0.8, 
		draw opacity=1, 
		text opacity=1, 
		draw=white!80!black
	},
	tick align=outside,
	tick pos=left,
	log basis y={10},
	x grid style={white!69.0196078431373!black},
	xmin=-0.05, xmax=1.05,
	xtick style={color=black},
	xtick={-0.2,0,0.2,0.4,0.6,0.8,1,1.2},
	xticklabels={−0.2,0.0,0.2,0.4,0.6,0.8,1.0,1.2},
	y grid style={white!69.0196078431373!black},
	ymin=-0.000200131261341054, ymax=0.00420275648816214,
	ytick style={color=black},
	ytick={-0.0005,0,0.0005,0.001,0.0015,0.002,0.0025,0.003,0.0035,0.004,0.0045},
	]
	
	\addplot [very thick, black]
	table {%
	0 0
	1 0
	};
	\addlegendentry{$\Gamma_C$ - reference conf.}

	\addplot [very thick, blue]
	table {%
	-1.99519844845245e-33 3.71378605350851e-34
	0.00278059421413092 0.00110896584105521
	0.00621123589041646 0.00112309534476257
	0.010006815536309 0.00120130966471094
	0.0140005610139415 0.00126366006024478
	0.0181254377307633 0.00126815176616846
	0.0223266457867111 0.00129336808063148
	0.0266229430833989 0.00125198272190745
	0.0309707568399264 0.00123153536985193
	0.0353618721742913 0.00119398361962227
	0.0397815897415389 0.00116479280911696
	0.0442379177429719 0.00113159894220507
	0.0487179515328029 0.00109834126402157
	0.0532197375522399 0.0010632626830777
	0.057741385853094 0.00102749982957055
	0.0622802522052812 0.000992296168249009
	0.0668347994093466 0.000956822547046837
	0.071403714530923 0.000921543364212297
	0.0759856384639183 0.000887015241721168
	0.0805768919587802 0.000852876095775886
	0.0851819776709905 0.000818280733673651
	0.08979679663835 0.000784723471156279
	0.0944186251157051 0.000751594867825561
	0.0990510587657051 0.000718712138680902
	0.10369125263894 0.000686604056917656
	0.108338770740015 0.00065529589724914
	0.112993328613533 0.000624521970285336
	0.117654767268675 0.000594580807525322
	0.122321185685559 0.00056520570591845
	0.126995140809005 0.000536028403512522
	0.131674907359598 0.000508086916651652
	0.136358857776833 0.000480797757439546
	0.141047935957538 0.000453839364999149
	0.145741715272982 0.000427353125906906
	0.150441172929303 0.000401963893919815
	0.155144798313627 0.000377169525513781
	0.159852600746688 0.000353340265049269
	0.164563480604422 0.000329811385949835
	0.169279273624811 0.00030703893438001
	0.173998746280393 0.000285003260148653
	0.178721676667155 0.000263678666212195
	0.183447961649506 0.000243047397857082
	0.188177655164017 0.000223181672881399
	0.192910651484657 0.000204029238714898
	0.197646934685915 0.000185830036426136
	0.202385552491325 0.000168123502367607
	0.207127013075528 0.000150883938380434
	0.21187189549084 0.000134666144134377
	0.216619424028027 0.000119192918658625
	0.221369725475798 0.000104789874119925
	0.226122008414433 9.08556694001054e-05
	0.230876695020063 7.75383939487537e-05
	0.235634287046026 6.54172787156754e-05
	0.240394002550086 5.40028426191507e-05
	0.245155890700203 4.33449660897775e-05
	0.249919938308947 3.3569866877476e-05
	0.254686223608091 2.4789899267601e-05
	0.259454631844834 1.72463706566126e-05
	0.264224937453405 1.07821965602535e-05
	0.268997084861833 5.43512721773658e-06
	0.273771267666661 1.9865910128381e-06
	0.278546450077299 4.28239900330553e-07
	0.283323782550944 4.49489385819131e-07
	0.288102245845773 4.03168053889558e-07
	0.292881980168591 3.86910985722588e-07
	0.297663173323816 3.69320591318956e-07
	0.302445855343551 3.59164069688708e-07
	0.3072299467409 3.47069843387209e-07
	0.312015418669421 3.40668085479197e-07
	0.316802301961632 3.31578466637246e-07
	0.321590612898995 3.27168160568312e-07
	0.326380353022174 3.18962704976839e-07
	0.331171519799369 3.13743350103219e-07
	0.335963720741708 3.0555906947632e-07
	0.340757231648599 3.05649961038546e-07
	0.345552380221943 3.08323515866292e-07
	0.350348758641798 3.04891438108729e-07
	0.355146428699277 3.02801051041896e-07
	0.359945081152761 2.97082298060445e-07
	0.364744956391728 2.98547414836667e-07
	0.369546333743421 3.04155795972948e-07
	0.374348906090053 3.0344597159693e-07
	0.379152428404394 3.00696140367864e-07
	0.383957106503341 3.0422086286452e-07
	0.388763181694168 3.11313286941976e-07
	0.393570371938103 3.11549945917394e-07
	0.398378728209187 3.17269034329933e-07
	0.403188041677972 3.16845497708851e-07
	0.407998485226585 3.22160496094761e-07
	0.412810243685286 3.33063309939473e-07
	0.417623105812271 3.42547744885276e-07
	0.422436938734795 3.46849839853071e-07
	0.427251899068989 3.53154626625068e-07
	0.43206812327784 3.75653027570176e-07
	0.436885503374373 3.85358689237867e-07
	0.441704041488417 4.27049453307499e-07
	0.446523817684977 4.10715641653631e-07
	0.451344858959317 1.06472877159175e-06
	0.456167193460848 4.50709533487719e-06
	0.46099061028706 9.63524697239382e-06
	0.465814717346219 1.53832809752851e-05
	0.470639906158784 2.30581969780272e-05
	0.47546578051116 3.16245832936636e-05
	0.480292440026602 4.10917871604858e-05
	0.48511986113777 5.14875895865788e-05
	0.489948022491504 6.27695133686767e-05
	0.494776917574435 7.48883508977296e-05
	0.499606535259156 8.78178647921239e-05
	0.504436857966394 0.000101551466997078
	0.509267878371651 0.000116054731063179
	0.5140995816543 0.000131321682579994
	0.518931955020359 0.000147346568535673
	0.523764981402601 0.000164111752907023
	0.528598654704444 0.000181604202288964
	0.53343295977182 0.000199827800308988
	0.538267884588331 0.000218764219189042
	0.54310341694955 0.000238414023507454
	0.547939544035524 0.000258772704286125
	0.552776252619412 0.000279832256998609
	0.557613531753524 0.000301587767482977
	0.562451369323735 0.000324037942575316
	0.567289751185342 0.00034717508936054
	0.572128668962109 0.00037099653090028
	0.576968110531422 0.000395496890207024
	0.58180806568082 0.000420678026795183
	0.586648520804679 0.000446531051813568
	0.591489466974183 0.000473062378051144
	0.59633089391353 0.000500257265362684
	0.601172785249347 0.000528130224428413
	0.606015137486194 0.000556660776197057
	0.610857931818447 0.000585858822521296
	0.615701162003923 0.000615722085111101
	0.620544814577655 0.000646240554877725
	0.625388876455181 0.000677421219244744
	0.630233345120842 0.000709244002332831
	0.635078196071903 0.000741731640371058
	0.639923431506151 0.000774860192152095
	0.644769039282343 0.000808631290588368
	0.649615002808641 0.0008430556194704
	0.654461314480397 0.000878117469244868
	0.659307966489935 0.000913816616644434
	0.664154942926657 0.000950162130429089
	0.669002233017154 0.000987133436359033
	0.673849830533656 0.00102474202203518
	0.678697724879892 0.00106297241187296
	0.6835459004824 0.00110183408268074
	0.688394350283801 0.00114131350416701
	0.693243063801832 0.00118141200282002
	0.698092033390732 0.00122212695970489
	0.702941249268456 0.0012634525053651
	0.707790688849768 0.00130539849523266
	0.712640351108036 0.00134794022639026
	0.71749021851745 0.00139108676451752
	0.722340290796972 0.00143482068216426
	0.727190552325645 0.00147914426536304
	0.732040997545456 0.00152405311328408
	0.736891607806871 0.00156954945436744
	0.741742375906836 0.00161562207908923
	0.746593294108077 0.00166226253744513
	0.751444345469725 0.00170946639715571
	0.756295525169189 0.0017572301123403
	0.761146823625294 0.00180554240046469
	0.765998220448988 0.00185440522044795
	0.770849717922047 0.001903795193835
	0.7757012926859 0.00195372218531184
	0.780552948916521 0.00200416257567565
	0.785404667127432 0.00205511749591456
	0.790256440778212 0.00210658422578975
	0.795108258673522 0.00215854622337377
	0.799960100900491 0.00221099775835979
	0.804811975108288 0.00226391798054676
	0.809663857447122 0.00231730606894292
	0.814515747467006 0.00237114777362976
	0.81936763450676 0.00242543361897568
	0.824219500238373 0.00248015758408316
	0.829071343766867 0.00253529288174499
	0.833923143146503 0.00259083559066952
	0.838774907571979 0.00264676140342439
	0.843626624465319 0.00270306373586665
	0.848478287783932 0.00275973438043404
	0.853329872889164 0.00281676315213094
	0.858181387185889 0.00287411461740315
	0.863032825918143 0.00293178516766806
	0.867884169650405 0.00298975864654093
	0.872735422966755 0.00304801809996611
	0.877586563389484 0.00310652700150993
	0.882437609887021 0.00316527554417376
	0.887288555859473 0.00322426002029949
	0.892139389437216 0.00328343923315159
	0.896990118688101 0.00334280571177502
	0.901840732524943 0.00340233630183997
	0.906691241964413 0.00346200517190772
	0.911541638711735 0.00352179134649817
	0.916391941846088 0.00358168065300777
	0.921242163235172 0.00364165644816552
	0.92609229319911 0.00370169602084194
	0.930942365093459 0.00376178601635621
	0.935792390153236 0.0038219323455869
	0.940642362369791 0.00388210089457759
	0.945492344725246 0.0039423048205862
	0.950342313943348 0.00400262522682109
	};
	\addlegendentry{$\Gamma_C$ - deformed conf.}
\end{axis}

\end{tikzpicture}
		}
		\subcaption{Displacement of the contact boundary.}
		\label{fig:Rdisplacement contact}
	\end{subfigure}
	\caption{Vertical displacement in the deformed configuration (\emph{left}), and representation of the contact boundary part $\Gamma_C$ in the reference (\emph{black}) and deformed (\emph{blue}) configuration (\emph{right}).}
	\label{fig:Rdisplacement}
\end{figure}
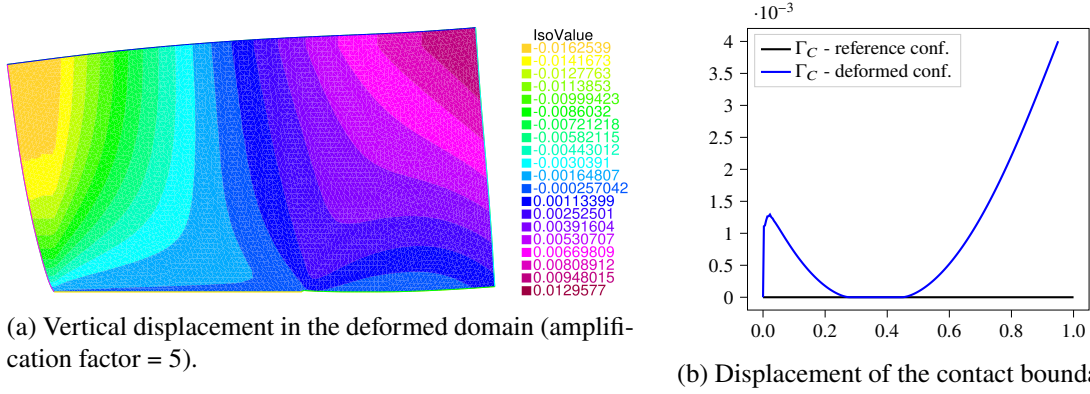

\begin{figure}[tb]
	\centering
	\begin{subfigure}{0.49\textwidth}
		\centering
		\includegraphics[width=0.95\textwidth]{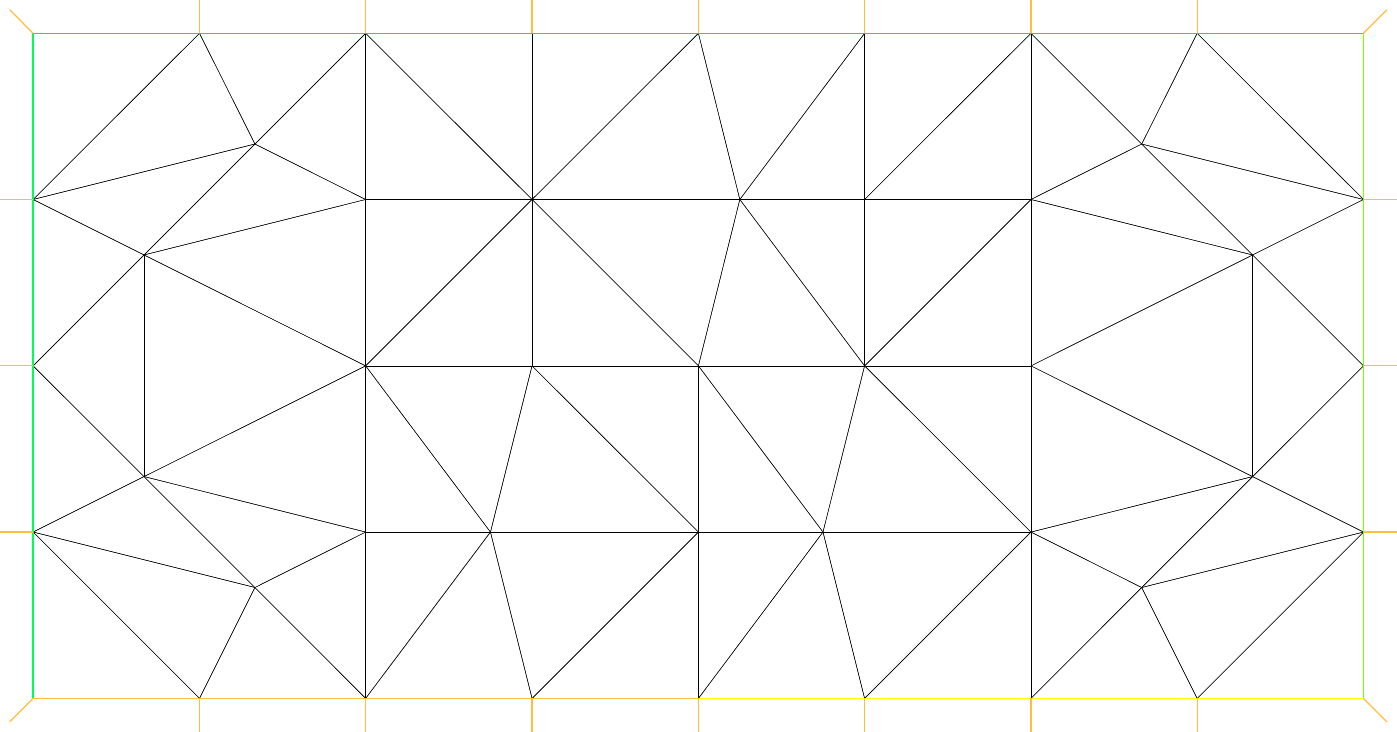}
	\end{subfigure}
	\hfill
	\begin{subfigure}{0.49\textwidth}
		\centering
		\includegraphics[width=0.9\textwidth]{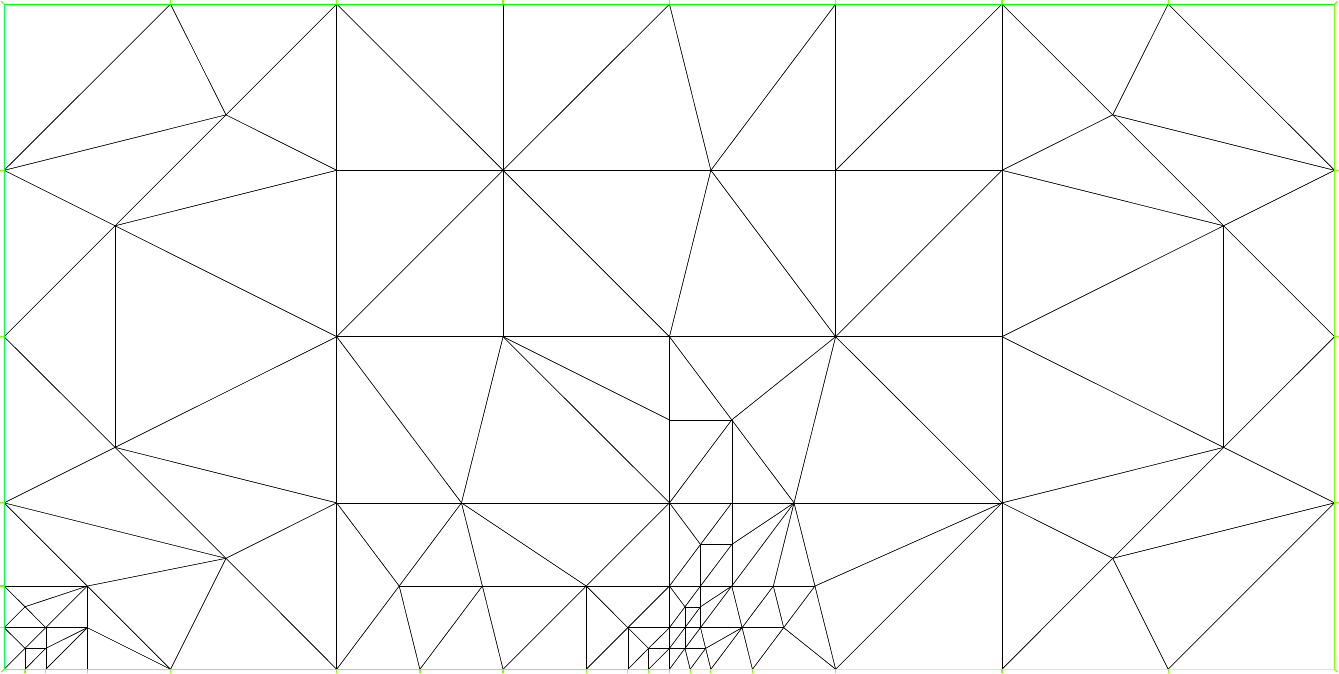}
	\end{subfigure}
	\begin{subfigure}{0.49\textwidth}
	    \vspace{0.6cm}
		\centering
		\includegraphics[width=0.9\textwidth]{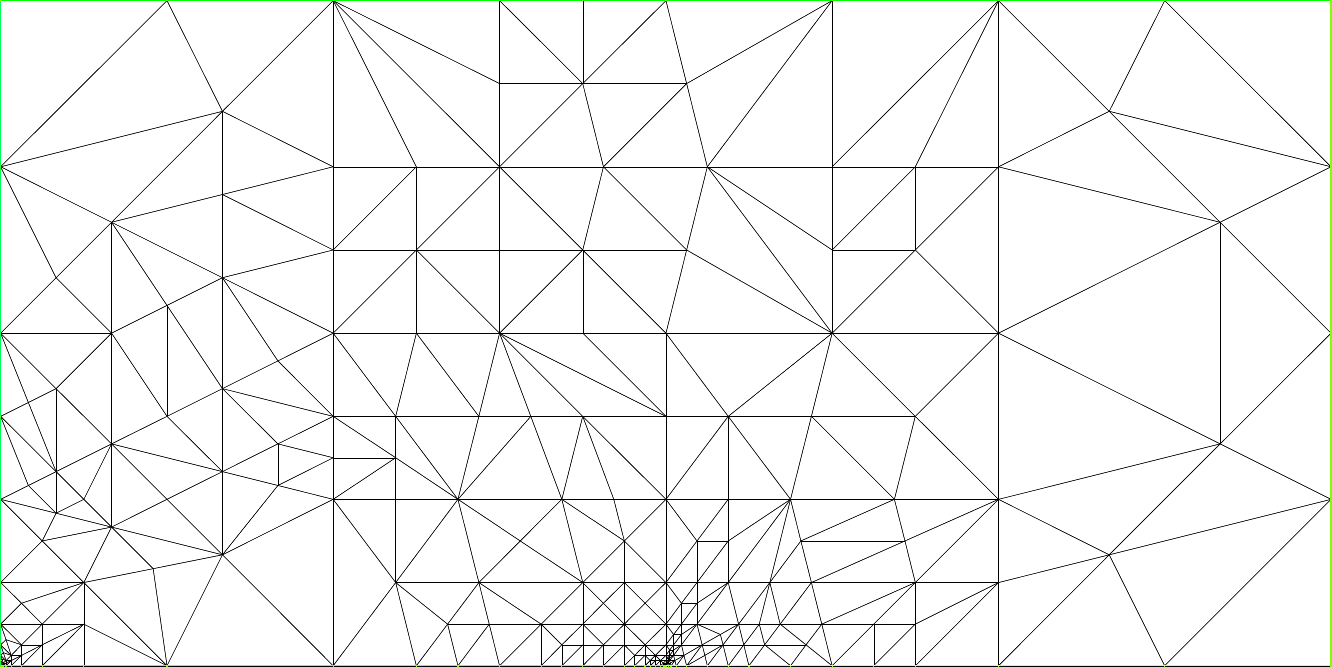}
	\end{subfigure}
	\hfill
	\begin{subfigure}{0.49\textwidth}
	    \vspace{0.6cm}
		\centering
		\includegraphics[width=0.9\textwidth]{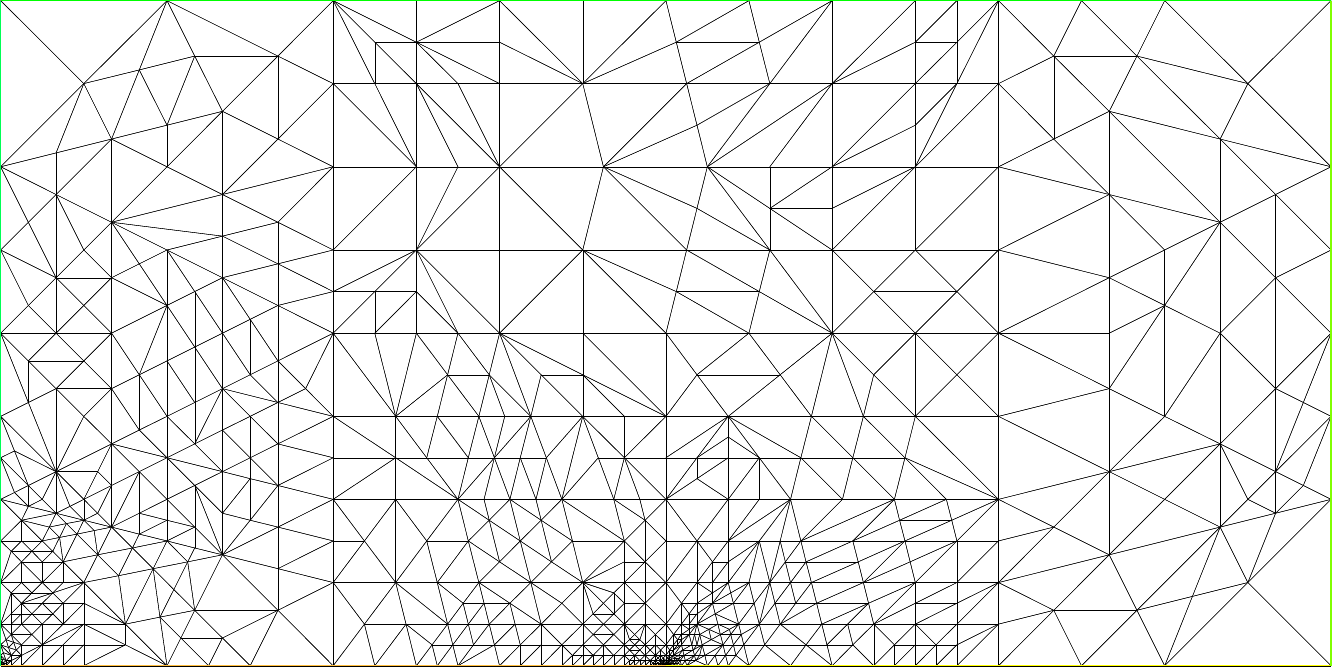}
	\end{subfigure}
	\caption{Initial mesh and adaptively refined mesh after 3, 7 and 11 steps, respectively.}
	\label{fig:Rrefinement mesh}
\end{figure}

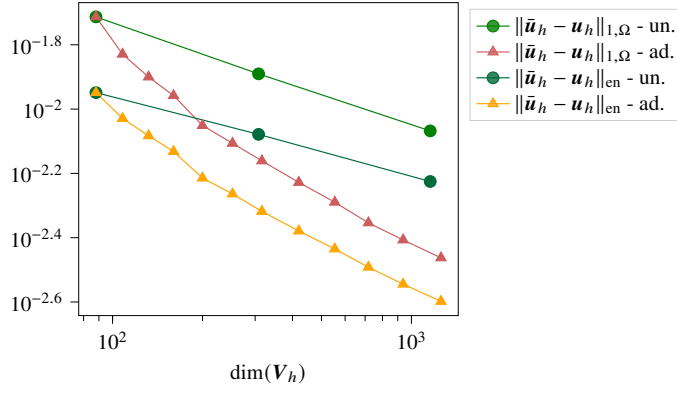
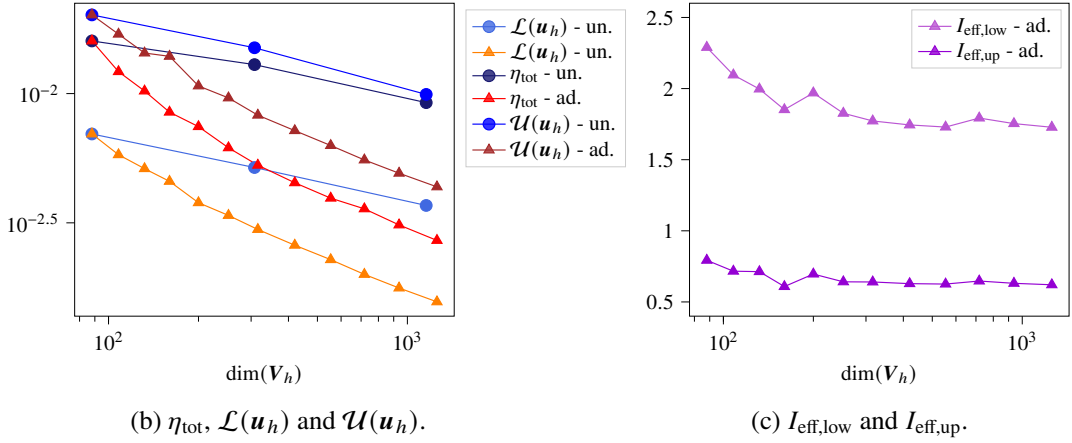
\begin{figure}[tb]
	\centering
	%\hspace{3.cm}
	\begin{subfigure}{0.49\textwidth}
		\centering
		%\hspace*{-10em}
		\resizebox{!}{0.7\textwidth}{%
		\begin{tikzpicture}
		
		\begin{axis}[
		legend cell align={left},
		legend style={
			%at={(-0.25,0.8)},
			%anchor=north east,
			at={(1.03,0.98)},
			anchor=north west, 
			fill opacity=0.8, 
			draw opacity=1, 
			text opacity=1, 
			draw=white!80!black
		},
		log basis x={10},
		log basis y={10},
		tick align=outside,
		tick pos=left,
		x grid style={white!69.0196078431373!black},
		xlabel={\(\displaystyle \text{dim}(\boldsymbol{V}_h)\)},
		xmin=77.0472821927514, xmax=1434.5476810394,
		xmode=log,
		xtick style={color=black},
		y grid style={white!69.0196078431373!black},
		ymin=0.00227878551662044, ymax=0.0214355833530327,
		ymode=log,
		ytick style={color=black}
		]
		\addplot [semithick, green!50.1960784313725!black, mark=*, mark size=3, mark options={solid}]
		table {%
			88 0.019359244481963
			308 0.0128723175295772
			1156 0.00855529847433709
		};
		\addlegendentry{$\lVert\bar{\boldsymbol{u}}_h-\boldsymbol{u}_h\rVert_{1,\Omega}$ - un.}
		
		\addplot [semithick, colorH1ad, mark=triangle*, mark size=3, mark options={solid}]
		table {%
			88 0.019359244481963
			108 0.0148218508073203
			132 0.0125857823025589
			160 0.0110171266140377
			200 0.00889089513146665
			252 0.0078286767927217
			316 0.00690276614764
			420 0.00591583176112279
			554 0.00513383153274524
			718 0.00442854942680068
			938 0.00391803725469944
			1256 0.00344613155593794
		};
		\addlegendentry{$\lVert\bar{\boldsymbol{u}}_h-\boldsymbol{u}_h\rVert_{1,\Omega}$ - ad.}
		
		\addplot [semithick, cadmiumgreen, mark=*, mark size=3, mark options={solid}]
		table {%
			88 0.0112537911755484
			308 0.00834561387631831
			1156 0.00595428400645755
		};
		\addlegendentry{$\lVert\bar{\boldsymbol{u}}_h-\boldsymbol{u}_h\rVert_{\text{en}}$ - un.}
		
		\addplot [semithick, colorEnergyad, mark=triangle*, mark size=3, mark options={solid}]
		table {%
			88 0.0112537911755484
			108 0.00935508360260589
			132 0.00826241656473109
			160 0.00737987753445346
			200 0.00610173958245025
			252 0.00544787751192368
			316 0.00480660440521382
			420 0.00417738888115692
			554 0.00367332603766472
			718 0.00322220883283672
			938 0.0028539747090177
			1256 0.00252319231417898
		};
		\addlegendentry{$\lVert\bar{\boldsymbol{u}}_h-\boldsymbol{u}_h\rVert_{\text{en}}$ - ad.}
		
%		\addplot[draw=none, semithick, colorTEun, mark=*, mark size=3, mark options={solid}] coordinates {(1,1)};
%		\addlegendentry{$\eta_{\text{tot}}$ - un.}
%		
%		\addplot[draw=none, semithick, red, mark=triangle*, mark size=3, mark options={solid}] coordinates {(1,1)};
%		\addlegendentry{$\eta_{\text{tot}}$ - ad.}
%		
%		\addplot [draw=none, semithick, blue, mark=*, mark size=3, mark options={solid}] coordinates {(1,1)};
%		%\addlegendentry{$\mathcal{U}(\boldsymbol{u}_h)$ - un.}
%		\addlegendentry{Upper bound - un.}
%		
%		\addplot [draw=none, semithick, colorUpad, mark=triangle*, mark size=3, mark options={solid}] coordinates {(1,1)};
%		%\addlegendentry{$\mathcal{U}(\boldsymbol{u}_h)$ - ad.}
%		\addlegendentry{Upper bound - ad.}
		\end{axis}
		
		\end{tikzpicture}
		}
		\subcaption{$\lVert\bar{\VEC{u}}_h-\VEC{u}_h\rVert_{1,\Omega}$ and $\energynorm{\bar{\VEC{u}}_h-\VEC{u}_h}$.}
		\label{fig:Rcomparison H1+energy}
	\end{subfigure}
	\vspace{4mm}
	\\
	
	\begin{subfigure}{0.49\textwidth}
		\centering
		\resizebox{!}{0.7\textwidth}{%
		% This file was created by tikzplotlib v0.9.8.
\begin{tikzpicture}

\begin{axis}[
legend cell align={left},
legend style={
	fill opacity=0.8,
	draw opacity=1,
	text opacity=1,
	at={(1.03,0.98)},
	anchor=north west,
	%at={(-0.15,0.98)},
	%anchor=north east,
	draw=white!80!black
},
log basis x={10},
log basis y={10},
tick align=outside,
tick pos=left,
x grid style={white!69.0196078431373!black},
xlabel={\(\displaystyle \text{dim}(\boldsymbol{V}_h)\)},
xmin=77.0472821927514, xmax=1434.5476810394,
xmode=log,
xtick style={color=black},
y grid style={white!69.0196078431373!black},
ymin=0.00137711002361681, ymax=0.0223606888797198,
ymode=log,
ytick style={color=black}
]

\addplot [semithick, colorEnergyun, mark=*, mark size=3, mark options={solid}]
table {%
88 0.00697930500790762
308 0.00517573001066578
1156 0.00369269018204865
};
\addlegendentry{$\mathcal{L}(\boldsymbol{u}_h)$ - un.}

\addplot [semithick, orange, mark=triangle*, mark size=3, mark options={solid}]
table {%
88 0.00697930500790762
108 0.00580177655854539
132 0.0051241332283595
160 0.00457680575643399
200 0.00378413824820842
252 0.00337863020632986
316 0.00298092950838008
420 0.00259070660575104
554 0.00227810010070686
718 0.00199832908685953
938 0.00176995997778662
1256 0.00156481744503361
};
\addlegendentry{$\mathcal{L}(\boldsymbol{u}_h)$ - un.}

\addplot [semithick, colorTEun, mark=*, mark size=3, mark options={solid}]
table {%
	88 0.0159802096388928
	308 0.0129442208742431
	1156 0.00922529215439778
};
\addlegendentry{$\eta_{\text{tot}}$ - un.}

\addplot [semithick, red, mark=triangle*, mark size=3, mark options={solid}]
table {%
	88 0.0159802096388928
	108 0.012158816235978
	132 0.0102319979750605
	160 0.00847563972675256
	200 0.00745206408905243
	252 0.00617097207369811
	316 0.00528333052929641
	420 0.00451924068893605
	554 0.00393989926925823
	718 0.00358270428959455
	938 0.0031040379697508
	1256 0.00270403730758916
};
\addlegendentry{$\eta_{\text{tot}}$ - ad.}

\addplot [semithick, blue, mark=*, mark size=3, mark options={solid}]
table {%
	88 0.0201541493722794
	308 0.0150402171961936
	1156 0.00992554537364106
};
\addlegendentry{$\mathcal{U}(\boldsymbol{u}_h)$ - un.}

\addplot [semithick, colorUpad, mark=triangle*, mark size=3, mark options={solid}]
table {%
	88 0.0201541493722794
	108 0.0169791270625966
	132 0.0143517129608077
	160 0.0139388121549992
	200 0.0107062539975773
	252 0.00962144466227862
	316 0.0082511614615848
	420 0.00718970483511617
	554 0.00629723841078549
	718 0.00553746909608647
	938 0.00491995537558808
	1256 0.00435727161955293
};
\addlegendentry{$\mathcal{U}(\boldsymbol{u}_h)$ - ad.}
\end{axis}

\end{tikzpicture}
		}
		\subcaption{$\eta_{\rm tot}$, $\mathcal{L}(\VEC{u}_h)$ and $\mathcal{U}(\VEC{u}_h)$.}
		\label{fig:Rcomparison energy+upper}
	\end{subfigure}
	\hfill
	\begin{subfigure}{0.49\textwidth}
		\centering
		\resizebox{!}{0.7\textwidth}{%
		% This file was created by tikzplotlib v0.9.8.
\begin{tikzpicture}

\begin{axis}[
legend cell align={left},
legend style={fill opacity=0.8, draw opacity=1, text opacity=1, draw=white!80!black},
log basis x={10},
tick align=outside,
tick pos=left,
x grid style={white!69.0196078431373!black},
xlabel={\(\displaystyle \text{dim}(\boldsymbol{V}_h)\)},
xmin=77.0472821927514, xmax=1434.5476810394,
xmode=log,
xtick style={color=black},
y grid style={white!69.0196078431373!black},
ymin=0.4, ymax=2.60, %ymax=1.5981688404304,
ytick style={color=black}
]

%\addplot [semithick, colorEffEnergy, mark=*, mark size=3, mark options={solid}]
%table {%
%88 2.28965629397012
%308 2.50094592406646
%1156 2.49825782819389
%};
%\addlegendentry{$I_{\text{eff,low}}$ - un.}

\addplot [semithick, colorEffEnergy, mark=triangle*, mark size=3, mark options={solid}]
table {%
88 2.28965629397012
108 2.09570570553418
132 1.99682512516099
160 1.85186791351974
200 1.96928959785773
252 1.8264715866616
316 1.77237687588511
420 1.744404665084
554 1.72946714151662
718 1.79284999310346
938 1.75373342262376
1256 1.72802093699249
};
\addlegendentry{$I_{\text{eff,low}}$ - ad.}

%\addplot [semithick, colorEffUp, mark=*, mark size=3, mark options={solid}]
%table {%
%88 0.792899235969368
%308 0.860640554946172
%1156 0.929449396191073
%};
%\addlegendentry{$I_{\text{eff,up}}$ - un.}

\addplot [semithick, colorEffUp, mark=triangle*, mark size=3, mark options={solid}]
table {%
88 0.792899235969368
108 0.716103730842721
132 0.712946113331734
160 0.60806040231432
200 0.696047757762776
252 0.641376871177332
316 0.640313555115141
420 0.628571101676252
554 0.625655090731554
718 0.64699309872927
938 0.630907748706922
1256 0.620580386922631
};
\addlegendentry{$I_{\text{eff,up}}$ - ad.}

\end{axis}

\end{tikzpicture}	
		}
		\subcaption{$I_{\text{eff,low}}$ and $I_{\text{eff,up}}$.}
		\label{fig:Rcomparison effectivity}
	\end{subfigure}
	\caption{Comparison between the uniform case (circles) and the adaptive one (triangle) for the $H^1$-norm $\lVert\bar{\VEC{u}}_h-\VEC{u}_h\rVert_{1,\Omega}$ and the energy norm $\energynorm{\bar{\VEC{u}}_h-\VEC{u}_h}$ (\emph{top}), and for the global total estimator $\eta_{\text{tot}}$, $\mathcal{L}(\VEC{u}_h)$ and $\mathcal{U}(\VEC{u}_h)$ (\emph{bottom-left}). Corresponding effectivity indices $I_{\text{eff,low}}$ and $I_{\text{eff,up}}$ (\emph{bottom-right}).}
	\label{fig:Rcomparison uniform-adaptive energy+upper}
\end{figure}

%As first domain w
We consider a body that, in its reference configuration, occupies the rectangular domain $\Omega = (-1,1)\times (0,1)$ (see Figure \ref{fig:domain illustration}), with mechanical parameters $E = 1$ and $\nu = 0.3$, corresponding to Lamé coefficients $\mu \approx 0.385$ and $\lambda \approx 0.577$. The body is subjected to a weight force $\VEC{f} = (0,-0.01)$. Homogeneous Dirichlet boundary conditions are enforced on $\Gamma_D = (-1,0)\times \{0\}$, and the the body in its undeformed configuration is in contact with a rigid horizontal interface on $\Gamma_C = (0,1) \times \{0\}$. The Nitsche parameter is $\gamma_0 = 100 E$, whereas the regularization parameter which defines the operator $[\,\cdot\,]_{\text{reg},\delta}$ is $\delta = E/100$. A pressure $\VEC{g}_N = (-0.0275, 0)$ acts on the right side of the body $\Gamma_{N, 1} = \{1\} \times (0, 1)$, and the rest of the boundary is free, i.e., $\VEC{g}_N = \VEC{0}$ on $\Gamma_{N, 2} = \{-1\} \times (0, 1) \cup (-1,1) \times \{1\}$.
Since a closed-form solution is not available for this configuration, we take as reference solution the function $\bar{\VEC{u}}_h$ computed solving \eqref{Nitsche-based_method theta=0} with Lagrange $\mathcal{P}^2$ finite elements on a fine mesh ($h \approx 0.0084$). To compute the approximated solution $\VEC{u}_h$, we use Lagrange $\mathcal{P}^1$ elements (while this choice is known to lock in the quasi-incompressible limit, it is admissible for the set of parameters considered here and is compatible with the use of the lowest-order mixed finite elements available in FreeFem++ to compute the equilibrated stress reconstructions). 
In the deformed configuration, the body is in contact with the rigid foundation in a non-empty interval $I_C \subset \Gamma_C$ which is approximately %$(0.265, 0.552)$.
$(0.279, 0.447)$. Figure \ref{fig:Rvertical displacement} shows the vertical displacement in the deformed domain with an amplification factor equal to 5.
Moreover, in Figure \ref{fig:Rdisplacement contact}, which display the contact boundary part $\Gamma_C$ in the reference domain (black) and in the deformed domain (blue), we can easily identify the contact interval $I_C$.

We refine adaptively the initial mesh following the distribution of the local total estimator $\eta_{\text{tot},T}$, refining only the elements in which the value of this estimator is higher. In particular, at each refinement step, the 6\% elements with larger estimated error are refined, i.e., are subdivided into four sub-triangles dividing each edge by two. 
Figure \ref{fig:Rrefinement mesh} shows the initial mesh and the result of adaptive refinement after 3, 7, and 11 steps, respectively. We remark that the refinement is concentrated on the endpoints of $\Gamma_D$ (singularities due to the homogeneous Dirichlet conditions) and near the contact interval $I_C$.
Figure \ref{fig:Rcomparison H1+energy} compares the convergence on uniformly and adaptively refined meshes for the $H^1$-norm $\left\lVert\bar{\VEC{u}}_h-\VEC{u}_h\right\rVert_{1,\Omega}$ and for the energy norm $\energynorm{\bar{\VEC{u}}_h-\VEC{u}_h}$, showing the corresponding curves as functions of the number of degrees of freedom. In particular, the uniform refinement is performed dividing all triangles of the mesh into four sub-triangles.
The adaptive approach provides better convergence rates, i.e., adopting it we can achieve a fixed level of precision with fewer degrees of freedom. The asymptotic convergence rates are approximately 0.309 and 0.255 in the uniform case, and 0.450 and 0.449 in the adaptive case, for the $H^1$-norm and energy norm, respectively (the optimal convergence rate for smooth solutions is $0.5$). 

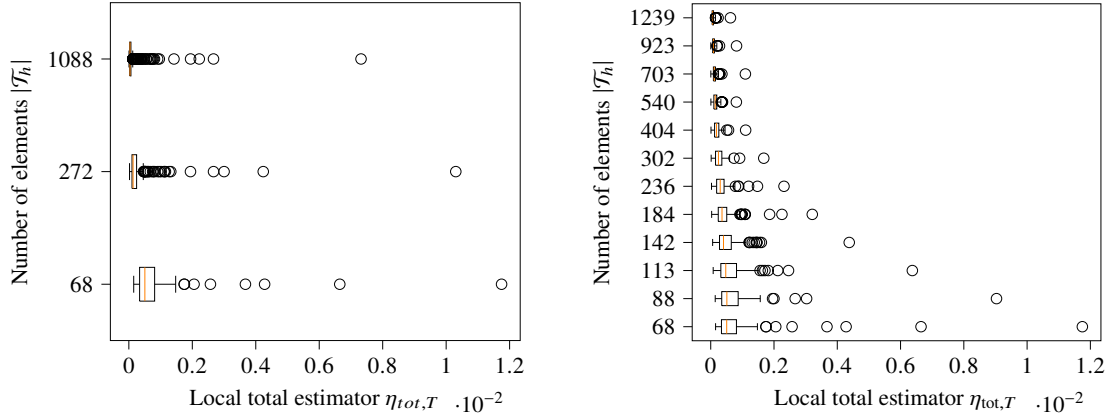
\begin{figure}[tb]
	\centering
	\begin{subfigure}{0.49\textwidth}
		\centering
		\resizebox{!}{0.75\textwidth}{%
		\begin{tikzpicture}
		
		\begin{axis}[
		tick align=outside,
		tick pos=left,
		x grid style={white!69.0196078431373!black},
		xlabel={Local total estimator  \(\displaystyle \eta_{tot,T}\)},
		xmin=-0.000582048838, xmax=0.012333783278,
		xtick style={color=black},
		y grid style={white!69.0196078431373!black},
		ylabel={Number of elements  \(\displaystyle |\mathcal{T}_h|\)},
		ymin=0.5, ymax=3.5,
		ytick style={color=black},
		ytick={1,2,3},
		yticklabels={68,272,1088}
		]
		\addplot [black]
		table {%
			0.0003391955 0.85
			0.0003391955 1.15
			0.000812641 1.15
			0.000812641 0.85
			0.0003391955 0.85
		};
		\addplot [black]
		table {%
			0.0003391955 1
			0.000153371 1
		};
		\addplot [black]
		table {%
			0.000812641 1
			0.00147556 1
		};
		\addplot [black]
		table {%
			0.000153371 0.925
			0.000153371 1.075
		};
		\addplot [black]
		table {%
			0.00147556 0.925
			0.00147556 1.075
		};
		\addplot [black, mark=o, mark size=2.5, mark options={solid,fill opacity=0}, only marks]
		table {%
			0.00257125 1
			0.00427881 1
			0.00367427 1
			0.0117467 1
			0.0017418 1
			0.00175141 1
			0.00205507 1
			0.00664734 1
		};
		\addplot [black]
		table {%
			9.9594525e-05 1.85
			9.9594525e-05 2.15
			0.00024469925 2.15
			0.00024469925 1.85
			9.9594525e-05 1.85
		};
		\addplot [black]
		table {%
			9.9594525e-05 2
			2.15639e-05 2
		};
		\addplot [black]
		table {%
			0.00024469925 2
			0.000454826 2
		};
		\addplot [black]
		table {%
			2.15639e-05 1.925
			2.15639e-05 2.075
		};
		\addplot [black]
		table {%
			0.000454826 1.925
			0.000454826 2.075
		};
		\addplot [black, mark=o, mark size=2.5, mark options={solid,fill opacity=0}, only marks]
		table {%
			0.00060749 2
			0.000548118 2
			0.000963014 2
			0.000501367 2
			0.00300154 2
			0.00046648 2
			0.000908374 2
			0.0008131 2
			0.00266886 2
			0.00127031 2
			0.00194236 2
			0.00112046 2
			0.0103036 2
			0.00101756 2
			0.00132162 2
			0.000776645 2
			0.000517528 2
			0.000494967 2
			0.000515152 2
			0.000581396 2
			0.000618537 2
			0.000491316 2
			0.000465734 2
			0.000678604 2
			0.00423561 2
			0.000749989 2
			0.00110929 2
			0.00115297 2
		};
		\addplot [black]
		table {%
			2.6977875e-05 2.85
			2.6977875e-05 3.15
			6.558055e-05 3.15
			6.558055e-05 2.85
			2.6977875e-05 2.85
		};
		\addplot [black]
		table {%
			2.6977875e-05 3
			5.03444e-06 3
		};
		\addplot [black]
		table {%
			6.558055e-05 3
			0.000122772 3
		};
		\addplot [black]
		table {%
			5.03444e-06 2.925
			5.03444e-06 3.075
		};
		\addplot [black]
		table {%
			0.000122772 2.925
			0.000122772 3.075
		};
		\addplot [black, mark=o, mark size=2.5, mark options={solid,fill opacity=0}, only marks]
		table {%
			0.000127388 3
			0.000129227 3
			0.000124711 3
			0.000179513 3
			0.000171223 3
			0.000144164 3
			0.000173388 3
			0.000231408 3
			0.000240262 3
			0.000157255 3
			0.000201869 3
			0.000132047 3
			0.000175125 3
			0.000138022 3
			0.000197916 3
			0.000190849 3
			0.000400237 3
			0.00221915 3
			0.000381378 3
			0.000722008 3
			0.000152383 3
			0.000196939 3
			0.000351078 3
			0.000194275 3
			0.000233155 3
			0.000203482 3
			0.000176436 3
			0.000630527 3
			0.00194576 3
			0.000907025 3
			0.00142202 3
			0.000274426 3
			0.000472057 3
			0.000266009 3
			0.00036985 3
			0.000695539 3
			0.000294452 3
			0.000387566 3
			0.000400879 3
			0.000372881 3
			0.000495271 3
			0.000311264 3
			0.000363275 3
			0.000807468 3
			0.00731735 3
			0.000747546 3
			0.000959877 3
			0.000125908 3
			0.000173068 3
			0.000138717 3
			0.000191523 3
			0.000243897 3
			0.00054643 3
			0.000397719 3
			0.000190209 3
			0.000270823 3
			0.000238201 3
			0.000238602 3
			0.000139413 3
			0.000148929 3
			0.000169586 3
			0.000173312 3
			0.000131528 3
			0.000162278 3
			0.00013273 3
			0.00014486 3
			0.000135165 3
			0.000128905 3
			0.000127305 3
			0.000161848 3
			0.000135394 3
			0.000131553 3
			0.000131151 3
			0.000163817 3
			0.000162932 3
			0.000214043 3
			0.000170775 3
			0.000135243 3
			0.000131967 3
			0.000132732 3
			0.000237462 3
			0.000253062 3
			0.000232692 3
			0.000301417 3
			0.000181079 3
			0.000169379 3
			0.000151935 3
			0.000189814 3
			0.000198019 3
			0.000143038 3
			0.000220991 3
			0.000204997 3
			0.000271833 3
			0.00023255 3
			0.000239584 3
			0.000147245 3
			0.000256328 3
			0.000173793 3
			0.000127049 3
			0.00266798 3
			0.000560782 3
			0.000681679 3
			0.000777998 3
			0.000346006 3
			0.000234216 3
			0.000266713 3
			0.00019064 3
			0.000232423 3
			0.000214535 3
			0.000175494 3
			0.000153645 3
			0.000249433 3
			0.000273005 3
			0.000307101 3
			0.000356009 3
		};
		\addplot [orange]
		table {%
			0.0005049065 0.85
			0.0005049065 1.15
		};
		\addplot [orange]
		table {%
			0.0001331415 1.85
			0.0001331415 2.15
		};
		\addplot [orange]
		table {%
			3.6273e-05 2.85
			3.6273e-05 3.15
		};
		\end{axis}
		
		\end{tikzpicture}
		}
	\end{subfigure}
	\hfill
	\begin{subfigure}{0.49\textwidth}
		\centering
		\resizebox{!}{0.75\textwidth}{%
		\begin{tikzpicture}
		
		\begin{axis}[
		tick align=outside,
		tick pos=left,
		x grid style={white!69.0196078431373!black},
		xlabel={Local total estimator  \(\displaystyle \eta_{\text{tot},T}\)},
		xmin=-0.0005854529695, xmax=0.0123339453795,
		xtick style={color=black},
		y grid style={white!69.0196078431373!black},
		ylabel={Number of elements  \(\displaystyle \lvert\mathcal{T}_h\rvert\)},
		ymin=0.5, ymax=12.5,
		ytick style={color=black},
		ytick={1,2,3,4,5,6,7,8,9,10,11,12},
		yticklabels={68,88,113,142,184,236,302,404,540,703,923,1239}
		]
		\addplot [black]
		table {%
			0.0003391955 0.75
			0.0003391955 1.25
			0.000812641 1.25
			0.000812641 0.75
			0.0003391955 0.75
		};
		\addplot [black]
		table {%
			0.0003391955 1
			0.000153371 1
		};
		\addplot [black]
		table {%
			0.000812641 1
			0.00147556 1
		};
		\addplot [black]
		table {%
			0.000153371 0.875
			0.000153371 1.125
		};
		\addplot [black]
		table {%
			0.00147556 0.875
			0.00147556 1.125
		};
		\addplot [black, mark=o, mark size=2.5, mark options={solid,fill opacity=0}, only marks]
		table {%
			0.00257125 1
			0.00427881 1
			0.00367427 1
			0.0117467 1
			0.0017418 1
			0.00175141 1
			0.00205507 1
			0.00664734 1
		};
		\addplot [black]
		table {%
			0.0003478185 1.75
			0.0003478185 2.25
			0.00086744675 2.25
			0.00086744675 1.75
			0.0003478185 1.75
		};
		\addplot [black]
		table {%
			0.0003478185 2
			0.000141809 2
		};
		\addplot [black]
		table {%
			0.00086744675 2
			0.0015653 2
		};
		\addplot [black]
		table {%
			0.000141809 1.875
			0.000141809 2.125
		};
		\addplot [black]
		table {%
			0.0015653 1.875
			0.0015653 2.125
		};
		\addplot [black, mark=o, mark size=2.5, mark options={solid,fill opacity=0}, only marks]
		table {%
			0.00199921 2
			0.00903311 2
			0.00266975 2
			0.00303075 2
			0.00195459 2
		};
		\addplot [black]
		table {%
			0.000319578 2.75
			0.000319578 3.25
			0.000815071 3.25
			0.000815071 2.75
			0.000319578 2.75
		};
		\addplot [black]
		table {%
			0.000319578 3
			7.74106e-05 3
		};
		\addplot [black]
		table {%
			0.000815071 3
			0.00152067 3
		};
		\addplot [black]
		table {%
			7.74106e-05 2.875
			7.74106e-05 3.125
		};
		\addplot [black]
		table {%
			0.00152067 2.875
			0.00152067 3.125
		};
		\addplot [black, mark=o, mark size=2.5, mark options={solid,fill opacity=0}, only marks]
		table {%
			0.00245873 3
			0.00163337 3
			0.00171372 3
			0.00156068 3
			0.006373 3
			0.00181922 3
			0.00211935 3
		};
		\addplot [black]
		table {%
			0.00027765325 3.75
			0.00027765325 4.25
			0.0006497465 4.25
			0.0006497465 3.75
			0.00027765325 3.75
		};
		\addplot [black]
		table {%
			0.00027765325 4
			6.02331e-05 4
		};
		\addplot [black]
		table {%
			0.0006497465 4
			0.00110944 4
		};
		\addplot [black]
		table {%
			6.02331e-05 3.875
			6.02331e-05 4.125
		};
		\addplot [black]
		table {%
			0.00110944 3.875
			0.00110944 4.125
		};
		\addplot [black, mark=o, mark size=2.5, mark options={solid,fill opacity=0}, only marks]
		table {%
			0.00122844 4
			0.0014598 4
			0.00123247 4
			0.00145665 4
			0.00137446 4
			0.0012103 4
			0.00143921 4
			0.0016003 4
			0.00437782 4
			0.00131554 4
			0.0015424 4
		};
		\addplot [black]
		table {%
			0.00023791125 4.75
			0.00023791125 5.25
			0.00050634575 5.25
			0.00050634575 4.75
			0.00023791125 4.75
		};
		\addplot [black]
		table {%
			0.00023791125 5
			2.94232e-05 5
		};
		\addplot [black]
		table {%
			0.00050634575 5
			0.000907914 5
		};
		\addplot [black]
		table {%
			2.94232e-05 4.875
			2.94232e-05 5.125
		};
		\addplot [black]
		table {%
			0.000907914 4.875
			0.000907914 5.125
		};
		\addplot [black, mark=o, mark size=2.5, mark options={solid,fill opacity=0}, only marks]
		table {%
			0.00225124 5
			0.00186072 5
			0.000948111 5
			0.00108239 5
			0.000972145 5
			0.00320953 5
			0.000922495 5
			0.00107909 5
			0.000997824 5
			0.000968645 5
			0.00109225 5
		};
		\addplot [black]
		table {%
			0.000186753 5.75
			0.000186753 6.25
			0.00042178825 6.25
			0.00042178825 5.75
			0.000186753 5.75
		};
		\addplot [black]
		table {%
			0.000186753 6
			2.45278e-05 6
		};
		\addplot [black]
		table {%
			0.00042178825 6
			0.000763704 6
		};
		\addplot [black]
		table {%
			2.45278e-05 5.875
			2.45278e-05 6.125
		};
		\addplot [black]
		table {%
			0.000763704 5.875
			0.000763704 6.125
		};
		\addplot [black, mark=o, mark size=2.5, mark options={solid,fill opacity=0}, only marks]
		table {%
			0.000867919 6
			0.000884352 6
			0.00147776 6
			0.00119664 6
			0.000781148 6
			0.00231595 6
		};
		\addplot [black]
		table {%
			0.00015172675 6.75
			0.00015172675 7.25
			0.000351395 7.25
			0.000351395 6.75
			0.00015172675 6.75
		};
		\addplot [black]
		table {%
			0.00015172675 7
			1.38682e-05 7
		};
		\addplot [black]
		table {%
			0.000351395 7
			0.000596088 7
		};
		\addplot [black]
		table {%
			1.38682e-05 6.875
			1.38682e-05 7.125
		};
		\addplot [black]
		table {%
			0.000596088 6.875
			0.000596088 7.125
		};
		\addplot [black, mark=o, mark size=2.5, mark options={solid,fill opacity=0}, only marks]
		table {%
			0.000913515 7
			0.000736531 7
			0.00167648 7
		};
		\addplot [black]
		table {%
			0.000122155 7.75
			0.000122155 8.25
			0.00025601425 8.25
			0.00025601425 7.75
			0.000122155 7.75
		};
		\addplot [black]
		table {%
			0.000122155 8
			8.21833e-06 8
		};
		\addplot [black]
		table {%
			0.00025601425 8
			0.000452621 8
		};
		\addplot [black]
		table {%
			8.21833e-06 7.875
			8.21833e-06 8.125
		};
		\addplot [black]
		table {%
			0.000452621 7.875
			0.000452621 8.125
		};
		\addplot [black, mark=o, mark size=2.5, mark options={solid,fill opacity=0}, only marks]
		table {%
			0.000505872 8
			0.000562396 8
			0.00110365 8
		};
		\addplot [black]
		table {%
			9.948105e-05 8.75
			9.948105e-05 9.25
			0.0001912865 9.25
			0.0001912865 8.75
			9.948105e-05 8.75
		};
		\addplot [black]
		table {%
			9.948105e-05 9
			4.91566e-06 9
		};
		\addplot [black]
		table {%
			0.0001912865 9
			0.000320171 9
		};
		\addplot [black]
		table {%
			4.91566e-06 8.875
			4.91566e-06 9.125
		};
		\addplot [black]
		table {%
			0.000320171 8.875
			0.000320171 9.125
		};
		\addplot [black, mark=o, mark size=2.5, mark options={solid,fill opacity=0}, only marks]
		table {%
			0.000343998 9
			0.000346845 9
			0.000332804 9
			0.000373113 9
			0.000355079 9
			0.000363022 9
			0.000358802 9
			0.000810181 9
			0.000343055 9
		};
		\addplot [black]
		table {%
			8.44456e-05 9.75
			8.44456e-05 10.25
			0.0001491825 10.25
			0.0001491825 9.75
			8.44456e-05 9.75
		};
		\addplot [black]
		table {%
			8.44456e-05 10
			2.94829e-06 10
		};
		\addplot [black]
		table {%
			0.0001491825 10
			0.000245027 10
		};
		\addplot [black]
		table {%
			2.94829e-06 9.875
			2.94829e-06 10.125
		};
		\addplot [black]
		table {%
			0.000245027 9.875
			0.000245027 10.125
		};
		\addplot [black, mark=o, mark size=2.5, mark options={solid,fill opacity=0}, only marks]
		table {%
			0.000258828 10
			0.000264714 10
			0.000257244 10
			0.000260662 10
			0.000268404 10
			0.000274537 10
			0.00025666 10
			0.00109557 10
			0.000265348 10
			0.000353181 10
			0.000280113 10
			0.000246707 10
			0.000292497 10
		};
		\addplot [black]
		table {%
			6.261285e-05 10.75
			6.261285e-05 11.25
			0.000115371 11.25
			0.000115371 10.75
			6.261285e-05 10.75
		};
		\addplot [black]
		table {%
			6.261285e-05 11
			1.79241e-06 11
		};
		\addplot [black]
		table {%
			0.000115371 11
			0.000192781 11
		};
		\addplot [black]
		table {%
			1.79241e-06 10.875
			1.79241e-06 11.125
		};
		\addplot [black]
		table {%
			0.000192781 10.875
			0.000192781 11.125
		};
		\addplot [black, mark=o, mark size=2.5, mark options={solid,fill opacity=0}, only marks]
		table {%
			0.000197613 11
			0.000210392 11
			0.000815648 11
			0.000276945 11
			0.000198486 11
		};
		\addplot [black]
		table {%
			4.541865e-05 11.75
			4.541865e-05 12.25
			8.83022e-05 12.25
			8.83022e-05 11.75
			4.541865e-05 11.75
		};
		\addplot [black]
		table {%
			4.541865e-05 12
			1.79492e-06 12
		};
		\addplot [black]
		table {%
			8.83022e-05 12
			0.000149927 12
		};
		\addplot [black]
		table {%
			1.79492e-06 11.875
			1.79492e-06 12.125
		};
		\addplot [black]
		table {%
			0.000149927 11.875
			0.000149927 12.125
		};
		\addplot [black, mark=o, mark size=2.5, mark options={solid,fill opacity=0}, only marks]
		table {%
			0.00015307 12
			0.000162991 12
			0.000161772 12
			0.000156332 12
			0.00062354 12
			0.000232858 12
			0.000157062 12
			0.000158747 12
			0.000154072 12
		};
		\addplot [orange]
		table {%
			0.0005049065 0.75
			0.0005049065 1.25
		};
		\addplot [orange]
		table {%
			0.00050997 1.75
			0.00050997 2.25
		};
		\addplot [orange]
		table {%
			0.000474967 2.75
			0.000474967 3.25
		};
		\addplot [orange]
		table {%
			0.0004035615 3.75
			0.0004035615 4.25
		};
		\addplot [orange]
		table {%
			0.000356605 4.75
			0.000356605 5.25
		};
		\addplot [orange]
		table {%
			0.0003000335 5.75
			0.0003000335 6.25
		};
		\addplot [orange]
		table {%
			0.0002440125 6.75
			0.0002440125 7.25
		};
		\addplot [orange]
		table {%
			0.000183152 7.75
			0.000183152 8.25
		};
		\addplot [orange]
		table {%
			0.000142186 8.75
			0.000142186 9.25
		};
		\addplot [orange]
		table {%
			0.000112765 9.75
			0.000112765 10.25
		};
		\addplot [orange]
		table {%
			8.74541e-05 10.75
			8.74541e-05 11.25
		};
		\addplot [orange]
		table {%
			6.25907e-05 11.75
			6.25907e-05 12.25
		};
		\end{axis}
		
		\end{tikzpicture}
		}
	\end{subfigure}
	\caption{Evolution of the distribution of the local total estimator $\eta_{\text{tot},T}$ over the mesh with uniform refinement (\emph{left}) and adaptive refinement (\emph{right}). The right panel shows that the interval $(\min_{T\in\mathcal{T}_h} \eta_{\text{tot},T}, \max_{T\in\mathcal{T}_h} \eta_{\text{tot},T})$ shrinks much faster in the adaptively refined case than in the uniformly refined one. The labels of the $y$-axis report the number of elements of the corresponding mesh.}
	\label{fig:Rboxplot}
\end{figure}

\begin{figure}[tb]
	\centering
	\begin{subfigure}{0.49\textwidth}
		\centering
		\includegraphics[width=0.9\textwidth]{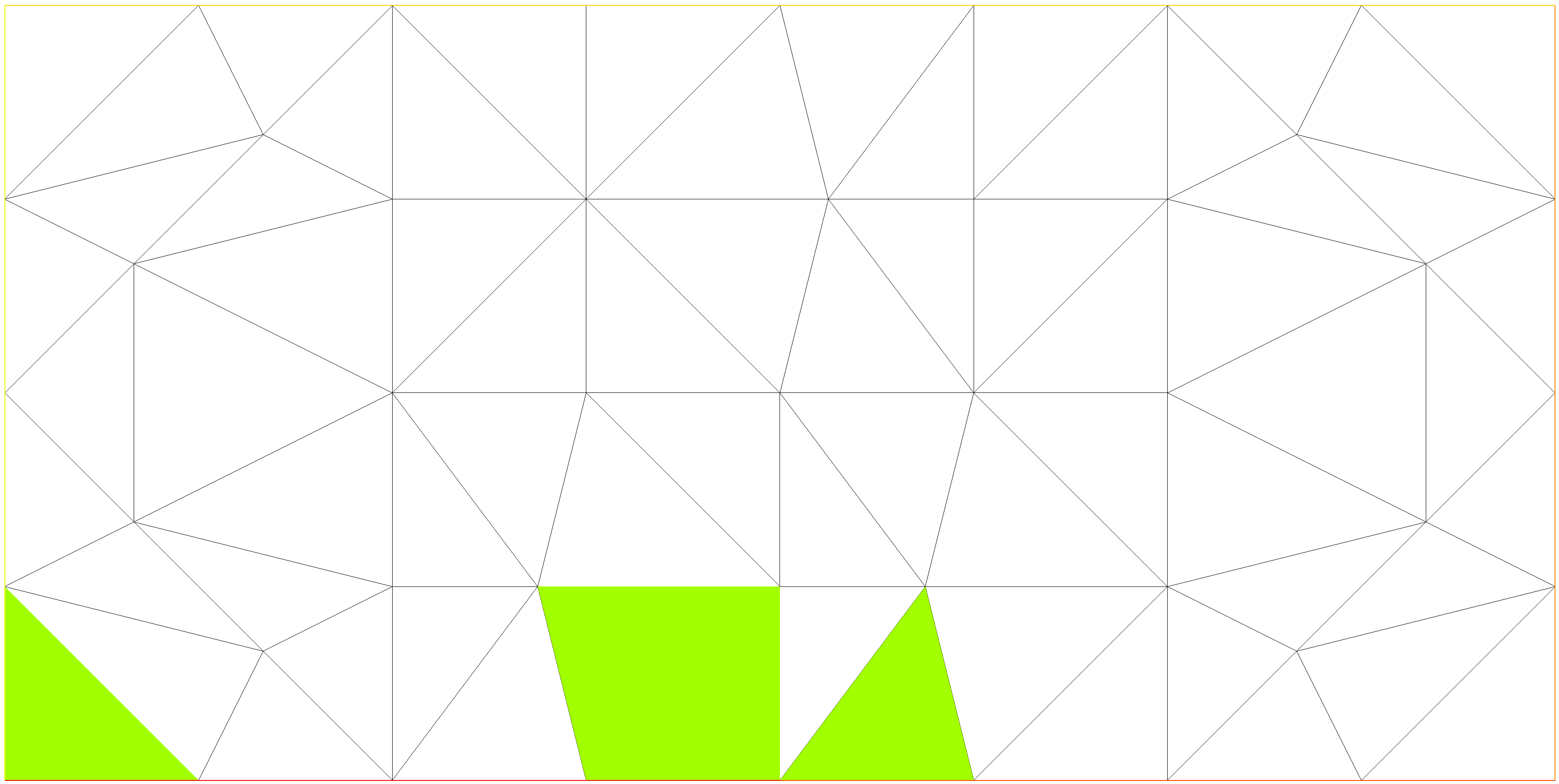}
	\end{subfigure}
	\hfill
	\begin{subfigure}{0.49\textwidth}
		\centering
		\includegraphics[width=0.9\textwidth]{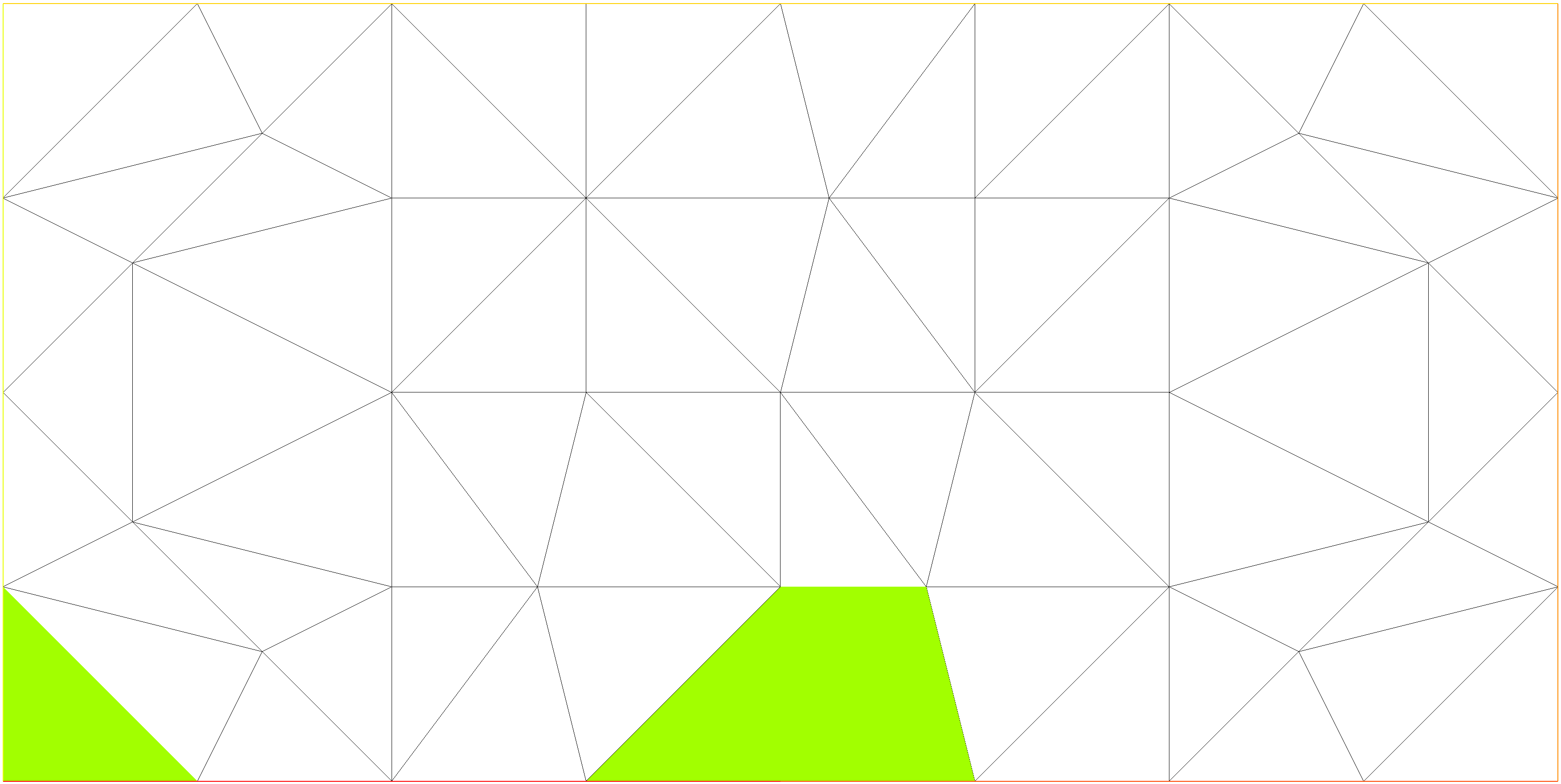}
	\end{subfigure}
	\begin{subfigure}{0.49\textwidth}
		\vspace{0.6cm}
		\centering
		\includegraphics[width=0.9\textwidth]{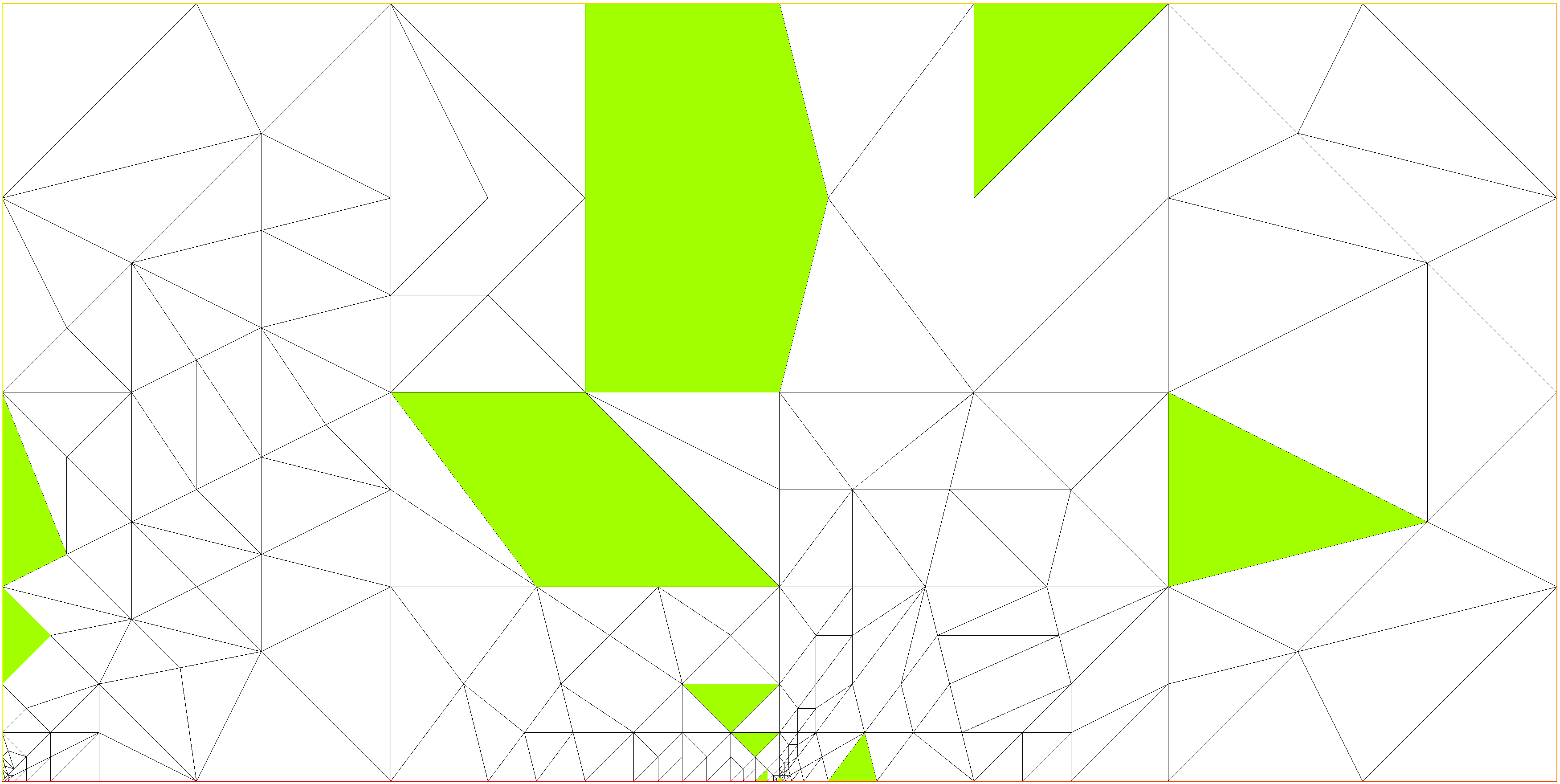}
	\end{subfigure}
	\hfill
	\begin{subfigure}{0.49\textwidth}
		\vspace{0.6cm}
		\centering
		\includegraphics[width=0.9\textwidth]{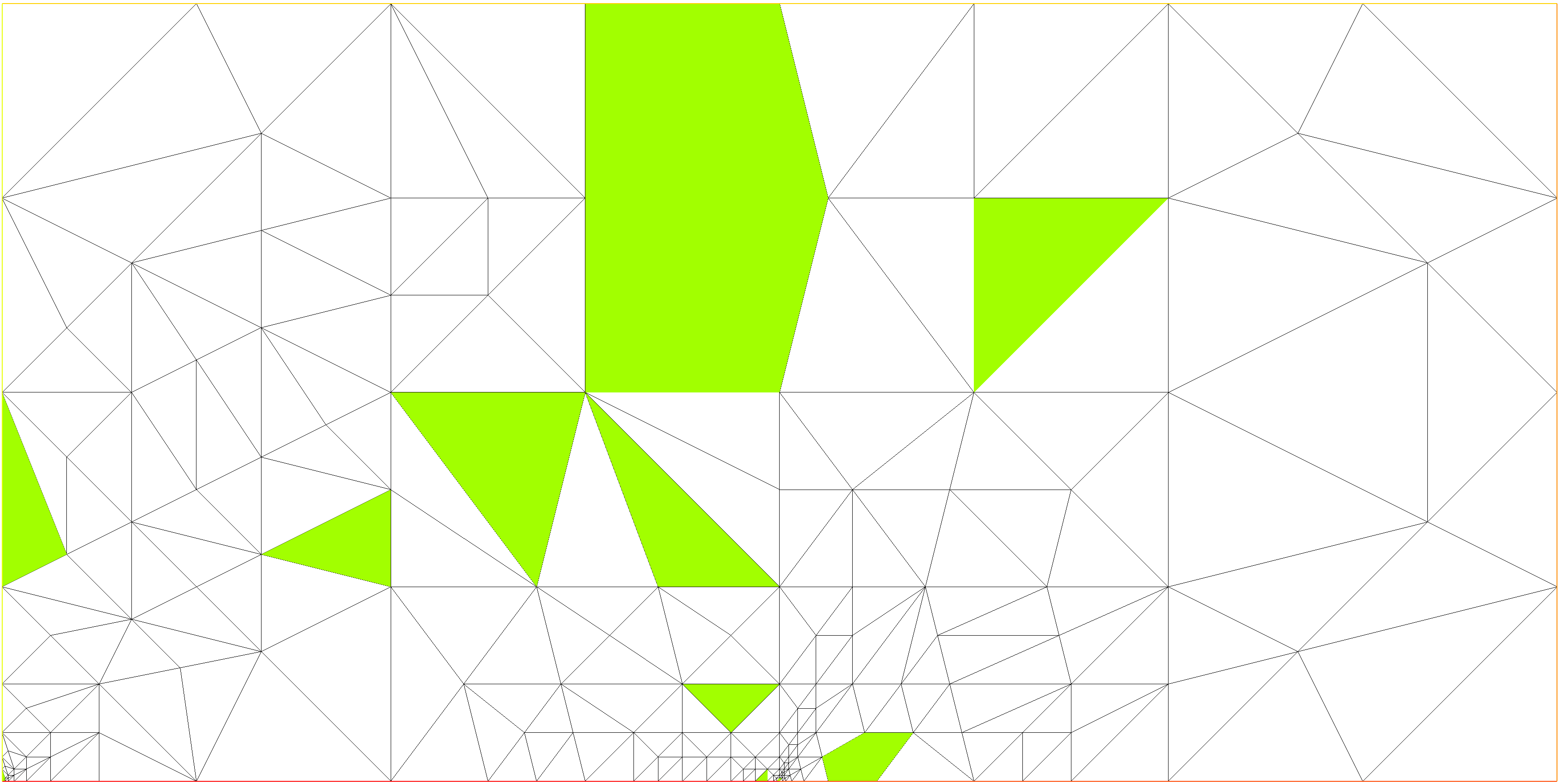}
	\end{subfigure}
	\begin{subfigure}{0.49\textwidth}
		\vspace{0.6cm}
		\centering
		\includegraphics[width=0.9\textwidth]{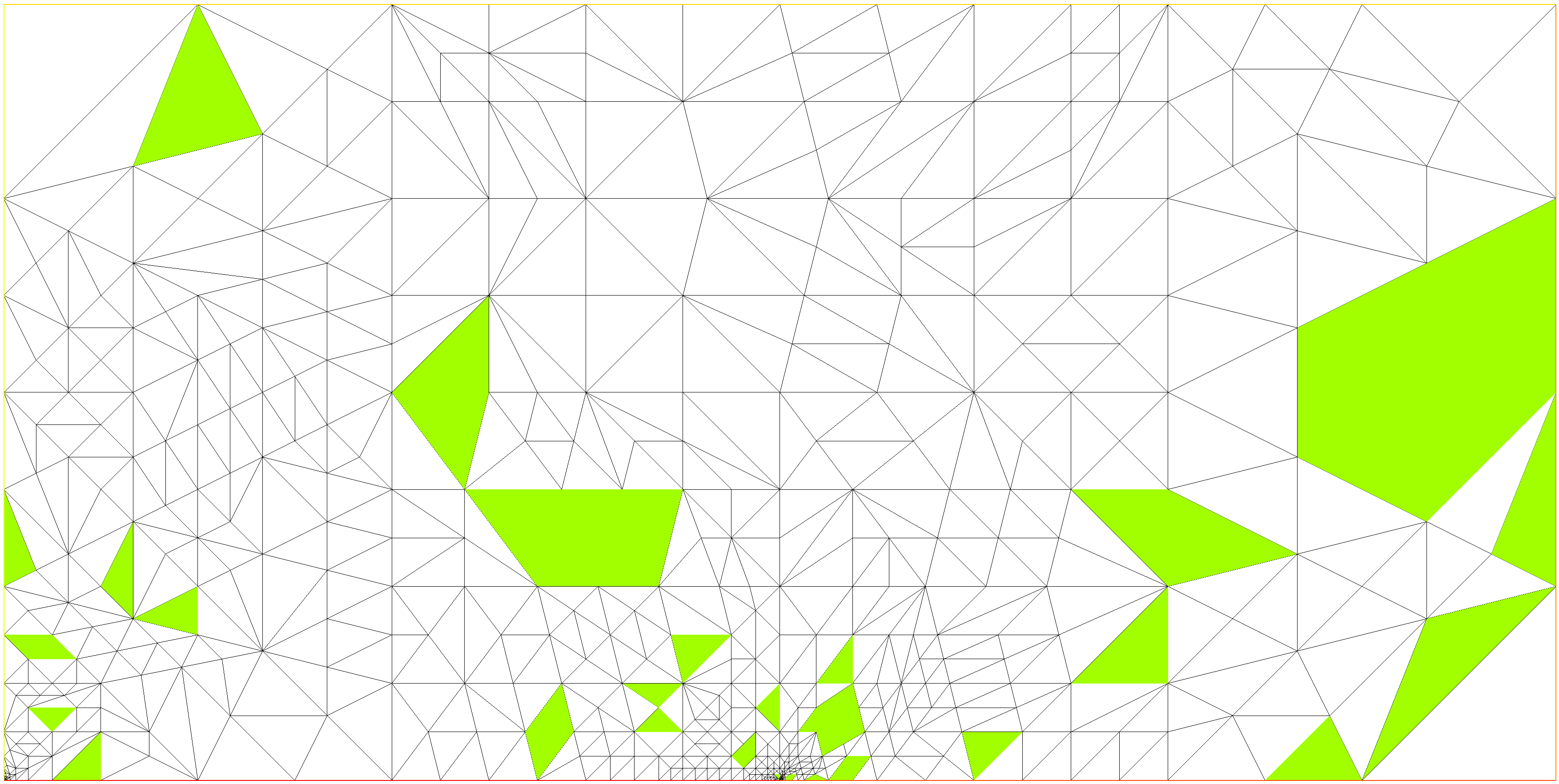}
	\end{subfigure}
	\hfill
	\begin{subfigure}{0.49\textwidth}
		\vspace{0.6cm}
		\centering
		\includegraphics[width=0.9\textwidth]{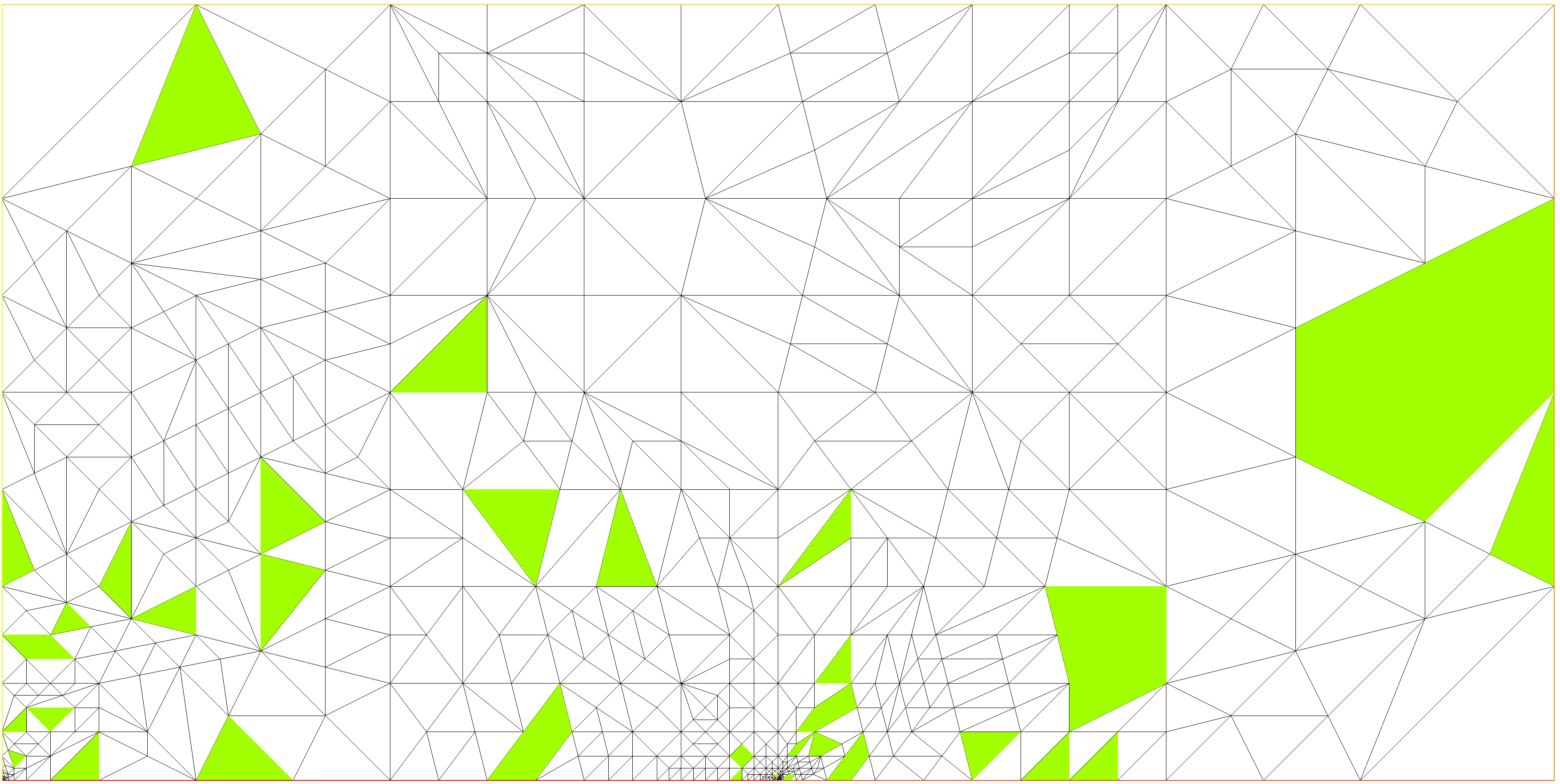}
	\end{subfigure}
	\caption{Triangles to refine following the distribution of $\energynorm{\bar{\VEC{u}}_h-\VEC{u}_h}$ (\emph{left}) and of $\eta_{\rm tot}$ (\emph{left}) for the initial mesh (\emph{top}), and adaptively refined mesh after 6 and 10 steps (\emph{middle} and \emph{bottom}, respectively).}
	\label{fig:Rrefinement mesh comparison}
\end{figure}

\begin{table}[tb]
	\centering
	\begin{tabular}{c|cccccccccccc}
		\toprule
		& Initial & 1\textsuperscript{st} & 2\textsuperscript{nd} & 3\textsuperscript{rd} & 4\textsuperscript{th} & 5\textsuperscript{th} & 6\textsuperscript{th} & 7\textsuperscript{th} & 8\textsuperscript{th} & 9\textsuperscript{th} & 10\textsuperscript{th} & 11\textsuperscript{th} \\
		\midrule
		$N_{\text{reg}}$ & 7 & 0 & 1 & 0 & 0 & 0 & 0 & 0 & 0 & 0 & 0 & 0 \\
		\hline
		$N_{\text{lin}}$ & 26 & 2 & 4 & 5 & 3 & 4 & 4 & 4 & 5 & 8 & 8 & 7 \\
		\bottomrule
	\end{tabular}
	\caption{Number of regularization iterations $N_{\text{reg}}$ and Newton iterations $N_{\text{lin}}$ at each refinement step of the Algorithm \ref{algorithm}.}
	\label{tab:Rnumber of regularization and linearization iterations}
\end{table}

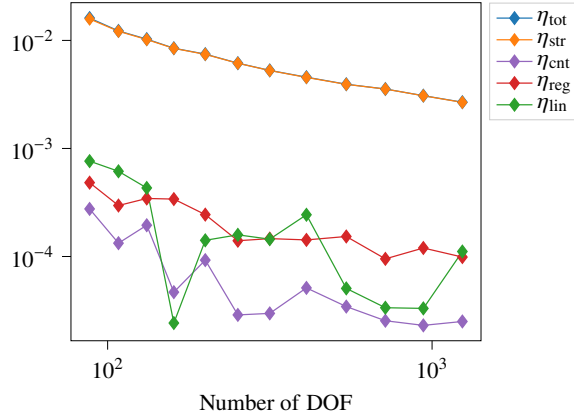
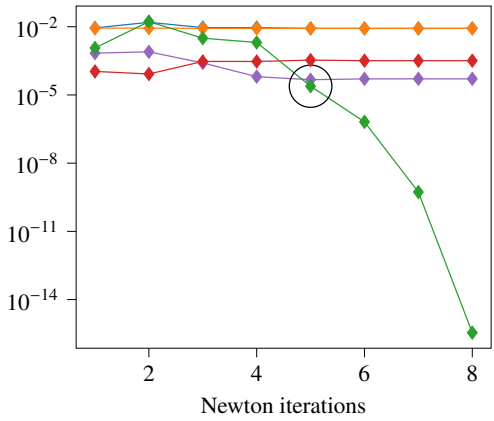
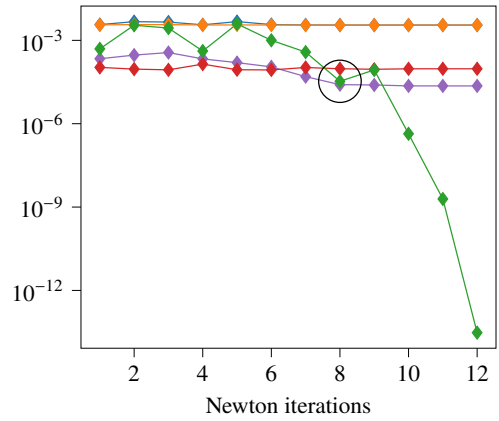
\begin{figure}[tb]
	\centering
	\begin{subfigure}{0.49\textwidth}
		\centering
		\resizebox{!}{0.75\textwidth}{%
		\begin{tikzpicture}
		
		\begin{axis}[
		legend cell align={left},
		legend style={
			fill opacity=1,
			draw opacity=1,
			text opacity=1,
			at={(1.02,1)},
			anchor=north west,
			draw=white!80!black
		},
		log basis x={10},
		log basis y={10},
		tick align=outside,
		tick pos=left,
		x grid style={white!69.0196078431373!black},
		xlabel={\(\displaystyle \mathrm{Number\ of\ DOF}\)},
		xmin=77.0966879864317, xmax=1415.36559935239,
		xmode=log,
		xtick style={color=black},
		xtick={100,1000},
		xticklabels={\(\displaystyle {10^{2}}\),\(\displaystyle {10^{3}}\)},
		y grid style={white!69.0196078431373!black},
		ymin=1.66086577220755e-05, ymax=0.0223475501817069,
		ymode=log,
		ytick style={color=black},
		ytick={1e-06,1e-05,0.0001,0.001,0.01,0.1,1},
		yticklabels={
			\(\displaystyle {10^{-6}}\),
			\(\displaystyle {10^{-5}}\),
			\(\displaystyle {10^{-4}}\),
			\(\displaystyle {10^{-3}}\),
			\(\displaystyle {10^{-2}}\),
			\(\displaystyle {10^{-1}}\),
			\(\displaystyle {10^{0}}\)
		}
		]
		\addplot [semithick, colorEstTot, mark=diamond*, mark size=3, mark options={solid}]
		table {%
			88 0.0161067309281877
			108 0.0122424923396395
			132 0.0102806921820822
			160 0.00846658210022344
			200 0.00746614304893199
			252 0.00615813412944203
			316 0.00527826737670964
			410 0.00456788050334957
			544 0.00393067123214782
			718 0.00354995917410312
			940 0.00308044974540385
			1240 0.00268018579431163
		};
		\addlegendentry{$\eta_{\mathrm{tot}}$}
		
		\addplot [semithick, colorEstFlux, mark=diamond*, mark size=2.8, mark options={solid}]
		table {%
			88 0.0157900364238697
			108 0.0121525010447791
			132 0.0101861689581428
			160 0.00843825050636362
			200 0.00743042563528934
			252 0.00614199070929103
			316 0.00526285980644922
			410 0.00454402069139958
			544 0.00391981736964773
			718 0.00354590924453527
			940 0.00307446026515113
			1240 0.00266979048184193
		};
		\addlegendentry{$\eta_{\mathrm{str}}$}
		
		\addplot [semithick, colorEstDisc, mark=diamond*, mark size=3, mark options={solid}]
		table {%
			88 0.000275692300669159
			108 0.000132863426281898
			132 0.000194755403588718
			160 4.67767481466012e-05
			200 9.25971012361318e-05
			252 2.88020589995251e-05
			316 2.97537996053192e-05
			410 5.12145396974364e-05
			544 3.44070361960387e-05
			718 2.54384956372241e-05
			940 2.30439568121995e-05
			1240 2.504680607899e-05
		};
		\addlegendentry{$\eta_{\mathrm{cnt}}$}
		
		\addplot [semithick, colorEstReg, mark=diamond*, mark size=3, mark options={solid}]
		table {%
			88 0.000482545056548585
			108 0.000296362258812555
			132 0.000344493357704873
			160 0.000340038196617796
			200 0.000244132682718152
			252 0.000139935060751482
			316 0.000146400703341916
			410 0.000142411051483284
			544 0.000153217269055861
			718 9.50812381498946e-05
			940 0.000119881629938028
			1240 9.88129392085373e-05
		};
		\addlegendentry{$\eta_{\mathrm{reg}}$}
		
		\addplot [semithick, colorEstLin, mark=diamond*, mark size=3, mark options={solid}]
		table {%
			88 0.000763622940608494
			108 0.000614748491638921
			132 0.000431007394490703
			160 2.42214414007696e-05
			200 0.000141201335286458
			252 0.00015872620330927
			316 0.000144684250847324
			410 0.000243399647487574
			544 5.0671978566673e-05
			718 3.35720741235714e-05
			940 3.30887490003539e-05
			1240 0.000111209723819221
		};
		\addlegendentry{$\eta_{\mathrm{lin}}$}
		\end{axis}
		
		\end{tikzpicture}
		}
		\subcaption{Global estimators with the stopping criteria.}
		\label{fig:Rtotal estimators stopping criteria}
	\end{subfigure}
	\vspace{6mm}
	\hfill \\
	\begin{subfigure}{0.49\textwidth}
		\centering
		\resizebox{!}{0.75\textwidth}{%
		\begin{tikzpicture}
		
		\definecolor{color0}{rgb}{0.12156862745098,0.466666666666667,0.705882352941177}
		\definecolor{color1}{rgb}{1,0.498039215686275,0.0549019607843137}
		\definecolor{color2}{rgb}{0.172549019607843,0.627450980392157,0.172549019607843}
		\definecolor{color3}{rgb}{0.83921568627451,0.152941176470588,0.156862745098039}
		\definecolor{color4}{rgb}{0.580392156862745,0.403921568627451,0.741176470588235}
		
		\begin{axis}[
		legend cell align={left},
		legend style={
			fill opacity=1,
			draw opacity=1,
			text opacity=1,
			at={(1.02,1)},
			anchor=north west,
			draw=white!80!black
		},
		log basis y={10},
		tick align=outside,
		tick pos=left,
		x grid style={white!69.0196078431373!black},
		xlabel={\(\displaystyle \mathrm{Newton\ iterations}\)},
		xmin=0.65, xmax=8.35,
		xtick style={color=black},
		xtick={2,4,6,8},
		xticklabels={
			\(\displaystyle {2}\),
			\(\displaystyle {4}\),
			\(\displaystyle {6}\),
			\(\displaystyle {8}\)
		},
		y grid style={white!69.0196078431373!black},
		ymin=7.24232117617029e-17, ymax=0.0811486968834287,
		ymode=log,
		ytick style={color=black},
		ytick={1e-20,1e-17,1e-14,1e-11,1e-08,1e-05,0.01,10,10000},
		yticklabels={
			\(\displaystyle {10^{-20}}\),
			\(\displaystyle {10^{-17}}\),
			\(\displaystyle {10^{-14}}\),
			\(\displaystyle {10^{-11}}\),
			\(\displaystyle {10^{-8}}\),
			\(\displaystyle {10^{-5}}\),
			\(\displaystyle {10^{-2}}\),
			\(\displaystyle {10^{1}}\),
			\(\displaystyle {10^{4}}\)
		}
		]
		\addplot [semithick, black, mark=o, mark size=10, mark options={solid,fill opacity=0}, only marks, forget plot]
		table {%
			5 2.42214414007696e-05
		};
		\addplot [semithick, colorEstTot, mark=diamond*, mark size=3, mark options={solid}]
		table {%
			1 0.00886895505468049
			2 0.0156520956263925
			3 0.00918347848293294
			4 0.00912320789481024
			5 0.00846658210022344
			6 0.00846216719568984
			7 0.00846205249166506
			8 0.00846205245822907
		};
		%\addlegendentry{$\eta_{\mathrm{tot}}$}
		
		\addplot [semithick, colorEstFlux, mark=diamond*, mark size=3, mark options={solid}]
		table {%
			1 0.00867249376309238
			2 0.00867903580485046
			3 0.00846826860226237
			4 0.00842685167924894
			5 0.00843825050636362
			6 0.00843822628047762
			7 0.00843822522753046
			8 0.00843822528563166
		};
		%\addlegendentry{$\eta_{\mathrm{flux}}$}
		
		\addplot [semithick, colorEstDisc, mark=diamond*, mark size=3, mark options={solid}]
		table {%
			1 0.000692691230992491
			2 0.000786087768928768
			3 0.000256339617055796
			4 6.36993337721891e-05
			5 4.67767481466012e-05
			6 5.07955571967095e-05
			7 5.10288315340424e-05
			8 5.10290225645116e-05
		};
		%\addlegendentry{$\eta_{\mathrm{disc}}$}
		
		\addplot [semithick, colorEstReg, mark=diamond*, mark size=3, mark options={solid}]
		table {%
			1 0.00010854719598036
			2 8.34136820799527e-05
			3 0.000297546195075444
			4 0.000296805217610341
			5 0.000340038196617796
			6 0.000321790437702206
			7 0.000321339331039565
			8 0.000321338973357954
		};
		%\addlegendentry{$\eta_{\mathrm{reg}}$}
		
		\addplot [semithick, colorEstLin, mark=diamond*, mark size=3, mark options={solid}]
		table {%
			1 0.00116759761268164
			2 0.0167964593199034
			3 0.00313332906888427
			4 0.00203411107440218
			5 2.42214414007696e-05
			6 6.53524439294408e-07
			7 5.35257445311509e-10
			8 3.4989810332292e-16
		};
		%\addlegendentry{$\eta_{\mathrm{lin}}$}
		\end{axis}
		
		\end{tikzpicture}
		}
		\subcaption{3rd adaptively refined mesh.}
		\label{fig:Rtotal estimators 5th}
	\end{subfigure}
	\begin{subfigure}{0.49\textwidth}
		\centering
		\resizebox{!}{0.75\textwidth}{%
		\begin{tikzpicture}
		
		\definecolor{color0}{rgb}{0.12156862745098,0.466666666666667,0.705882352941177}
		\definecolor{color1}{rgb}{1,0.498039215686275,0.0549019607843137}
		\definecolor{color2}{rgb}{0.172549019607843,0.627450980392157,0.172549019607843}
		\definecolor{color3}{rgb}{0.83921568627451,0.152941176470588,0.156862745098039}
		\definecolor{color4}{rgb}{0.580392156862745,0.403921568627451,0.741176470588235}
		
		\begin{axis}[
		legend cell align={left},
		legend style={
			fill opacity=1,
			draw opacity=1,
			text opacity=1,
			at={(1.02,1)},
			anchor=north west,
			draw=white!80!black
		},
		log basis y={10},
		tick align=outside,
		tick pos=left,
		x grid style={white!69.0196078431373!black},
		xlabel={\(\displaystyle \mathrm{Newton\ iterations}\)},
		xmin=0.45, xmax=12.55,
		xtick style={color=black},
		xtick={2,4,6,8,10,12},
		xticklabels={
			\(\displaystyle {2}\),
			\(\displaystyle {4}\),
			\(\displaystyle {6}\),
			\(\displaystyle {8}\),
			\(\displaystyle {10}\),
			\(\displaystyle {12}\)
		},
		y grid style={white!69.0196078431373!black},
		ymin=8.31109973258627e-15, ymax=0.0171386813527961,
		ymode=log,
		ytick style={color=black},
		ytick={1e-18,1e-15,1e-12,1e-09,1e-06,0.001,1,1000},
		yticklabels={
			\(\displaystyle {10^{-18}}\),
			\(\displaystyle {10^{-15}}\),
			\(\displaystyle {10^{-12}}\),
			\(\displaystyle {10^{-9}}\),
			\(\displaystyle {10^{-6}}\),
			\(\displaystyle {10^{-3}}\),
			\(\displaystyle {10^{0}}\),
			\(\displaystyle {10^{3}}\)
		}
		]
		\addplot [semithick, black, mark=o, mark size=10, mark options={solid,fill opacity=0}, only marks, forget plot]
		table {%
			8 3.35720741235714e-05
		};
		\addplot [semithick, colorEstTot, mark=diamond*, mark size=3, mark options={solid}]
		table {%
			1 0.00368243218591813
			2 0.00470888730300549
			3 0.00455803478316666
			4 0.00361821542548295
			5 0.00472319238563059
			6 0.00370203019785874
			7 0.00357207616499098
			8 0.00354995917410312
			9 0.00354983989659642
			10 0.00354811474306052
			11 0.00354810109506967
			12 0.00354811045284062
		};
		%\addlegendentry{$\eta_{\mathrm{tot}}$}
		
		\addplot [semithick, colorEstFlux, mark=diamond*, mark size=3, mark options={solid}]
		table {%
			1 0.00360006910982575
			2 0.00369303017727755
			3 0.00369154679440686
			4 0.00357007888600865
			5 0.0035880518067944
			6 0.00356533410347973
			7 0.00354731914400419
			8 0.00354590924453527
			9 0.00354554460712866
			10 0.00354549467687667
			11 0.00354548834490896
			12 0.00354549774259006
		};
		%\addlegendentry{$\eta_{\mathrm{flux}}$}
		
		\addplot [semithick, colorEstDisc, mark=diamond*, mark size=3, mark options={solid}]
		table {%
			1 0.000217615522168555
			2 0.000294137029232074
			3 0.000363764099082753
			4 0.000215311572431499
			5 0.00015780221765536
			6 0.000110992090298589
			7 4.94525778302834e-05
			8 2.54384956372241e-05
			9 2.47456174592694e-05
			10 2.29622674211289e-05
			11 2.28992550070809e-05
			12 2.28990133341044e-05
		};
		%\addlegendentry{$\eta_{\mathrm{disc}}$}
		
		\addplot [semithick, colorEstReg, mark=diamond*, mark size=3, mark options={solid}]
		table {%
			1 0.000106508230105686
			2 9.2211601816027e-05
			3 8.67811346534672e-05
			4 0.000139460304052869
			5 8.74864736866408e-05
			6 8.62493302907503e-05
			7 0.000105371617807531
			8 9.50812381498946e-05
			9 9.08496789255673e-05
			10 9.45673077754595e-05
			11 9.47328096490061e-05
			12 9.47334606220761e-05
		};
		%\addlegendentry{$\eta_{\mathrm{reg}}$}
		
		\addplot [semithick, colorEstLin, mark=diamond*, mark size=3, mark options={solid}]
		table {%
			1 0.000494309349934653
			2 0.00350213733671345
			3 0.00276841062164648
			4 0.000412684736935089
			5 0.00391767981758867
			6 0.0010043353066113
			7 0.000377547887346969
			8 3.35720741235714e-05
			9 8.61866979657707e-05
			10 4.33611673669147e-07
			11 1.95750678165015e-09
			12 3.01578420649252e-14
		};
		%\addlegendentry{$\eta_{\mathrm{lin}}$}
		\end{axis}
		
		\end{tikzpicture}
		}
		\subcaption{9th adaptively refined mesh.}
		\label{fig:Rtotal estimators 9th}
	\end{subfigure}
	\caption{Global estimators $\eta_{\text{tot}}$, $\eta_{\text{str}}$, $\eta_{\text{lin}}$, $\eta_{\text{reg}}$ and $\eta_{\text{cnt}}$ as function of the number of degrees of freedom using the global stopping criteria of Algorithm \ref{algorithm} (\emph{top}), and as function of Newton iterations for the 3rd and 9th adaptively refined mesh (\emph{bottom-left} and \emph{bottom-right}, respectively).}%; comparison of the global stress, discretization and regularization estimator without and with the stopping criteria (\emph{bottom}).}
	\label{fig:Rtotal estimators}
\end{figure}

We recall that the measure of the error is the dual norm $\dualnormresidual{\VEC{u}_h}$ defined by \eqref{eq - definition dual norm}, which is not computable (it can be, however, approximated through an elliptic lifting). As a consequence, recalling Theorems \ref{theorem - upper bound of the energy norm} and \ref{theorem - lower bound of the energy norm}, we compare the global total estimator $\eta_{\text{tot}}$ with %the energy norm $\energynorm{\bar{\VEC{u}}_h-\VEC{u}_h}$ and with the quantity
the following quantities:
\[
\mathcal{L}(\VEC{u}_h) \coloneqq \mu^{\nicefrac{1}{2}} \energynorm{\bar{\VEC{u}}_h-\VEC{u}_h}
\]
and
\[
\mathcal{U}(\VEC{u}_h) \coloneqq (d \lambda + 4\mu)^{\nicefrac{1}{2}} \energynorm{\bar{\VEC{u}}_h-\VEC{u}_h} + \left(\sum_{F\in\mathcal{F}_h^C} h_F \left\lVert \sigma^n(\bar{\VEC{u}}_h)-\left[P_{1,\gamma}^n(\VEC{u}_h)\right]_{\mathbb{R}^-}\right\rVert_F^2\right)^{\nicefrac{1}{2}}.
\]
In particular, the latter incorporates an additional error component on the contact interface.
Furthermore, we define the two following effectivity indices:
\[
I_{\text{eff,low}} \coloneqq \frac{\eta_{\text{tot}}}{\mathcal{L}(\VEC{u}_h)} = \frac{\eta_{\rm tot}}{\mu^{\nicefrac{1}{2}} \energynorm{\bar{\VEC{u}}_h-\VEC{u}_h}}
\qquad \text{and} \qquad
I_{\text{eff,up}} \coloneqq \frac{\eta_{\text{tot}}}{\mathcal{U}(\VEC{u}_h)}.
\]
The results are illustrated in Figure \ref{fig:Rcomparison energy+upper} and \ref{fig:Rcomparison effectivity}. The total estimator always remains between the energy norm rescaled by a Lamé parameter %$\energynorm{\bar{\VEC{u}}_h-\VEC{u}_h}$ 
$\mathcal{L}(\VEC{u}_h)$
and the upper bound for the dual residual norm $\mathcal{U}(\VEC{u}_h)$, i.e., $I_{\text{eff,low}} > 1$ and $I_{\text{eff,up}} < 1$.
Figure \ref{fig:Rboxplot} shows the distribution of the local total estimator $\eta_{\text{tot},T}$ at each mesh refinement step in both the uniform and adaptive frameworks. %In particular, the uniform refinement is performed dividing all triangle of the mesh into four sub-triangles.
Here, we use boxplots to see where values concentrate. With the adaptive approach, all the local estimators $\{\eta_{\text{tot},T}\}_{T\in \mathcal{T}_h}$ are contained in an interval that becomes smaller and smaller at each refinement iteration, and the decrease of the maximum value is significantly faster  than in the uniformly refined computation. Indeed, in the latter there is always a value which is much bigger than the others even if the number of degrees of freedom and the number of elements are high (in the last case, $\lvert{V}_h\rvert = 1156$ and $\lvert\mathcal{T}_h\rvert = 1088$), showing that the error concentrates in specific areas.
Figure \ref{fig:Rrefinement mesh comparison} compares the selection of triangles to refine (highlighted in green) using the distribution of the energy norm $\left\lVert \bar{\VEC{u}}_h-\VEC{u}_h \right\rVert_{\rm en, T}$ (\emph{left}) and the total estimator $\eta_{\rm tot, T}$ (\emph{right}) for the initial mesh and the adaptively refined mesh after 6 and 10 steps. The sets of selected triangles are concentrated in the same zones.

Finally, we apply the fully adaptive Algorithm \ref{algorithm} with $\gamma_{\text{reg}} = 0.04$ and $\gamma_{\text{lin}} = 0.08$. The initial regularization parameter $\delta$ is taken equal to the Young modulus $E$ and, at each step in which the global stopping criterion %\eqref{global stopping criterion1} 
shown in Line \ref{alg:global stopping criterion reg} of Algorithm \ref{algorithm}
is not satisfied, we divide it by 2. Table \ref{tab:Rnumber of regularization and linearization iterations} contains the number of regularization and Newton iterations, denoted by $N_{\text{reg}}$ and $N_{\text{lin}}$ respectively, and Figure \ref{fig:Rtotal estimators stopping criteria} displays the curves of the different global estimators for 11 adaptive refinement steps as functions of the degrees of freedom. The same estimators are shown in Figure \ref{fig:Rtotal estimators 5th} and \ref{fig:Rtotal estimators 9th} as functions of the Newton iterations for the 3rd and 9th adaptively refined meshes. A circle underlines the step (5th and 8th, respectively) at which the global stopping criterion %\eqref{global stopping criterion2} 
of Line \ref{alg:global stopping criterion lin} 
is reached. At this step, %the linearization estimator is bigger than regularization estimator (which satisfies the global stopping criterion %\eqref{global stopping criterion1}),
%of Line \ref{alg:global stopping criterion reg}),
%but smaller than the other ones that have already stabilized.
the regularization estimator satisfies the global stopping criterion of Line \ref{alg:global stopping criterion reg}, and the other ones have already stabilized.

\section*{Acknowledgements}

This project has received funding from the European Union’s Horizon 2020 research and innovation program under grant agreement No 847593.

\begin{small}
  
\bibliography{refs}
%% \nocite{*}
\bibliographystyle{siam}

\end{small}

\end{document}